\documentclass[12pt,reqno]{amsart}
\usepackage{graphicx,indentfirst}
\usepackage{amsmath,amssymb,mathrsfs}
\usepackage{amsthm,amscd}
\usepackage{verbatim}
\usepackage{appendix}
\usepackage{enumitem,titletoc}
\usepackage{mathtools}
\usepackage[utf8]{inputenc}
\usepackage{imakeidx}
\makeindex[columns=2, title=Alphabetical Index]
\usepackage{fancyhdr}
\usepackage{amsfonts,color}
\usepackage[all]{xy}
\usepackage{tikz-cd}
\usepackage{syntonly}
\usepackage{fancyhdr}
\usepackage{array}
\usepackage[left=2.3cm,right=2.3cm,bottom=3cm,top=3cm]{geometry}
\usepackage{microtype}

\usepackage{tikz}
\usepackage{extarrows}
\usepackage{hyperref}
\usepackage{setspace}
\allowdisplaybreaks[2]

\def\XXint#1#2#3{{\setbox0=\hbox{$#1{#2#3}{\int}$ }
		\vcenter{\hbox{$#2#3$ }}\kern-.6\wd0}}

\newtheorem{thm}{Theorem}[section]
\newtheorem{prop}[thm]{Proposition}

\newtheorem{defn}[thm]{Definition}
\newtheorem{lem}[thm]{Lemma}
\newtheorem{cor}[thm]{Corollary}
\newtheorem{conj}[thm]{Conjecture}
\newtheorem{rk}{Remark}

\allowdisplaybreaks
\newcommand{\del}{\partial}
\newcommand{\dbar}{\bar\partial}

\newcommand{\ddbar}{\sqrt{-1}\partial\bar\partial}

\newcommand{\R}{\mathbb R}
\newcommand{\rank}{{\rm rk}}
\newcommand{\htheta}{\hat{\theta}}
\newcommand{\cX}{\mathcal{X}}
\newcommand{\cE}{\mathcal{E}}
\newcommand{\cS}{\mathcal{S}}
\newcommand{\cQ}{\mathcal{Q}}
\newcommand{\cA}{\mathcal{A}}
\newcommand{\cF}{\mathcal{F}}
\newcommand{\cL}{\mathcal{L}}
\newcommand{\Tr}{{\rm Tr}}
\newcommand{\Herm}{{\rm Herm}}

\makeindex

\DeclareMathOperator{\osc}{osc}
\DeclareMathOperator{\vol}{Vol}

\definecolor{citationblue}{HTML}{008000}
\hypersetup{
	colorlinks=true,
	citecolor=citationblue,
	filecolor=green,
	linkcolor=blue,
	urlcolor=black
}

\numberwithin{equation}{section}

\author{Tristan C. Collins}
\email{\href{mailto:tristanc@math.toronto.edu}{tristanc@math.toronto.edu}}

\author{Yukai Zhang}
\email{yukai.zhang@mail.utoronto.ca}
\address{Department of Mathematics, University of Toronto, 40 St. George Street, Toronto, ON, Canada}

\title[The Deformed Vortex Equations]{The Deformed Vortex Equations and Equivariant  Stability Conditions}

\begin{document}
\begin{abstract}
 We study the higher rank deformed Hermitian-Yang-Mills (dHYM) equations for $SU(2)$-equivariant holomorphic vector bundles over $\mathcal{X}= X\times \mathbb{P}^1$ for $X$ a compact Riemann surface.  For a class of vector bundles of ``vortex type",  we establish the equivalence between existence of solutions to higher rank dHYM equations and an appropriate notion of algebro-geometric stability, called $Z$-stability.  We show that $Z$-stability for vector bundles of ``vortex type" is implied by, and conjecturally equivalent to, Bridgeland stability in $D^{b}{\rm Coh}^{SU(2)}(\mathcal{X})$. 
\end{abstract}
\maketitle
\tableofcontents
\section{Introduction}

A fundamental theme in K\"ahler geometry is the correspondence between the existence of canonical geometric structures, as governed by non-linear partial differential equations, and notions of algebro-geometric stability. The foundations of this correspondence were established by the classical Donaldson–Uhlenbeck–Yau (DUY) theorem \cite{Do83, Do85, Do87, UY}, equating the existence of a Hermitian–Yang–Mills connection on a vector bundle to Mumford–Takemoto slope stability.  More recently,  Chen–Donaldson–Sun \cite{CDS1,CDS2,CDS3} resolved the Yau–Tian–Donaldson conjecture that a Fano manifold admits a K\"ahler–Einstein metric if and only if it is K-stable.

In recent years, attention has increasingly turned toward fully nonlinear differential equations on vector bundles.  These considerations are motivated by geometry, mathematical physics, and connections with algebraic geometry. Prominent examples of such equations are the deformed Hermitian–Yang–Mills (dHYM) equation \cite{MMMS, LYZ} and more generally (polynomially) $Z$-critical connections \cite{JBM, DMS, DNST}, Demailly’s vector bundle Monge–Amp\`ere approach to the Griffiths' conjecture \cite{DemGr}, and the gauge-theoretic sector of the Hull–Strominger system describing vacuum configurations of the heterotic string \cite{Hull, Strominger, PhongSurvey}. From the perspective of analysis these equations present severe analytic challenges. While the scalar theory of fully nonlinear elliptic equations is deeply developed, there exists no general analytic framework for treating fully nonlinear elliptic systems.  Perhaps more troublingly, these nonlinear systems are not even known to be elliptic in general \cite{AshPin, ChenGhosh}. As a result, while these equations have generated considerable interest, very little is known about existence beyond the perturbative regime \cite{DMS, Leung97, LaraSaEarp}, homogeneous settings where the construction of solutions can be reduced to an algebraic problem \cite{COSD, Correa}, or for certain rank $2$ bundles where the equation can be reduced to a pair of elliptic scalar equations \cite{Pingali, Takahashi, Ghosh}.

These geometric developments parallel developments in algebraic geometry, also motivated by mathematical physics, which have led to new, generalized notions of algebro-geometric stability, namely Bridgeland stability, and it is expected that the existence of solutions to these higher rank nonlinear systems is equivalent to generalized notions of stability (see e.g. \cite{CY, CYarX, DMS}).  These notions of stability represent a dramatic departure from the classical notions of Mumford-Takemoto or Gieseker stability as holomorphic vector bundles can be destabilized by torsion sheaves, a fact which is reflected analytically in the absence of a priori curvature bounds for nonlinear systems.

In this paper we establish, for certain K\"ahler surfaces, a correspondence between existence of solutions to the higher rank deformed Hermitian-Yang-Mills equation and an algebro-geometric notion of stability closely related to Bridgeland stability. In fact, as we will show, our notion of stability is implied by, and conjecturally equivalent to, Bridgeland stability.  To our knowledge, this is the first non-perturbative existence–stability correspondence for the higher-rank deformed Hermitian–Yang–Mills equation in a setting where the equation remains a genuinely matrix-valued, nonlinear system.  Our result therefore provides the first evidence for a correspondence between Bridgeland stability and existence of solutions to the higher rank dHYM equation.  Moreover, our results provide concrete evidence that the algebraic structures encoded by Bridgeland stability play an important role in the  analytic study of fully nonlinear PDEs on holomorphic vector bundles.

To introduce our setting, let $(\cX,\omega)$ be a compact K\"ahler manifold, and let $\cE\rightarrow \cX$ be a holomorphic vector bundle.  The deformed Hermitian-Yang-Mills (dHYM) equation seeks a Hermitian metric $h$ on $\cE$ such that the curvature $F_{h}$ of the Chern connection satisfies the system of equations
\begin{equation}\label{eq: dHYMIntroduction}
{\rm Im}\left(e^{-i\htheta}\left(\omega\otimes I_{\cE} - F_{h}\right)^n\right)=0
\end{equation}
\begin{equation}\label{eq: dHYMPosIntroduction}
{\rm Re}\left(e^{-i\htheta}(\omega\otimes I_{\cE} - F_{h})^n\right)>0
\end{equation}
where the real and imaginary parts are defined using $h$, and $e^{-i\htheta}$ is a (possibly ill-defined) constant determined by ${\rm ch}(\cE)$ and $\omega$.  We refer to the constraint~\eqref{eq: dHYMPosIntroduction} as {\em the calibration constraint}, for reasons we will explain in Section~\ref{sec:background}.  

For a coherent sheaf $\cF\rightarrow \cX$ define the central charge by
\begin{equation}\label{eq: centralChargeIntro}
Z_{\cX}(\cF)= -\int_{X}e^{-\frac{i}{2\pi}\omega}ch(\cF) \in \mathbb{C}.
\end{equation}
Suppose that $\dim_{\mathbb{C}}\cX=2$.  Assuming that $Z_{\cX}(\cF) \in \mathbb{C}^*$ we define
\[
\Theta(\cF) = {\rm Arg}(Z_{\cX}(\cF))
\]
where ${\rm Arg}$ is the principal value of the argument of $Z_{\cX}(\cF)$.

\begin{defn}\label{defn: Zstableintro}
    Let $\cE\rightarrow \cX$ be a vector bundle, and assume that $\Theta(\cE) \in (0,\pi)$. We say that $\cE$ is $Z$-stable in ${\rm Coh}(\cX)$ if:
    \begin{itemize}
        \item[(i)] for any non-zero coherent sheaf $\cQ$ such that $\cE \twoheadrightarrow \cQ$, we have ${\rm Im}(Z_{\cX}(\cQ))>0$, and
        \item[(ii)] for any exact sequence of coherent sheaves
        \[
        0\rightarrow \cS \rightarrow \cE\rightarrow \cQ \rightarrow 0
        \]
        with $\Theta(\cS), \Theta(\cQ)\in (0,\pi]$ we have
        \[
        \Theta(\cS) < \Theta(\cE) \quad \text{ or equivalently } \quad \Theta(\cE) < \Theta(\cQ).
        \]
    \end{itemize}
\end{defn}

Consider a K\"ahler surface of the form $\cX= X\times \mathbb{P}^1$ with $X$ a compact Riemann surface and let $p_i$ denote projection to the $i$-th factor.  $SU(2)$ acts on $\cX$ via the standard action on $\mathbb{P}^1$.  Equip $\cX$ with the K\"ahler form $\omega_{\cX}= p_1^*\omega_{X}\oplus \frac{\sigma}{2\pi}p_2^*\omega_{FS}$ where $\omega_{X}$ is a K\"ahler form on $X$.  Let $E_1\rightarrow X$ be a holomorphic vector bundle and $E_2\rightarrow X$ a holomorphic line bundle.  We consider holomorphic vector bundles $\cE \rightarrow \cX$ arising as non-split, $SU(2)$-equivariant extensions
\[
0 \rightarrow p_1^{*}E_1 \rightarrow \mathcal{E} \rightarrow p_1^*E_2\otimes p_2^*\mathcal{O}_{\mathbb{P}^1}(2) \rightarrow 0.
\]
We say that such vector bundles are of {\em vortex type}.   Our main result is the following:

\begin{thm}\label{thm: mainTheorem}
    Let $(X,\omega_{X})$ be a compact Riemann surface, and $\cX= X\times \mathbb{P}^1$ equipped with the K\"ahler form $\omega_{\cX}= \omega_{X}\oplus \frac{\sigma}{2\pi}\omega_{FS}$.  Let $\cE\rightarrow \cX$ be an irreducible $SU(2)$-equivariant holomorphic vector bundle of vortex type with ${\rm Im}(Z_{\cX}(\cE))>0$ and $\Theta(\cE) \in (0, \arctan(\frac{\sigma}{4\pi})]\cup [\frac{\pi}{2}, \pi]$.  Then $\cE$ admits a solution of the deformed Hermitian-Yang-Mills system~\eqref{eq: dHYMIntroduction}+~\eqref{eq: dHYMPosIntroduction} if and only if $\cE$ is $SU(2)$-equivariantly $Z$-stable in ${\rm Coh}^{SU(2)}(\cX)$.
\end{thm}

By $SU(2)$-equivariantly $Z$-stable we mean that, in Definition~\ref{defn: Zstableintro} one only needs to check surjections $\cE\twoheadrightarrow\cQ$ in the abelian category ${\rm Coh}^{SU(2)}(\cX)$ of $SU(2)$-equivariant coherent sheaves. In fact we achieve more than what is stated in Theorem~\ref{thm: mainTheorem}: we give a complete classification of all solutions of the dHYM equation with ${\rm Im}(Z_{\cX}(\cE)) >0$ and ${\rm Re}(Z_{\cX}(\cE))\leq 0$; see Corollary~\ref{cor: classifySolRealNeg}.  This class turns out to be quite rigid, consisting only of the rank $2$ vector bundles generated by taking $E_1 \simeq E_2$.  Furthermore, we show that there are no solutions of the equation at all with $\tan(\htheta) \in [-\frac{\sigma}{2\pi},0)$. 

We prove that $Z$-stability for an $SU(2)$-equivariant holomorphic vector bundle of vortex type is implied by Bridgeland stability for the stability condition on $D^{b}{\rm Coh}^{SU(2)}(\cX)$ constructed by Arcara-Bertram \cite{AB}.  In particular, we obtain the following corollary.

\begin{cor}
    In the setting of Theorem~\ref{thm: mainTheorem}, if $\cE \in D^{b}{\rm Coh}^{SU(2)}(\cX)$ is a Bridgeland stable $SU(2)$-equivariant holomorphic vector bundle of vortex type, then $\cE$ admits an equivariant solution of the deformed Hermitian-Yang-Mills equation~\eqref{eq: dHYMIntroduction}+~\eqref{eq: dHYMPosIntroduction}.
\end{cor}

At the end of the introduction, we give an outline of the paper noting the particular results containing more precise versions of the results subsumed into Theorem~\ref{thm: mainTheorem}.

We emphasize that in Theorem~\ref{thm: mainTheorem} torsion sheaves can appear as bona fide destabilizing quotients, with the most important class of such objects given by holomorphic vector bundles supported on analytic subvarieties.  This is a significant departure from the classical theory of Mumford-Takemoto stability, where it suffices to consider torsion-free quotients which can be treated as vector bundles over a Zariski open set.

The strategy of the proof is to use the $SU(2)$-equivariance to reduce the system~\eqref{eq: dHYMIntroduction} on $\cX$ to a system of equations for metrics on a holomorphic triple $(E_1,E_2,\Phi)\rightarrow X$, where $E_1, E_2$ are holomorphic vector bundles, $\rank(E_2)=1$ and $\Phi \in H^0(X,E_1\otimes E_2^{\vee})$. This system is a highly nonlinear extension of the classical vortex equations \cite{Bradlow, GPVortex, BG-P}. We prove a priori estimates for these equations using a variety of techniques from nonlinear PDE, and compactness/contradiction arguments inspired by the work of Uhlenbeck-Yau \cite{UY}.

We now outline the general framework of the paper. 
\begin{itemize}
\item Section~\ref{sec:background} is an extended introduction. We discuss how the dHYM system~\eqref{eq: dHYMIntroduction} +~\eqref{eq: dHYMPosIntroduction} arises from mirror symmetry via the consideration of special Lagrangians.  We discuss in more detail previous work on existence of solutions for holomorphic line bundles and the relation to Bridgeland stability conditions.  We then outline what is known for the case of the dHYM equation for general vector bundles. Finally, we discuss the connection of the dHYM equation to nonlinear extensions of the vortex equations, which are of independent interest.
\\
\item Section~\ref{sec: dimReduction} introduces the dimensional reduction framework and reduces the dHYM system~\eqref{eq: dHYMIntroduction}+~\eqref{eq: dHYMPosIntroduction} to the {\em deformed vortex equations}; see equations~\eqref{eq: dimReduceDHYM}+~\eqref{eq: dimReduceDHYM-Pos}.  We introduce a method of continuity interpolating between a classical vortex type equation and the deformed vortex system.  We obtain formulas for the central charge~\eqref{eq: centralChargeIntro}.  We also collect several analytic consequences of the solvability of the dHYM system, and classify all bundles admitting solutions with ${\rm Re}(Z_{\cX}(\cE))\leq 0$ and ${\rm Im}(Z_{\cX}(\cE))> 0$; see Proposition~\ref{prop: charSolCosNeg}.
\\
\item Section~\ref{sec: stability} introduces the notion of $Z$-stability relevant for the dHYM equation.  We show that it suffices to test stability against a fairly simple class of objects; vector bundles and vector bundles supported over reduced $\mathbb{P}^1$ fibers of $\cX\rightarrow X$.  We define stability from two equivalent points of view; one in terms of $SU(2)$-equivariant coherent sheaves on $\cX$, along the lines of Definition~\ref{defn: Zstableintro}, and in terms of holomorphic triples on the underlying Riemann surface $X$; see Definition~\ref{def: stabOnXxP} and Definition~\ref{defn: stabOfTriple}.  We classify  all stable bundles with ${\rm Re}(Z_{\cX}(\cE))\leq 0$, finding exact agreement with the classification of solutions to the dHYM equation with the same property.  Finally, we prove the necessity of stability for angles in the range $\hat{\theta} \in (0, \arctan(\frac{\sigma}{4\pi})]\cup [\frac{\pi}{2}, \pi]$.  Together with the results of the previous section, this completes the proof of Theorem~\ref{thm: mainTheorem} for bundles with ${\rm Re}(Z_{\cX}(\cE))\leq 0$; see Corollary~\ref{cor: classifySolRealNeg}.
\\
\item Sections~\ref{sec: initialEquation}--\ref{sec: fixedPoint} are dedicated to proving that stability implies existence of solutions for bundles having ${\rm Re}(Z_{\cX}(\cE))> 0$; the main result of these sections is stated in Theorem~\ref{thm: existenceCos>0}.  In Section~\ref{sec: initialEquation}, following Bradlow \cite{Bradlow}, we prove the existence of solutions to the method of continuity at time $t=0$. Section~\ref{sec: Estimates}, in which we undertake the analysis of the deformed vortex equations, is the technical heart of the paper.  We prove that, under the stability assumption, and the assumption $\tan(\htheta) \leq \frac{\sigma}{4\pi}$ we have a priori estimates to all orders for solutions of the method of continuity.  The $C^0$ estimate for the metric on $\cE$ requires two compactness and contradiction arguments.  Technically, the most challenging aspect of these arguments is controlling the nonlinear terms in the equation along diverging sequences of metrics; see e.g. Proposition~\ref{prop: UYargumentH}.  In Section~\ref{sec: fixedPoint} we complete the proof of the main theorem via a topological degree argument.
\\
\item Section~\ref{sec: comments} investigates how our solutions fit into the existing literature on the higher rank dHYM equation.  We show by example that the deformed Hermitian-Yang-Mills equations, together with the estimates we prove, do not algebraically imply $Z$-positivity in the sense of \cite{DMS, KellerScarpa}.  We prove unconditionally that they satisfy the weaker Griffiths-$Z$-positivity condition.  This latter result depends in an essential way on the global analysis of the equation, and in particular on a priori estimates.  This should be compared with the results of Chen-Ghosh \cite{ChenGhosh} and Ballal-Pingali \cite{AshPin} which show that ellipticity for the higher rank dHYM system is not an {\em algebraic} consequence of solvability, unlike the case of the dHYM equation on line bundles. We outline an approach to studying the dimensional reduction of the dHYM equation via an infinite dimensional variational approach along the lines of Collins-Yau \cite{CY}. Finally, we discuss connections with Bridgeland stability, and prove that $Z$-stability is implied by Bridgeland stability in $D^{b}{\rm Coh}^{SU(2)}(\cX)$.
\end{itemize}
\smallskip

The assumption that $\htheta \in (0, \arctan(\frac{\sigma}{4\pi})]$ appears in exactly two places.  The first is in Lemma~\ref{lem: typeBStabNecessary} establishing the necessity of stability for quotients $0 \rightarrow \mathcal{S}\rightarrow \mathcal{E} \rightarrow \mathcal{Q} \rightarrow 0$ where $\mathcal{S}$ is the pull-back of a subbundle $S\subset E_1$.  The second appearance is in the proof of the key $C^0$ estimate in Proposition~\ref{prop: UYargumentH}. While we conjecture that equivariant $Z$-stability is equivalent to existence of solutions to the deformed Vortex equations for $\hat{\theta} \in (\arctan(\frac{\sigma}{4\pi}),\frac{\pi}{2})$, proving this result seems to require some new techniques, including possibly the development of the higher rank GIT picture, as developed in the rank $1$ case by Collins-Yau \cite{CY, CYarX}; see Conjecture~\ref{conj: stabTypeB} and the surrounding discussion, as well as Section~\ref{sec: comments}.  
\smallskip

\noindent {\bf Acknowledgements}: The authors would like to thank D. H. Phong and S.-T. Yau for their interest and encouragement.  The authors are very grateful to S. Brendle for suggesting the use of a topological degree argument in Section~\ref{sec: fixedPoint} and bringing to their attention the reference \cite{ChangGursky}.  The first author would also like to thank M. Douglas for a very helpful discussion on $\Pi$-stability and BPS $D$-branes.  T.C.C.~is supported in part by NSERC Discovery grant RGPIN-2024-518857.  The results in this paper will form part of the second author's forthcoming thesis at the University of Toronto \cite{YukaiThesis}.

\section{Background}\label{sec:background}

This section situates our main theorem within the broader literature surrounding the mirror symmetry, the deformed Hermitian–Yang–Mills equation, and stability conditions.  This area of research has seen a great deal of interest recently starting with the initial work of Jacob-Yau \cite{JacobYau} for holomorphic line bundles; see e.g. \cite{Ballal, ChanJacob, COSD, GChen, CCL, ChuLee, ChuLeeTak, CLSY,CJY, CXY, CY, CYarX, Correa, DatMetSong, DatPing, DMS, Ghosh, Jacob, JacobSheu, KellerScarpa, LaraSaEarp, Mete, Murakami, Murakami2, Pingali, Song, Tak1, Tak2, Takahashi} and the references therein.  One goal of the following exposition is to emphasize that, for higher rank bundles, one must include the calibration constraint~\eqref{eq: dHYMPosIntroduction} together with~\eqref{eq: dHYMIntroduction}.  To motivate this, we recall some basic facts about mirror symmetry and special Lagrangians.

\subsection{Mirror symmetry, special Lagrangians, and the dHYM equation}\,

Let $(\cX,g,\omega, \Omega)$ be a compact Calabi-Yau manifold with $\dim_{\mathbb{C}}\cX=n$.  A central role in type IIA string theory, and thereby mirror symmetry, is played by {\em special Lagrangians} (sLags), as introduced in the foundational work of Harvey-Lawson \cite{HL}.  

\begin{defn}
A smooth, oriented submanifold $L\subset \cX$ with $\dim_{\mathbb{R}}L=n$ is said to be special Lagrangian with phase $\htheta$
\begin{equation}\label{eq: backCalibsLag}
{\rm Re}(e^{-i\htheta}\Omega)\big|_{L} = d{\rm Vol}_{L}
\end{equation}
where $d{\rm Vol}_{L}$ is the Riemannian volume form on $L$ induced by the Ricci-flat metric $g$.
\end{defn}

Special Lagrangians are calibrated submanifolds and are thereby volume minimizing in their homology class. Typically, the special Lagrangian condition~\eqref{eq: backCalibsLag} is written in the following equivalent way:
\begin{equation}\label{eq: sLagImPart}
\omega|_{L} = 0 \quad \text{ and }\quad {\rm Im}(e^{-i\htheta}\Omega)|_{L}=0.
\end{equation}

For $L$ a smooth, connected, orientable submanifold one can show that ~\eqref{eq: sLagImPart} is equivalent to~\eqref{eq: backCalibsLag} up to possibly reversing the orientation of $L$.  However, we wish to emphasize that for special Lagrangians that are not smooth, or not connected, condition~\eqref{eq: sLagImPart} is {\em not} equivalent to being special Lagrangian, since it does not imply the calibration condition.  For example, suppose we have two Lagrangians $L_1, L_2$ of phase $\htheta$ (so in particular, $L_1,L_2$ are oriented).  For simplicity we can assume $L_1, L_2$ are disjoint, or intersect transversally.  Consider the union $L_{\pm}:= L_1 \cup \pm L_2$ where $\pm L_2$ denotes $L_2$ with possibly reversed orientation.  Then we have
\[
\omega|_{L_{\pm}}=0 \quad \text{ and } \quad {\rm Im}(e^{-i\htheta}\Omega)|_{L_{\pm}}=0
\]
however, only $L_{+}$ is calibrated by ${\rm Re}(e^{-i\htheta}\Omega)$.  Indeed, if $[L_1]=[L_2]$ in $H_{n}(X)$ (eg. if $L_1, L_2$ are fibers of a special Lagrangian fibration), then we have
\[
\int_{L_{-}}{\rm Re}(e^{-i\htheta}\Omega)=0.
\]
For this reason, the generalization of the condition~\eqref{eq: sLagImPart} beyond the smooth case (or even to smooth oriented Lagrangians) is the following

\begin{defn}\label{defn: sLagIm+Re}
    We say that $L\subset \cX$ is special Lagrangian if
    \begin{equation}\label{eq: SlagIntroDef}
    \omega|_{L}=0 \quad \text{ and } \quad {\rm Im}(e^{-i\htheta}\Omega)|_{L}=0, \quad \text{ and } \quad {\rm Re}(e^{-i\htheta}\Omega)|_{L} >0.
    \end{equation}
\end{defn}
We have been deliberately vague about what regularity assumptions should be imposed on the special Lagrangian $L$ in Definition~\ref{defn: sLagIm+Re}.  For example, one can assume that $L$ is smooth outside of a codimension $2$ set, or interpret~\eqref{eq: SlagIntroDef} weakly using the theory of currents.

\subsubsection{The deformed Hermitian-Yang-Mills equation}

We briefly summarize the origins of the deformed Hermitian-Yang-Mills equations (dHYM); we refer the reader to \cite{LYZ, MMMS, CXY} and the references therein for a more thorough treatment. The dHYM equation governs supersymmetric $D$-branes in type IIB string theory.  For complex line bundles the equations of motion were computed independently by Mari\~no-Minasian-Moore-Strominger \cite{MMMS} from an action functional point of view, and by Leung-Yau-Zaslow \cite{LYZ} by applying the Fourier-Mukai transform to a special Lagrangian section of an SYZ fibration in the setting of semi-flat mirror symmetry.  Solutions of the equations of motion correspond to Hermitian, holomorphic line bundles $(\cL,h)\rightarrow (\cX,\omega)$ such that
\begin{equation}\label{eq: backgroundDHYMRank1}
{\rm Im}\left(e^{-i\htheta}(\omega -F_h)^n\right)=0
\end{equation}
where $F_h$ is the curvature of the Chern connection associated to $h$, and $e^{-i\htheta}$ is a constant. These equations also appear in the theory of generalized geometry where they describe generalized calibrated cycles \cite{Gualtieri, Martucci-Smyth}. As mentioned in both \cite{MMMS, LYZ} the calculation of the equations of motion does not carry over to the case of non-abelian gauge groups. It has therefore been proposed  in the physics literature \cite{MinTom}, as well as in the mathematics literature \cite{CYarX, DMS} that one study the natural extension of~\eqref{eq: backgroundDHYMRank1} to higher-rank bundles. We argue that, from the perspective of mirror symmetry, this equation should be supplemented by the calibration constraint.

Suppose that $\cX$ is a compact Calabi-Yau manifold admitting a special Lagrangian fibration $\pi: \cX\rightarrow B$, and $\check{\pi}: \check{\cX}\rightarrow B$ is the mirror Calabi-Yau equipped with the dual special Lagrangian fibration. Under mirror symmetry, a vector bundle $\cE \rightarrow \cX$ admitting a solution of the higher rank dHYM equation is mirror to a special Lagrangian {\em multi-section} of the SYZ fibration $\check{\pi}$; that is, a sLag $L\subset \check{\cX}$ such that $\check{\pi}(L)=B$ and $L$ intersects the generic fiber of $\check{\pi}$ transversely in $r'>1$ points.  The details of this correspondence are not well understood in general, due to quantum corrections arising from the branching locus of the Lagrangian multi-section \cite{LYZ}.  Nevertheless, we can already draw conclusions from the case of direct sums of line bundles where the gauge groups are abelian and hence the calculations of \cite{LYZ, MMMS} carry over. Suppose $\cL_1, \cL_2$ are holomorphic line bundles over $\cX$ and consider $\cE= \cL_1 \oplus \cL_2$.  If $\cL_i$ is mirror to a Lagrangian $L_i$, then $\cE$ is mirror to the union $L_1\cup L_2$.  As we saw in the previous section, $L_1\cup L_2$ is special Lagrangian if and only if it satisfies the system~\eqref{eq: SlagIntroDef} which, by \cite{LYZ} is mirror to the system~\eqref{eq: dHYMIntroduction}+~\eqref{eq: dHYMPosIntroduction}.

Another perspective in support of the system~\eqref{eq: dHYMIntroduction}+~\eqref{eq: dHYMPosIntroduction} is based on a calibration argument generalizing an argument of Jacob-Yau \cite{JacobYau} from the rank $1$ case.  This argument can be viewed either as the formal mirror of the special Lagrangian calibration argument of Harvey-Lawson \cite{HL} or as an independent BPS condition in type IIB.  Consider the quantity
\[
\zeta(h):= e^{-i\htheta}(\omega \otimes I_{\cE}-F)^n 
\]
which is (formally) mirror to $e^{-i\htheta}\Omega|_{L}$ for the dual Lagrangian multisection $L$.  Then we have
\[
|Z_{\cX}(\cE)| = \bigg|\int_{X} e^{-i\htheta}\Tr(\zeta(h))\bigg| \leq \int_{X}|{\rm Tr}\zeta(h)|\,\omega^n \leq \int_{X}{\rm Tr}|\zeta(h)|\,\omega^n 
\]
where
\[
|\zeta(h)| := \left(\zeta(h)\zeta(h)^{\dagger}\right)^{1/2} 
\]
so that ${\rm Tr}|\zeta(h)|$ is the Schatten $1$-norm.  Tracing the equality case we have:
\begin{itemize}
    \item[(i)]${\rm Tr}|\zeta(h)| = |{\rm Tr}(\zeta(h)|$ if and only if there is a real valued function $\varphi(x)$ so that\\ ${\zeta(h) = e^{i\varphi(x)} |\zeta(h)|}$
    \item[(ii)] $\bigg|\int_{X} e^{-i\htheta}\Tr(\zeta(h))\bigg| \leq \int_{X}|{\rm Tr}\zeta(h)|\,\omega^n$ with equality if and only if $e^{i(\varphi(x)-\htheta)}=1$.
\end{itemize}
Thus, we have
\[
\int_{X}{\rm Tr}|\zeta(h)|\,\omega^n \geq |Z_{\cX}(\cE)| 
\]
with equality if and only if the equations ~\eqref{eq: dHYMIntroduction}+~\eqref{eq: dHYMPosIntroduction} hold.

The system~\eqref{eq: dHYMIntroduction} is a fully nonlinear system which may not be elliptic.  For these reasons, the higher rank dHYM equation has seen a limited amount of study.  Dervan-McCarthy-Sektnan \cite{DMS} studied a broad class of nonlinear partial differential equations on vector bundles motivated by connections to Bridgeland stability conditions on $D^{b}Coh(\cX)$ with polynomial central charges.  They prove that if a holomorphic vector bundle is {\em asymptotically $Z$-stable with respect to subbundles} then it admits a solution of the dHYM equation {\em in the large volume limit}.  Roughly speaking, the infinite volume limit of the dHYM equation is the classical Hermitian-Yang-Mills equation, while asymptotic $Z$-stability implies Mumford-Takemoto semi-stability, and so the existence result can be approached by a singular perturbation problem along the lines of \cite{Leung97}.  A key feature of this asymptotic regime is that the equation~\eqref{eq: dHYMIntroduction} becomes elliptic.  Dervan-McCarthy-Sektnan \cite{DMS} introduce a notion of subsolution, which we will call $Z$-positivity (following \cite{KellerScarpa}) for the system~\eqref{eq: dHYMIntroduction}, based on analogy with the case of line bundles \cite{CJY}. The subsolution condition implies the ellipticity of the system~\eqref{eq: dHYMIntroduction} \cite[Lemma 2.36]{DMS}, and holds automatically in the asymptotic regime \cite[Lemma 2.38]{DMS}.

One possible expectation is that $Z$-positivity, and hence ellipticity, might hold in a sufficiently small neighborhood of any solution of~\eqref{eq: dHYMIntroduction}. The pointwise examples of Chen–Ghosh \cite{ChenGhosh} and Ballal-Pingali \cite{AshPin}  show that such a conclusion cannot follow from the equation by purely algebraic considerations.  On the other hand, solutions of nonlinear PDEs are fundamentally non-local and so one might hope that the pointwise examples of \cite{ChenGhosh, AshPin} never appear in global solutions.  One outcome of our work is numerical evidence that $Z$-positivity may not hold, even at solutions of the extended system~\eqref{eq: dHYMIntroduction}+~\eqref{eq: dHYMPosIntroduction}. On the other hand, by imposing~\eqref{eq: dHYMPosIntroduction}, the system remains elliptic and a weaker notion of positivity, analogous to Griffiths positivity, does hold; see Section~\ref{sec: comments}. 

As we shall see in the course of the proof, imposing the positivity condition~\eqref{eq: dHYMPosIntroduction} in addition to~\eqref{eq: dHYMIntroduction} is essential to both the algebro-geometric and analytic theory, being closely connected with ellipticity, Bridgeland stability conditions, and torsion sheaves. 

\subsection{Stability conditions and geometric PDEs in mirror symmetry}

Thomas \cite{Th} and Thomas-Yau \cite{ThY} proposed a notion of stability for Lagrangians and predicted that a Lagrangian $L$ could be deformed by Hamiltonian deformations to a special Lagrangian if and only if $L$ is stable in their sense.  This proposal was based in part on a moment map formalism for special Lagrangians discovered by Thomas \cite{Th}, and in part on the analogy with the Donaldson-Uhlenbeck-Yau theorem, motivated by mirror symmetry.  More recently Joyce \cite{Joyce} proposed a very broad update to the Thomas-Yau conjecture in the framework of Bridgeland stability and the mean curvature flow.  Joyce's conjecture gives a remarkably detailed picture for how singularities of the Lagrangian mean curvature flow can be used to {\em construct} a canonical Bridgeland stability condition on the derived Fukaya category.  Our interest is in the following broader conjecture

\begin{conj}\label{conj: folklore}
There is a Bridgeland stability condition on $D^{b}{\rm Fuk}(\check{\cX})$ (resp. $D^{b}Coh(\cX)$) so that the isomorphism class of a Lagrangian $L$ (resp. holomorphic vector bundle $\cE$) is stable if and only if it contains a special Lagrangian (resp. $\cE$ admits a metric solving the deformed Hermitian-Yang-Mills equation).
\end{conj}

From the string theory point of view, Conjecture~\ref{conj: folklore} can be viewed as the statement that BPS $D$-branes in type IIA/B can be described either algebraically (by a stability condition) or analytically (by solutions of geometric PDEs).  From this point of view, Conjecture~\ref{conj: folklore} is extremely natural.  While the analytic  approach to BPS D-branes was pioneered in \cite{MMMS, LYZ},   the algebraic approach was initiated by Douglas-Fiol-R\"omelsberger \cite{DFR} and Douglas \cite{Doug}, inspired in part by the Donaldson-Uhlenbeck-Yau theorem \cite{Do83, Do85, Do87, UY}.  The key idea of this approach is to bypass the equations of motion for BPS branes, and instead define an algebro-geometric notion, called $\Pi$-stability, which encodes the structural features of BPS D-branes and agrees with physical expectations in various asymptotic regimes.  The proposal of \cite{DFR, Doug} can then be summarized as ``an object in $D^{b}Coh(X)$ is a supersymmetric $A/B$-brane if it is $\Pi$-stable". Bridgeland \cite{Br} placed the ideas of $\Pi$-stability on a rigorous and very general mathematical foundation \cite{Br}, developing the notion of a Bridgeland stability condition.  Following Bridgeland's pioneering work, the subject of stability conditions on categories, and particularly $D^{b}Coh(X)$ and $D^{b}{\rm Fuk}(\check{X})$, has generated a tremendous amount of interest; see, for example \cite{AB, BMT, MacriSchmidt, ChunyiLi, ChunyiLi2} and the references therein.

\subsubsection{The dHYM equation in rank $1$ and stability conditions}
The dHYM equation on line bundles was first studied in Jacob-Yau \cite{JacobYau}.  Collins-Jacob-Yau \cite{CJY} introduced a notion of subsolution for the dHYM equation, based on work of Sz\'ekelyhidi \cite{Sz} and Guan \cite{Guan} and proved that, in the supercritical phase regime, existence of a subsolution is equivalent to existence of solution to the dHYM equation.  In \cite{CJY} it was also observed that the existence of a subsolution implied certain intersection inequalities, which yield algebro-geometric obstructions to existence.  Collins-Yau \cite{CY, CYarX} developed an infinite dimensional GIT approach to the existence of solutions to the dHYM equation for line bundles $\cL \rightarrow \cX$ and developed a conjectural picture relating existence of solutions to the dHYM equation in the supercritical phase case with Bridgeland-type stability conditions.  We will give only a very rough overview of this picture, since our focus in this paper is on the higher rank case.  Roughly speaking, this theory has two components:
\begin{itemize}
    \item[(A)] $\dim_{\mathbb{C}}\cX-2$ necessary Chern number inequalities (involving $ch_1(\cL),...,ch_n(\cL)$).
    \item[(B)] Stability inequalities arising from torsion quotients, $\cL \twoheadrightarrow \cL \otimes \mathcal{O}_{Z}$.
\end{itemize}
The Chern number inequalities were proved to be necessary by Collins-Xie-Yau \cite{CXY} in dimension $3$, and by Han-Jin \cite{HanJin} in dimension $4$.  Interestingly, the Chern number inequality in dimension $3$ is identical to the improved Bogomolov-Gieseker inequality conjectured by Bayer-Macri-Toda \cite{BMT} in their study of Bridgeland stability conditions on $3$-folds.  This conjecture turns out to be false in general \cite{Schmidt}, but holds for line bundles admitting solutions of the dHYM equation \cite{CXY}. 

Beginning with a breakthrough of Chen \cite{GChen}, and subsequent works of Chu-Lee \cite{ChuLee}, Chu-Lee-Takahashi \cite{ChuLeeTak}, Datar-Pingali \cite{DatPing}, Song \cite{Song}, and Ballal \cite{Ballal} it is now understood that if the inequalities in (B) hold along a {\em test family} (see \cite{GChen} for this notion), then a solution to the rank $1$ dHYM equation exists. In complex dimension $3$ it was shown by Chu-Lee \cite{ChuLee} that the test family condition could be replaced by the Chern number inequality in $(A)$, leading to a complete resolution of the Collins-Yau conjecture in dimension $3$.

When $\dim_{\mathbb{C}}\cX=2$ the picture simplifies considerably.  There are no necessary Chern number inequalities in $(A)$. The stability inequalities in $(B)$ imply existence as a consequence of the Demailly-P\u{a}un theorem \cite{DP} and Yau’s solution of the Calabi conjecture \cite{Yau}; see \cite{CJY}.

Since Bridgeland stability conditions are known to exist in general on complex surfaces \cite{AB} it is natural to ask whether Bridgeland stable line bundles admit solutions of the dHYM equation, and vice versa.  This was taken up in \cite{CShi} on the blow-up of $\mathbb{P}^2$ at a point, and in \cite{CLSY} for certain line bundles on Weierstrass elliptic $K3$ surfaces.  Perhaps surprisingly, it was found that line bundles admitting solutions of the dHYM equation are Bridgeland stable, but not conversely.  In fact, by \cite{CLSY} Bridgeland stable line bundles can be ``very far" from admitting solutions of the dHYM equations.  

Beyond the supercritical phase case the only results we are aware of are due to Jacob-Sheu \cite{JacobSheu} and Jacob \cite{Jacob} in the Calabi symmetric setting.

\subsubsection{The dHYM equation in higher rank and stability conditions}

In comparison with the rank $1$ case very little is known for the higher rank dHYM equation.  A natural conjectural picture emerges by considering direct sums of line bundles and applying the rank-one theory. The rough conjecture is that existence of solutions to the higher rank dHYM equation on a vector bundle $\cE$ should be equivalent to
\begin{itemize}
    \item[(A)] $\dim_{\mathbb{C}}\cX-2$ necessary Chern number inequalities (involving $ch_0(\cE), ch_1(\cE),...,ch_n(\cE)$).
    \item[(B)] Stability inequalities arising from (possibly torsion) quotients, $\cE \twoheadrightarrow \cQ$.
\end{itemize}

The only result we are aware of which makes this connection is the work of Keller-Scarpa \cite{KellerScarpa} who proved that if $\cE \rightarrow \cX$ is a holomorphic vector bundle of rank $2$ over a complex surface that admits a $Z$-positive (in the sense of \cite{DMS}) solution of ~\eqref{eq: dHYMIntroduction} then the conditions arising from surjections as in $(B)$ are {\em necessary}.  The argument of Keller-Scarpa uses $Z$-positivity together with the rank $2$ assumption to identify a pointwise positive quantity which can be integrated over $X$ to give the desired inequality.  As we will see in Section~\ref{sec: stability}, proving the necessity of stability for arbitrary rank bundles without $Z$-positivity is fundamentally non-local, depending in an essential way on an integrated Bochner-type formula together with weaker notions of positivity. 

\subsubsection{Vortices}

Our approach to Theorem~\ref{thm: mainTheorem} is by dimensional reduction, using the $SU(2)$-equivariance to rewrite the equation in terms of a pair of metrics $(h_1,h_2)$ on a {\em holomorphic triple} $(E_1, E_2, \Phi)$ where $E_i\rightarrow X$ are holomorphic vector bundles $\rank(E_2)=1$ and $\Phi \in H^{0}(X, E_1\otimes E_2^{\vee})$.  As shown by Garcia-Prada \cite{GPVortex}, if one pursues this strategy in the setting of the Hermitian-Yang-Mills equation the resulting equations on $X$ are the Vortex equations, originally introduced by Bradlow \cite{Bradlow} as a generalization of the vortices studied by Jaffe-Taubes \cite{JaffeTaubes}.  The classical  equations couple the curvature of the metric on $E_1$ (resp. $E_2)$ with the natural endomorphism of $E_1$ given by $\Phi\Phi^{\dagger}$ (resp. $\Phi^{\dagger}\Phi$) in a similar way to the equations for Higgs bundles.  In contrast, the dimensionally reduced dHYM system~\eqref{eq: dHYMIntroduction}+~\eqref{eq: dHYMPosIntroduction} leads to a system in which the derivative $\nabla \Phi$ appears quadratically and the ``Higgs field" $\Phi\Phi^{\dagger}$ couples quadratically with the curvature (see~\eqref{eq: dimReduceDHYM} + ~\eqref{eq: dimReduceDHYM-Pos}); we call these equations the deformed Vortex equations.  

Pingali \cite{Pingali}, Takahashi \cite{Takahashi} and Ghosh \cite{Ghosh} studied ~\eqref{eq: dHYMIntroduction} (and related equations) on $SU(2)$-equivariant holomorphic vector bundles of rank $2$.  In this case, the system~\eqref{eq: dHYMIntroduction} reduces to two coupled scalar equations, and the stability-theoretic phenomena that arise in arbitrary rank are absent from the reduced problem.

Other nonlinear extensions of the Vortex equations have recently generated a great deal of interest, perhaps most notably the gravitating vortices; see e.g. \cite{ACGFGP, ACGFGPPing, GFPY}.

\section{Dimensional reduction of the deformed Hermitian-Yang-Mills equation}\label{sec: dimReduction}

Consider the complex surface given by $\cX := X \times \mathbb{P}^1$ and let $p_1: \cX \rightarrow X$ and $p_2:\cX \rightarrow \mathbb{P}^1$. Given $\sigma \in \mathbb{R}_{>0}$,  fix a K\"ahler form on $\cX$ by
\[
\widetilde{\omega}_{\sigma}:= p_1^*\omega \oplus \frac{\sigma}{2\pi} p_2^*\omega_{FS}
\]
where $\frac{1}{2\pi}\omega_{FS}\in c_1(\mathcal{O}_{\mathbb{P}^1}(1))$.  We have 
\[
\frac{1}{2\pi}\int_{\mathbb{P}^1}\omega_{FS}=1, \quad \int_{X}\omega :=V_{X}>0.
\]
The group $SU(2)$ acts on $\mathbb{P}^1 = SU(2)/U(1)$ by multiplication and in this section we reduce the dHYM equation ~\eqref{eq: dHYMIntroduction}+~\eqref{eq: dHYMPosIntroduction} for $SU(2)$ equivariant bundles on the complex surface $\cX= X \times \mathbb{P}^1$ to a coupled system of equations on $X$.

\subsection{$SU(2)$-equivariant holomorphic vector bundles}\label{sec: equivariantVBconstruction}

The following proposition describes equivariant complex vector bundles on $\cX$.
\begin{prop}[Garcia-Prada \cite{GPVortex}]
\label{prop: topologicalClass}
Every $SU(2)$-equivariant complex vector bundle $\mathcal{E}$ over $X\times \mathbb{P}^1$ can be equivariantly decomposed, uniquely up to isomorphism, as
\begin{equation}\label{eq: direcSumDecomp}
\mathcal{E}= \bigoplus_{i=1}^{m} \mathcal{E}_i
\end{equation}
where $\mathcal{E}_i = p_1^*E_i \otimes p_2^*\mathcal{O}_{\mathbb{P}^1}(n_i)$ with $E_i$ complex vector bundles over $X$ and $n_i\in \mathbb{Z}$ distinct.  Furthermore, if $\tilde{h}$ is any $SU(2)$-invariant Hermitian metric on $\mathcal{E}$, then the decomposition~\eqref{eq: direcSumDecomp} is orthogonal.  That is
\[
\tilde{h}= \oplus_{i=1}^{m}\tilde{h}_i
\]
where $\tilde{h}_i = p_1^*h_i \otimes p_2^*h_i'$ where $h_i$ is a Hermitian metric on $E_i$ and $h_i'$ is an $SU(2)$-invariant metric on $\mathcal{O}_{\mathbb{P}^1}(n_i)$.
\end{prop}

We emphasize that Proposition~\ref{prop: topologicalClass} classifies complex vector bundles, not holomorphic vector bundles.

\begin{defn}\label{defn: vortexBundle}
We say that an $SU(2)$ equivariant complex vector bundle $\cE\rightarrow \cX$ is a vortex bundle if
\begin{equation}\label{eq: vortexBundleChoice}
\mathcal{E}= \mathcal{E}_1\oplus \mathcal{E}_2 = p_1^{*}E_1 \oplus p_1^*E_2\otimes p_2^*\mathcal{O}_{\mathbb{P}^1}(2)
\end{equation}
where $E_i \rightarrow X$ are holomorphic vector bundles, and $\rank(E_2)=1$. 
\end{defn}
Following Garcia-Prada \cite{GPVortex}, an $SU(2)$-equivariant holomorphic structure on $\mathcal{E}\rightarrow \cX$ is determined by a triple $(\dbar_{E_1},\dbar_{E_2},\Phi)$
where $\dbar_{E_i}$ is a holomorphic structure on $E_i\rightarrow X$ for $i=1,2$, and $\Phi \in H^{0}(X, E_1\otimes E_2^{\vee})$ denotes a holomorphic map from $E_2\rightarrow E_1$.  We will make this construction explicit below, but for now we recall the following definition, due to Bradlow-Garcia-Prada \cite{BG-P}. 

\begin{defn}\label{defn: holcTriple}
    A holomorphic triple on $X$ is a triple $(E_1,E_2,\Phi)$ where $E_1, E_2$ are holomorphic vector bundles and $\Phi \in H^{0}(X,{\rm Hom}(E_2, E_1))$.
\end{defn}

We now recall the construction of the Chern connection on $\mathcal{E}$, given the data of $(\dbar_{E_1},\dbar_{E_2},\Phi)$, following \cite{GPVortex}.  Let $h_1, h_2$ be Hermitian metrics on $E_1, E_2$ and let $d_{A_i}$ denote the associated Chern connections on $E_i$.  By a slight abuse of notation, we let $d_{A_i}$ also denote the Chern connection on $p_1^*E_i$ induced by the pull-back metric.  Let $h_{FS}$ be the Fubini-Study metric on $\mathcal{O}_{\mathbb{P}^1}(1)$, and let $d_{A_{FS}}$ denote the Chern connection on $\mathcal{O}_{\mathbb{P}^1}(1)$ and recall that $\omega_{FS} = iF_{h_{FS}}$.  By abuse of notation, we also use $d_{A_{FS}}$ to denote the Chern connection on $p_2^*\mathcal{O}_{\mathbb{P}^1}(1)$.  The connection on $\cE$ is determined by
\[
d_{A} = \begin{pmatrix} d_{A_1} & \beta\\
-\beta^{\dagger}
 & d_{A_2 \otimes A^{\otimes 2}_{FS}}\end{pmatrix}
 \]
 where
 \[
 \beta = p_1^*\Phi\wedge p_2^*\alpha
 \]
and $\alpha$  is the unique up to scale $SU(2)$-invariant section of $\Lambda_{\mathbb{P}^1}^{0,1}\otimes \mathcal{O}_{\mathbb{P}^1}(-2)$.  Let us suppress the pull-backs $p_i^*$ for simplicity.  Fix a choice of $\alpha$ by setting
\[
\alpha = \sqrt{\frac{\sigma}{2\pi}}\frac{dz}{(1+|z|^2)^2} \otimes d\bar{z} \quad \text{ so that} \quad
\alpha \wedge \alpha^\dagger = \frac{\sigma}{2\pi}i \omega_{FS}.
\]
In particular, this implies
\[
\begin{aligned}
\beta \wedge \beta^{\dagger} = i \frac{\sigma}{2\pi} p_1^*(\Phi \Phi^\dagger)\otimes p_2^*\omega_{FS}\qquad \text{and}\qquad
\beta^{\dagger} \wedge \beta = -i\frac{\sigma}{2\pi}p_1^*|\Phi|_h^2\otimes p_2^*\omega_{FS}
\end{aligned}
\]
By the standard formula for the decomposition of the curvature we have
\[
F_{A} = \begin{pmatrix}F_{A_1} - \beta \wedge \beta^{\dagger} & \nabla \beta\\
-(\nabla \beta)^{\dagger} & F_{A_2}- \beta^{\dagger}\wedge \beta\end{pmatrix}.
\]
Since $\alpha$ is parallel we have that
$
\nabla \beta = p_1^*(\nabla \Phi) \wedge \alpha
$
where $\nabla$ denotes the $(1,0)$ component of the Chern connection on $E_1\otimes E_2^{\vee}\rightarrow X$ defined by the metric $h:=h_1h_2^{-1}$.  Thus, the curvature of the Chern connection on $\mathcal{E}$ is given by 
\[
F_{A} = \begin{pmatrix} p_1^*F_{A_1}- i\frac{\sigma}{2\pi} \Phi \Phi^{\dagger} \omega_{FS} & p_1^*(\nabla\Phi) \otimes \alpha\\
-\left(p_1^*\nabla\Phi\right)^{\dagger} \otimes \alpha^{\dagger} & p_1^*F_{A_2}  - 2i\omega_{FS} + i\frac{\sigma}{2\pi} |\Phi|^2_{h}\omega_{FS} \end{pmatrix}
\]
We need to compute the square of $\widetilde{\omega}_{\sigma}\otimes I_{\mathcal{E}} - F_{A}$.  Let $\Lambda$ denote the contraction with $\omega$, then
\[
\frac{(\widetilde{\omega}_{\sigma}\otimes I_{\mathcal{E}} - F_{A})^2}{\widetilde{\omega}_{\sigma}^2} = \begin{pmatrix} M_{11} & M_{12}\\ M_{21} & M_{22} \end{pmatrix}
\]
where
\begin{equation}\label{eq: curvComponentSquare}
\begin{aligned}
M_{11}&=  I_{E_1} -\Lambda F_{A_1} + i \Phi\Phi^\dagger  -\frac{1}{2}\{i\Lambda F_{A_1},\Phi\Phi^{\dagger}\}+ \frac{i}{2}\Lambda\left( \nabla^{1,0}\Phi \wedge(\nabla^{1,0}\Phi)^\dagger\right)\\
M_{12}&=M_{21}=0\\
M_{22} &= (1+\frac{4\pi i}{\sigma})I_{E_2} -(1+\frac{4\pi i}{\sigma} )\Lambda F_{A_2} -i|\Phi|^2+|\Phi|^2i\Lambda F_{A_2} +\frac{1}{2}|\nabla^{1,0}\Phi|^2.
\end{aligned}
\end{equation}
and we have used $\{A , B\}= AB+BA$ to denote the anti-commutator.  In particular, we have
\begin{equation}\label{eq: firstStepCentCharge}
\begin{aligned}
\frac{1}{2\sigma}\int_{X\times \mathbb{P}^1}{\rm Tr}(\widetilde{\omega}_{\sigma} \otimes I_{\mathcal{E}}-F_{A})^2 &=  \rank(E_1)V_{X} +2\pi i\deg(E_1)+(1+\frac{4\pi i}{\sigma})V_{X}+(1+\frac{4 \pi i}{\sigma})2\pi i\deg E_2\\
&- \int_{X}{\rm Tr}\left[(i\Lambda F_{A_1}- I_{E_1}\otimes i\Lambda F_{A_2})\Phi\Phi^\dagger\right] + \int_{X}|\nabla^{1,0}\Phi|^2 
\end{aligned}
\end{equation}
We have the following integration by parts formula using that $\Phi$ is holomorphic, and recalling that $h= h_1h_2^{-1}$ is the metric on the bundle ${\rm Hom}(E_2,E_1)$:
\[
\begin{aligned}
0&=\int_{X}\Delta|\Phi|_h^2 \\
&= \int_{X}|\nabla^{1,0}\Phi|_h^2 - \int_{X}\langle \Phi, (i\Lambda F_{A_1}-I_{E_1}\otimes i\Lambda F_{A_2}) \Phi \rangle_{h} 
\end{aligned}
\]
where we suppress the volume form. Thus, the term on the second line of~\eqref{eq: firstStepCentCharge} is zero and we obtain
\begin{equation}\label{eq: sec3CentralCharge}
\begin{aligned}
Z_{\cX}(\mathcal{E}) &:= -\int_{X\times \mathbb{P}^1} e^{-\frac{i}{2\pi}\widetilde{\omega}_{\sigma}}ch(\mathcal{E})\\
&= \frac{1}{2(2\pi)^2}\int_{X\times \mathbb{P}^{1}}\Tr (\widetilde{\omega}_{\sigma} \otimes I_{\mathcal{E}} - F_A)^2\\
&= \frac{\sigma}{2\pi} \left( \left[\frac{V_{X}}{2\pi}\left(\rank(E_1)+1\right) -\frac{4\pi }{\sigma} \deg(E_2)\right] +i\bigg[\deg(E_1)+ \deg E_2 +\frac{2 V_{X}}{\sigma}\bigg]\right)\\
\end{aligned}
\end{equation}
Throughout this paper we shall assume that
\[
Z_{\cX}(\mathcal{E}) \in \{ z \in \mathbb{C}^* : {\rm Im}(z)>0\} \cup \{z \in \mathbb{C}^*: {\rm Im}(z)=0, {\rm Re}(z)<0\}
\]
and write
\[
Z_{\cX}(\mathcal{E}) = r e^{i\hat{\theta}} \qquad r\in \mathbb{R}_{>0}.
\]
In Section~\ref{sec: stability} we will also write
\[
\Theta(\cE) := {\rm Arg}( e^{i\htheta}) = \htheta \in (0, \pi]
\]
where $\rm Arg$ denotes the principal value of the argument.

From~\eqref{eq: curvComponentSquare} we find that the dimensionally reduced deformed Hermitian-Yang-Mills equation~\eqref{eq: dHYMIntroduction}+~\eqref{eq: dHYMPosIntroduction} for $\mathcal{E}$ is the following system on the compact Riemann surface $(X,g, \omega)$: 

\begin{equation}\label{eq: dimReduceDHYM}
\begin{aligned}
&\cos(\hat{\theta})\left(i\Lambda F_{A_1} + \Phi\Phi^\dagger \right)=\sin(\hat{\theta})\left(I_{E_1}-\frac{1}{2}\{i\Lambda F_{A_1}, \Phi\Phi^\dagger\} +\frac{1}{2}g^{p\bar{j}}\nabla_p\Phi \left(\nabla_j \Phi)^{\dagger}\right)\right)\\
&\cos(\hat{\theta})\left(i\Lambda F_{A_2} +\frac{4\pi}{\sigma}I_{E_2}  -|\Phi|^2\right)=\sin(\hat{\theta})\left(I_{E_2} +\left(|\Phi|^2-\frac{4\pi}{\sigma}\right)i\Lambda F_{A_2} +\frac{1}{2}|\nabla^{1,0}\Phi|^2\right)
\end{aligned}
\end{equation}
together with the {\em calibration constraint}
\begin{equation}\label{eq: dimReduceDHYM-Pos}
\begin{aligned}
    &\cos(\hat{\theta})\left(I_{E_1}-\frac{1}{2}\{i\Lambda F_{A_1}, \Phi\Phi^\dagger\} +\frac{1}{2}g^{p\bar{j}}\nabla_p\Phi \left(\nabla_j \Phi)^{\dagger}\right)\right) + \sin(\hat{\theta})\left(i\Lambda F_{A_1} + \Phi\Phi^\dagger \right)>0\\
    &\cos(\hat{\theta})\left(I_{E_2} +\left(|\Phi|^2-\frac{4\pi}{\sigma}\right)i\Lambda F_{A_2} +\frac{1}{2}|\nabla^{1,0}\Phi|^2\right)+\sin(\hat{\theta})\left(i\Lambda F_{A_2} +\frac{4\pi}{\sigma}I_{E_2}  -|\Phi|^2\right)>0\\
\end{aligned}
\end{equation}

We call these equations {\em the deformed Vortex equations}.  

\subsection{The method of continuity}

When $\cos(\htheta)>0$ we shall solve the system~\eqref{eq: dimReduceDHYM}+~\eqref{eq: dimReduceDHYM-Pos} using the following method of continuity:

\begin{equation}\label{eq: dimReduceDHYM-MOC}
\begin{aligned}
&\cos(\hat{\theta})(i\Lambda F_{A_1} + \Phi\Phi^\dagger )=\sin(\hat{\theta})(I_{E_1}-\frac{t}{2}\{i\Lambda F_{A_1}, \Phi\Phi^\dagger\} +\frac{t}{2}g^{p\bar{j}}\nabla_p\Phi \left(\nabla_j \Phi)^{\dagger}\right)\\
&\cos(\hat{\theta})\left(i\Lambda F_{A_2} +\frac{4\pi}{\sigma}I_{E_2}  -|\Phi|^2\right)=\sin(\hat{\theta})\left(I_{E_2} +\left(t|\Phi|^2-\frac{4\pi}{\sigma}\right)i\Lambda F_{A_2} +\frac{t}{2}|\nabla^{1,0}\Phi|^2\right).
\end{aligned}
\end{equation}
where $t\in[0,1]$.

\subsection{Notation and conventions}

We briefly fix some notation, which will be used throughout the paper.  As some parts of the paper require lengthy formulas and calculations, we have tried to reduce as much as possible the notation while preserving the clarity of the exposition.
\begin{itemize}
    \item $h_i$ will denote Hermitian metrics on the vector bundles $E_i\rightarrow X$ and $F_{i}$ will denote the curvature of the Chern connections of $h_i$.
    \item $h$ will denote the metric $h_{1}h_{2}^{-1}$ on $E_1\otimes E_{2}^{\vee}$, and $F$ will denote the curvature of $h$.
    \item The $(1,0)$ component of the Chern connection is denoted by $\nabla$, while the $(0,1)$ component is denoted by $\dbar$ in a holomorphic frame, and by $\overline{\nabla}$ in a general frame.
    \item Unless otherwise noted, all covariant derivatives and norms are computed using the Chern connections induced by $h_1, h_2$ and the K\"ahler metric $\omega$ on $X$.  We use $A^{\dagger}$ to denote the Hermitian conjugate of $A$
 defined using the metrics $h_1,h_2$.
 \item Integrals are always computed using the volume form $\omega$, but to ease notation this is suppressed.  
    \item We will often write (with a minor abuse of notation) 
    \[
    i\Lambda\left(\nabla \Phi \wedge (\nabla \Phi)^{\dagger}\right) = \nabla \Phi (\nabla \Phi)^{\dagger}
    \]
    To illustrate the last two points with an example, we will write
    \[
    \int_{X}|\nabla\Phi|^2_{h_1h_{2}^{-1}\otimes \omega} \omega  = \int_{X}|\nabla\Phi|^2 = \int_{X}{\rm Tr}\left(\nabla \Phi(\nabla \Phi)^{\dagger}\right)
    \]
    \item The symbol $\{ A, B\}$ denotes the anti-commutator.
\end{itemize}

\subsection{Curvature formulas and consequences}
We collect some consequences of the equations~\eqref{eq: dimReduceDHYM-MOC},~\eqref{eq: dimReduceDHYM} and~\eqref{eq: dimReduceDHYM-Pos} which will be essential in the analysis to follow. Consider the following linear map $T_{\Phi,t} \in {\rm End}\left({\rm Hom}(E_1,E_1)\right) $ defined by
\begin{equation}\label{eq: TPhiOperator}
\begin{aligned}
 T_{\Phi,t}(A) = \cos(\htheta) A + \frac{t}{2}\sin(\htheta)\{ \Phi\Phi^{\dagger}, A\}.
\end{aligned}
\end{equation}
Provided that $\cos(\htheta)>0$ and $t \geq 0$, the map $T_{\Phi,t}$ is invertible.  Furthermore, it is clear that $T_{\Phi,t}$ preserves the space of Hermitian endomorphisms, and the space of positive definite Hermitian endomorphisms. 

\begin{lem}\label{lem: TinverseFormula}
    Assume $\cos(\htheta), \sin(\htheta)>0$ and $t \geq 0$.  Then $T_{\Phi,t}$ is invertible and
    \[
    \begin{aligned}
    T_{\Phi,t}^{-1}(A) &= \frac{1}{\cos(\htheta)} A - \frac{\frac{t}{2}\sin(\htheta)}{\cos(\htheta)\left(\cos(\htheta)+\frac{t}{2}\sin(\htheta)|\Phi|^2\right) }\{A, \Phi\Phi^{\dagger}\}\\
    &\quad + \frac{\frac{t^2}{2}\sin^2(\htheta)\langle A\Phi, \Phi \rangle }{\cos(\htheta)\left(\cos(\htheta)+\frac{t}{2}\sin(\htheta)|\Phi|^2\right)\left(\cos(\htheta)+t\sin(\htheta)|\Phi|^2\right)}\Phi \Phi^{\dagger}
    \end{aligned}
    \]
\end{lem}
The proof is an elementary calculation, so we omit the details.  What will be useful for our purposes is that we can rewrite the equation~\eqref{eq: dimReduceDHYM-MOC}.

\begin{lem}\label{lem: H1CurveLemSimplfiedPositiveCos}
Let $h_i$, $i=1,2$ be metrics satisfying~\eqref{eq: dimReduceDHYM-MOC} for $t\in[0,1]$, and suppose that $\cos(\htheta),\sin(\htheta)>0$.  Then ~\eqref{eq: dimReduceDHYM-MOC} is equivalent to the following equation, with $\alpha = \tan(\htheta)$
\begin{equation}
\begin{aligned}
i\Lambda F_1 &= \alpha I - \frac{1+t\alpha^2}{1+t\alpha |\Phi|^2}\Phi\Phi^{\dagger}+ \frac{t\alpha}{2}\nabla \Phi(\nabla \Phi)^{\dagger}\\
& - \frac{t^2\alpha^2}{2(2+t\alpha|\Phi|^2)}\{\Phi\Phi^{\dagger}, \nabla \Phi(\nabla \Phi)^{\dagger}\} + \frac{t^3\alpha^3 |\langle \nabla \Phi, \Phi\rangle|^2}{2(2+t\alpha |\Phi|^2)(1+t\alpha|\Phi|^2)}\Phi\Phi^{\dagger} \\
\end{aligned}
\end{equation}
\end{lem}
\begin{proof}
    This is a tedious exercise in algebra.  First rewrite the equation for $i\Lambda F_1$ from ~\eqref{eq: dimReduceDHYM-MOC} as
    \begin{equation}\label{eq: dimReduceMOCUsingToperator}
    T_{\Phi, t}(i\Lambda F_1) = \sin(\htheta) I - \cos(\htheta)\Phi\Phi^{\dagger} + \frac{t}{2}\sin(\htheta)\nabla \Phi (\nabla \Phi)^{\dagger}
    \end{equation}
    and then apply $T_{\Phi,t}^{-1}$ using the formula from Lemma~\ref{lem: TinverseFormula}.  
\end{proof}

The conclusion of Lemma~\ref{lem: H1CurveLemSimplfiedPositiveCos} will be useful later in the proof, but for now we need a more general result which applies for general values of $\cos(\htheta)$.

\begin{lem}\label{lem: H1CurvLem}
     Let $h_i$, $i=1,2$ be metrics satisfying~\eqref{eq: dimReduceDHYM-MOC} for $t\in[0,1]$.  Let $V$ be a vector in the fiber of $E_1$ orthogonal to ${\rm span}\{\Phi\}$ with respect to $h_1$. Then curvature $i\Lambda F_{1}$ of the metric $h_1$ satisfies
   
      \begin{equation}\label{eq: curvH_1PhiPhi}
 \langle i\Lambda F_{1}\Phi, \Phi \rangle =\frac{-\cos(\hat{\theta})|\Phi|^4 +\sin(\hat{\theta})\left( |\Phi|^2 + \frac{t}{2}|\langle \nabla\Phi, \Phi \rangle|^2 \right) }{\left( \cos(\hat{\theta})+\sin(\hat{\theta})t|\Phi|^2\right)},
    \end{equation}
     and
     \[
    \cos(\hat{\theta})\langle i\Lambda F_{1}V, V \rangle = \sin(\hat{\theta})\left(|V|^2+\frac{t}{2}|\langle V, \nabla\Phi \rangle|^2\right).
     \]
In particular, $\cos(\hat{\theta}) + t\sin(\hat{\theta})|\Phi|^2 \ne 0$.
\end{lem}
\begin{proof}
    First, note that since $\sin(\htheta)>0$, if $\cos(\htheta)=0$, and $\Phi=0$ somewhere, then ~\eqref{eq: dimReduceDHYM-MOC} cannot hold.  Thus we may assume that $\cos(\htheta)=0$ and $\Phi=0$ do not hold simultaneously. Fix a point $p\in X$ such that $\Phi(p)\ne 0$.  We apply $\Phi$ to the first equation in~\eqref{eq: dimReduceDHYM-MOC}, and get
    \begin{equation}\label{eq: curvH_1Phi}
    \begin{aligned}
        \cos(\hat{\theta})(i\Lambda F_{1}\Phi + |\Phi|^2\Phi) = \sin(\hat{\theta})\left( \Phi - \frac{t}{2}|\Phi|^2i\Lambda F_{1}\Phi -\frac{t}{2}\langle i\Lambda F_{h_{1}}\Phi, \Phi \rangle \Phi + \frac{t}{2}g^{p\bar{j}} \nabla_{p}\Phi \langle \Phi, \nabla_j\Phi \rangle\right)
    \end{aligned}
    \end{equation}
    Taking the inner product with $\Phi$ yields
    \begin{equation}\label{eq: curveH_1PhiPhi0}
   \left( \cos(\hat{\theta})+\sin(\hat{\theta})t|\Phi|^2\right) \langle i\Lambda F_{1}\Phi, \Phi \rangle =\sin(\hat{\theta})\left( |\Phi|^2 + \frac{t}{2}|\langle \nabla\Phi, \Phi \rangle|^2 \right) -\cos(\hat{\theta})|\Phi|^4.
    \end{equation}
    We now observe that, since $\Phi \ne 0$, then $\cos(\hat{\theta})+\sin(\hat{\theta})t|\Phi|^2 =0$ only if $\cos(\hat{\theta}) <0$.  On the other hand, if $\cos(\hat{\theta}) <0$, then the right hand side of~\eqref{eq: curveH_1PhiPhi0} is non-negative, and hence $\cos(\hat{\theta})+\sin(\hat{\theta})t|\Phi|^2  \ne 0$ unless $\cos(\hat{\theta})=0$ and $\Phi=0$, which, as we noted, is impossible. Therefore we can divide through to obtain the first result
    \[
 \langle i\Lambda F_{1}\Phi, \Phi \rangle =\frac{-\cos(\hat{\theta})|\Phi|^4 +\sin(\hat{\theta})\left( |\Phi|^2 + \frac{t}{2}|\langle \nabla\Phi, \Phi \rangle|^2 \right) }{\left( \cos(\hat{\theta})+\sin(\hat{\theta})t|\Phi|^2\right)}.
   \]
  
    Next, if $V$ is orthogonal to $\Phi$, then applying $V$ to the first equation in~\eqref{eq: dimReduceDHYM-MOC} yields
    \begin{equation}\label{eq: curveH1orthogPhi}
    \cos(\hat{\theta})i\Lambda F_{1}V = \sin(\hat{\theta})\left(V - \frac{t}{2}\langle i\Lambda F_{1}V, \Phi \rangle \Phi  +\frac{t}{2}g^{p\bar{j}}\langle V, \nabla_{j}\Phi \rangle\nabla_{p}\Phi\right)
    \end{equation}
    Taking the inner product with $V$ and using that $V$ and $\Phi$ are orthogonal yields
    \[
    \cos(\hat{\theta})\langle i\Lambda F_{1}V, V \rangle = \sin(\hat{\theta})\left(|V|^2+\frac{t}{2}|\langle V, \nabla\Phi \rangle|^2\right)
    \]
    which completes the proof.
\end{proof}

Next we consider the curvature properties of $E_2$.

\begin{lem}\label{lem: H2CurvLem}
    Let $h_i$, $i=1,2$ be metrics satisfying~\eqref{eq: dimReduceDHYM-MOC} for $t\in[0,1]$.  The curvature $i\Lambda F_{h_{2}}$ satisfies
    \[
     i\Lambda F_{2} =\left(\frac{ -\cos(\hat{\theta})\left(\frac{4\pi}{\sigma} - |\Phi|^2\right) +\sin(\hat{\theta})\left(1 + \frac{t}{2}|\nabla \Phi|^2\right)}{\left(\cos(\hat{\theta}) + \left(\frac{4\pi}{\sigma}-t|\Phi|^2\right)\sin(\hat{\theta})\right)} \right) 
    \]
    provided $\cos(\hat{\theta}) + \left(\frac{4\pi}{\sigma}-t|\Phi|^2\right)\sin(\hat{\theta}) \ne 0$.  Furthermore, if either $\cos(\hat{\theta})\geq 0$ or $\cos(\hat{\theta})+\frac{4\pi}{\sigma} \sin(\hat{\theta}) <0$, then
\[
\cos(\hat{\theta}) + \left(\frac{4\pi}{\sigma}-t|\Phi|^2\right)\sin(\hat{\theta}) \ne 0.
\]
\end{lem}
\begin{proof}
     Proceeding as in the proof of Lemma~\ref{lem: H1CurvLem} we compute
    \begin{equation}\label{eq: H2curv1}
    \begin{aligned}
    \left(\cos(\hat{\theta}) + \left(\frac{4\pi}{\sigma}-t|\Phi|^2\right)\sin(\hat{\theta})\right)i\Lambda F_{2} &= -\cos(\hat{\theta})\left(\frac{4\pi}{\sigma} - |\Phi|^2\right) \\
    &\quad +\sin(\hat{\theta})\left(1 + \frac{t}{2}|\nabla \Phi|^2 \right)
    \end{aligned}
    \end{equation}
    As in the case of Lemma~\ref{lem: H1CurvLem}, we consider the conditions under which the prefactor on the left-hand side can vanish.  First, if $\cos(\hat{\theta})\geq 0$ and the prefactor vanishes, then, since $\sin(\hat{\theta})> 0$ we must have $|\Phi|^2 \geq \frac{4\pi}{\sigma}$.  But this implies that the right hand side of~\eqref{eq: H2curv1} is strictly positive.  Thus, we conclude that
    \[
     \left(\cos(\hat{\theta}) + \left(\frac{4\pi}{\sigma}-t|\Phi|^2\right)\sin(\hat{\theta})\right) \ne 0 \quad \text{ if } \cos(\hat{\theta})\geq 0
    \]
    Similarly, if $\cos(\hat{\theta})+\frac{4\pi}{\sigma} \sin(\hat{\theta}) <0$, then clearly the prefactor can never vanish.
\end{proof}

\subsection{The calibration constraint}

In this section we examine the properties of solutions of the system~\eqref{eq: dimReduceDHYM-MOC}.

\begin{lem}\label{lem: simplifiedPos}
  Suppose $(h_1, h_2)$ solve the system~\eqref{eq: dimReduceDHYM}.  Then the calibration constraint~\eqref{eq: dimReduceDHYM-Pos} is equivalent to the following:
\begin{equation}\label{eq: simplifedPositivity}
\begin{aligned}
i\Lambda F_1 + \Phi\Phi^{\dagger} &>0\\
i\Lambda F_2 + \frac{4\pi}{\sigma} - |\Phi|^2 &>0
\end{aligned}
\end{equation}
\end{lem}
\begin{proof}
    This is straightforward.  We just explain the proof of the first conclusion, with the second being the same.  Since $\sin(\htheta)>0$ we can use equation~\eqref{eq: dimReduceDHYM} to write
    \[
    I_{E_1} - \frac{1}{2}\{i\Lambda F_1, \Phi\Phi^{\dagger}\} + \frac{1}{2}g^{-1}\nabla \Phi (\nabla \Phi)^{\dagger} = \frac{\cos(\htheta)}{\sin(\htheta)}(i\Lambda F_1 + \Phi\Phi^{\dagger})
    \]
    Plugging this into~\eqref{eq: dimReduceDHYM-Pos} yields that, in combination with ~\eqref{eq: dimReduceDHYM}, the calibration constraint is equivalent to
    \[
    \frac{1}{\sin(\htheta)} (i\Lambda F_1 + \Phi\Phi^{\dagger}) >0.
    \]
\end{proof}

We now analyze the solvability of the system~\eqref{eq: dimReduceDHYM} and~\eqref{eq: dimReduceDHYM-Pos} in the case that $\cos(\htheta)\leq 0$.  

\begin{lem}\label{lem: postivityConstraints}
    Suppose that a triple $(E_1,E_2,\Phi)$ admits a solution of the dimension reduced deformed Hermitian-Yang-Mills system~\eqref{eq: dimReduceDHYM}+~\eqref{eq: dimReduceDHYM-Pos} with $\cos(\htheta)\leq 0$.  Then $\{\Phi= 0\} = \emptyset$ and the following inequalities hold pointwise on $X$:
    \begin{equation}\label{eq: consequenceOfPositivity}
    \cos(\htheta) + |\Phi|^2 \sin(\htheta) >0, \quad \text{ and } \quad \cos(\htheta) + \left(\frac{4\pi}{\sigma}- |\Phi|^2\right)\sin(\htheta) >0.
    \end{equation}
    In particular, we have $\cos(\htheta) + \frac{2\pi}{\sigma} \sin(\htheta) >0$.
\end{lem}
\begin{proof}
    By Lemma~\ref{lem: simplifiedPos},  a solution of the system~\eqref{eq: dimReduceDHYM}+~\eqref{eq: dimReduceDHYM-Pos} satisfies
    \[
    \begin{aligned}
    i\Lambda F_1 + \Phi\Phi^\dagger >0\\
    i\Lambda F_2 + \frac{4\pi}{\sigma} - |\Phi|^2 >0
    \end{aligned}
    \]
    If there is a point $p\in X$ where $\Phi(p)= 0$ then~\eqref{eq: dimReduceDHYM}+~\eqref{eq: dimReduceDHYM-Pos} implies that, at $p$, there
    \[
    \cos(\htheta)i\Lambda F_1 = \sin(\htheta) (I_{E_1} + \frac{1}{2}\nabla \Phi (\nabla \Phi)^{\dagger}) \qquad \text{ and } \quad  i\Lambda F_1 >0
    \]
    But these equations are clearly inconsistent since $\sin(\htheta)>0$ and $\cos(\htheta) \leq 0$.  Thus we may assume $\Phi$ never vanishes. Fix $p\in X$ and apply the first equation to $\Phi(p)e_2$ for $e_2$ a unit vector in the fiber of $(E_2)_p$, to get
    \[
    \langle i\Lambda F_1 \Phi, \Phi \rangle + |\Phi|^4 >0
    \]
    On the other hand, by Lemma~\ref{lem: H1CurvLem}, and more specifically equation~\eqref{eq: curvH_1PhiPhi}, we have
    \[
    \begin{aligned}
    0 &< \langle i\Lambda F_1 \Phi, \Phi \rangle + |\Phi|^4 \\
    &= \frac{ -\cos(\htheta)|\Phi|^4 + \sin(\htheta)\left(|\Phi|^2 + \frac{1}{2} |\langle \nabla \Phi, \Phi\rangle|^2\right)}{\cos(\htheta) + \sin(\htheta)|\Phi|^2} + |\Phi|^4\\
    &= \frac{ \sin(\htheta)|\Phi|^6 + \sin(\htheta)\left(|\Phi|^2 + \frac{1}{2} |\langle \nabla \Phi, \Phi\rangle|^2\right)}{\cos(\htheta) + \sin(\htheta)|\Phi|^2}
    \end{aligned}
    \]
    and since $\sin(\htheta)>0$ we conclude that $\cos(\htheta) + \sin(\htheta)|\Phi|^2>0$. 

    Next we consider the equation for $F_2$.  By Lemma~\ref{lem: H2CurvLem} we have
    \[
    \begin{aligned}
    0 &< i\Lambda F_2+ \frac{4\pi}{\sigma} - |\Phi|^2 \\
    &= \frac{-\cos(\htheta)(\frac{4\pi}{\sigma}-|\Phi|^2) + \sin(\htheta)(1+\frac{1}{2}|\nabla \Phi|^2)}{\cos(\htheta) + (\frac{4\pi}{\sigma} -|\Phi|^2)\sin(\htheta)} + \frac{4\pi}{\sigma} - |\Phi|^2\\
    &= \frac{\sin(\htheta)(\frac{4\pi}{\sigma}-|\Phi|^2)^2 + \sin(\htheta)(1+\frac{1}{2}|\nabla \Phi|^2)}{\cos(\htheta) + (\frac{4\pi}{\sigma} -|\Phi|^2)\sin(\htheta)}
    \end{aligned}
    \]
    and hence $\cos(\htheta) + (\frac{4\pi}{\sigma} -|\Phi|^2)\sin(\htheta)>0$.  Finally, by adding these two results we conclude that
    \[
    \cos(\htheta) + \frac{2\pi}{\sigma} \sin(\htheta) >0.
    \]
\end{proof}

The next result classifies solutions of the system~\eqref{eq: dimReduceDHYM}+~\eqref{eq: dimReduceDHYM-Pos} in the case that $\cos(\htheta) \leq 0$.

\begin{prop}\label{prop: charSolCosNeg}
    A triple $(E_1, E_2, \Phi)$ solves the system~\eqref{eq: dimReduceDHYM}+~\eqref{eq: dimReduceDHYM-Pos} with $\cos(\htheta)\leq 0$ if and only if 
    \begin{enumerate}
        \item[(i)] $\cos(\htheta) + \frac{2\pi}{\sigma} \sin(\htheta) >0$,
        \item[(ii)] $E_1$ has rank $1$, and $\Phi \in H^{0}(X, {\rm Hom}(E_2, E_1))$ is never vanishing.  In particular, $E_2 \simeq E_1$, and
        \item[(iii)] we have
        \[
        i\Lambda F_1 = i\Lambda F_2 = \frac{\sin(\htheta) - \frac{2\pi}{\sigma} \cos(\htheta)}{\cos(\htheta) + \frac{2\pi}{\sigma} \sin(\htheta)}
        \]
        and $|\Phi|^2 \equiv \frac{2\pi}{\sigma} $.
    \end{enumerate}
\end{prop}
\begin{proof}
    Let us first prove necessity.  The necessity of $(i)$ is already proved in Lemma~\ref{lem: postivityConstraints}.  By Lemma~\ref{lem: postivityConstraints}, in particular the first equation of~\eqref{eq: consequenceOfPositivity}, we see that if $\cos(\htheta) \leq 0$, then $\Phi$ never vanishes.  Let us assume that $E_1$ has rank at least $2$.  Let $p \in X$, and $e_2$ be a unit vector in the fiber of $E_2$ at $p$.  Let $V$ be a vector in the fiber of $E_1$ orthogonal to $\Phi(e_2)$.  Applying Lemma~\ref{lem: H1CurvLem} with $t=1$ we have
    \[
    \cos(\htheta) \langle i\Lambda F_1 V, V \rangle = \sin(\htheta)\left( |V|^2 + \frac{1}{2} |\langle V, \nabla \Phi \rangle|^2\right).
    \]
    Since the right hand side is strictly positive we immediately obtain a contradiction if $\cos(\htheta)=0$.  If $\cos(\htheta)<0$, then we must have
    \[
    \langle i\Lambda F_1 V, V\rangle <0
    \]
    but this contradicts ~\eqref{eq: dimReduceDHYM-Pos}, by Lemma~\ref{lem: simplifiedPos}.  It follows that $E_1$ has rank $1$ and $\Phi \in {\rm Hom}(E_2, E_1)$ is non-vanishing.

    Now consider $|\Phi|^2$.  Let $x_* \in X$ be a point where $|\Phi|^2$ achieves its minimum, and let $x^*$ be a point where $|\Phi|^2$ achieves its maximum.  Since $\Phi$ is holomorphic, at the points $x_{*}$ and $x^*$ we have $\langle \nabla \Phi, \Phi \rangle =0$, but since $\Phi$ is non-vanishing and $E_1$ has rank $1$, this implies $\nabla \Phi =0$ at $x_*$ and $x^*$.  We compute at $x_*$
    \[
    0 \leq \Delta |\Phi|^2(x_*) = -\langle i\Lambda F_1 \Phi, \Phi \rangle +\langle i\Lambda F_2 \Phi, \Phi \rangle.
    \]
    We combine~\eqref{eq: curvH_1PhiPhi} and Lemma~\ref{lem: H2CurvLem} to get
    \[
    \begin{aligned}
    0 &\leq \frac{\cos(\htheta)|\Phi|^4 -\sin(\htheta)|\Phi|^2}{\cos(\htheta)+\sin(\htheta)|\Phi|^2} +\frac{\sin(\htheta)|\Phi|^2 -\cos(\htheta)|\Phi|^2\left(\frac{4\pi}{\sigma}-|\Phi|^2\right)}{\cos(\htheta) + (\frac{4\pi}{\sigma}-|\Phi|^2)\sin(\htheta)}\\
    &= \frac{|\Phi|^2(2|\Phi|^2-\frac{4\pi}{\sigma})}{(\cos(\htheta) + (\frac{4\pi}{\sigma}-|\Phi|^2)\sin(\htheta))(\cos(\htheta)+|\Phi|^2\sin(\htheta))}.
    \end{aligned}
    \]
    Since the denominator is positive by Lemma~\ref{lem: postivityConstraints} we see that $\inf_{X} |\Phi|^2 \geq \frac{2\pi}{\sigma}$.  On the other hand, the same argument applied at $x_*$ yields $\sup_{X}|\Phi|^2 \leq \frac{2\pi}{\sigma}$.  Thus we get $|\Phi|^2 \equiv \frac{2\pi}{\sigma}$, and $\nabla \Phi \equiv 0$.  Finally, we get
    \[
    i\Lambda F_2 = i\Lambda F_1 = \frac{\sin(\htheta) - \frac{2\pi}{\sigma} \cos(\htheta)}{\cos(\htheta) + \frac{2\pi}{\sigma} \sin(\htheta)}
    \]
    This establishes the necessity of $(i)-(iii)$.  The sufficiency of $(i)-(iii)$ is clear.
\end{proof}

\section{Stability}\label{sec: stability}

This section defines the notion of stability relevant for existence of solutions to the deformed Hermitian-Yang-Mills system in our setting.  One can either view stability in terms of $SU(2)$-equivariant coherent sheaves on $\cX= X \times \mathbb{P}^1$, or in terms of holomorphic triples $(E_1,E_2,\Phi)$ on the Riemann surface $X$. We shall present both points of view.

The section is organized as follows:
\begin{itemize}
\item Step 1: $Z$-stability from the point of view of equivariant coherent sheaves.
\\
\item Step 2: $Z$-stability from the point of view of holomorphic triples.
\\
\item Step 3: Classify stable bundles of vortex type with ${\rm Re}(Z_{\cX}(\cE))\leq 0$ and obtain practical simplifications of $Z$-stability. In particular, show that it suffices to check stability against a small class of quotients.
\\
\item Step 4: Study some consequences of stability.
\\
\item Step 5: Prove the necessity of stability for the deformed Vortex equations.
\end{itemize}

\subsection{Stability of $SU(2)$-equivariant bundles on $\cX=X\times \mathbb{P}^1$}

As in Section~\ref{sec: dimReduction} we consider the $SU(2)$-equivariant vector bundle $\mathcal{E} \rightarrow \cX$ given by a holomorphic triple $(E_1, E_2, \Phi)$ where $E_2$ has rank $1$.  Stability is determined by testing a suitable notion of central charge against quotient objects for $\mathcal{E}$.  There are several new features of the relevant stability condition.  The first, and most striking in relation to the notion of slope stability, is the need to test stability against torsion quotients.  The second feature is the natural appearance of a torsion class, which further refines the stability notion. 

Let ${\rm Coh}^{SU(2)}(\mathcal{X})$ denote the abelian category of $SU(2)$-equivariant coherent sheaves on $\mathcal{X}$.  We consider the central charge
\[
{\rm Coh}^{SU(2)}(\mathcal{X}) \ni \mathcal{F} \mapsto Z_{\cX}(\mathcal{F}) = -\int_{\cX}e^{-\frac{\sqrt{-1}}{2\pi}\widetilde{\omega}_{\sigma}}ch(\mathcal{F}).
\]
The central charge $Z_{\cX}$ has many nice properties inherited from the total Chern character $ch$; most crucially, it is additive on exact sequences.

Since it will be useful later, we compute a simple example of the central charge evaluated on a torsion sheaf.  Let $F \simeq \mathbb{P}^1$ be a fiber of the projection map $ p_1:\mathcal{X} : = X\times \mathbb{P}^1 \rightarrow X
$ and let $\iota: F \hookrightarrow \cX$ be the inclusion.  Let $Q \rightarrow F$ be a vector bundle.  We view $\cQ:= \iota_*Q$ as a torsion sheaf on $\mathcal{X}$.  The Chern character of $\cQ$ is given by
\[
ch(\cQ) = \iota_*ch(Q)
\]
and so we have 
\begin{equation}\label{eq: centralChargeTorsionSheaf}
Z_{\cX}(\mathcal{Q}) = -\int_{\mathbb{P}^1}e^{-\frac{\sqrt{-1}}{2\pi}\iota^*\widetilde{\omega}_{\sigma}} ch(Q).
\end{equation}
We define the restricted upper half-space to be
\begin{equation}\label{eq: restUpperHS}
\mathbb{H}:= \{ z \in \mathbb{C} : {\rm Im}(z)>0\} \cup \{z \in \mathbb{C}: {\rm Im}(z)=0, {\rm Re}(z)<0\}
\end{equation}
We can now give the definition of $Z$-stability.  First we give the definition of a {\em torsion class}.

\begin{defn}\label{defn: torsionClass}
    Define the torsion class $\mathcal{T}\subset {\rm Coh}^{SU(2)}(\cX)$ to be the set of $SU(2)$-equivariant sheaves $\mathcal{F}$ such that, either
    \begin{itemize}
        \item[(i)] $\mathcal{F}$ is torsion, or
        \item[(ii)]for every surjection $\mathcal{F} \twoheadrightarrow \mathcal{Q} \ne 0$ in $Coh^{SU(2)}(\mathcal{X})$ we have ${\rm Im}(Z_{\cX}(\mathcal{Q}))>0$.
    \end{itemize}
\end{defn}

\begin{rk}\label{rk: torsionTerminology}
The ``torsion class" $\mathcal{T}$ is precisely the torsion class appearing in the construction of the tilted heart of the Bridgeland stability condition on a projective surface constructed by Arcara-Bertram \cite{AB}. 
\end{rk}

\begin{rk}\label{rk: torsionClassQuotients}
$\mathcal{T}$ is trivially closed under quotients.  That is, if $\mathcal{F} \in \mathcal{T}$, and $\mathcal{F}\twoheadrightarrow \mathcal{Q}$, then it follows that $\mathcal{Q} \in \mathcal{T}$.
\end{rk}

The next lemma shows that $SU(2)$-equivariant torsion sheaves are in the torsion class $\mathcal{T}$.

\begin{lem}\label{lem: torsionIsTorsion}
    If $\mathtt{T} \in {\rm Coh}^{SU(2)}(\cX)$ is a torsion sheaf, then ${\rm Im}(Z_{\cX}(\mathtt{T}))>0$.
\end{lem}
\begin{proof}
Any $SU(2)$-equivariant torsion sheaf on $\mathcal{X}$ has reduced support equal to a disjoint union of fibers of $p_1$.  Since the central charge is linear with respect to direct sums, it suffices to check the case where the support is a single fiber.  Let $\mathcal{I}$ denote the ideal sheaf of the fiber $F$ and let $\iota: F_ \hookrightarrow \cX$ be the inclusion.  Consider the $\mathcal{I}$-adic filtration
\[
\mathtt{T}\supset \mathcal{I}\mathtt{T} \supset \mathcal{I}^2\mathtt{T} \cdots \supset \mathcal{I}^N\mathtt{T} =0
\]
Let $\mathtt{T}_k = \mathcal{I}^k\mathtt{T}/\mathcal{I}^{k+1}\mathtt{T}$ denote the successive quotients.  Then $\mathtt{T}_k$  is a vector bundle $Q\rightarrow F$.  For this example we have
\begin{equation}\label{eq: centralChargeTorsionExpanded}
 Z_{\cX}(\iota_*Q) =  (i \frac{\sigma}{2\pi}\int_{\mathbb{P}^1} \rank(Q)\omega_{FS} -\int_{\mathbb{P}^1}
ch_1(Q))
\end{equation}
hence ${\rm Im}(Z_{\cX}(\mathtt{T}))>0$.  The result follows from the additivity of the central charge.
\end{proof}

\begin{defn}
   For any sheaf $\cF \in {\rm Coh}^{SU(2)}(\cX)$ with $Z_{\cX}(\cF) \in \mathbb{C}^*$, we define
    \[
   \Theta(\cF) = {\rm Arg}(Z_{\cX}(\cF))
    \]
    where ${\rm Arg} \in (-\pi, \pi]$ denotes the principal branch of the argument.  In particular, we have
    \[
    \Theta: \mathcal{T} \rightarrow (0, \pi]
    \]
\end{defn}

\begin{defn}\label{def: stabOnXxP}
    Let $\mathcal{E}\rightarrow \mathcal{X}$ be an $SU(2)$-equivariant vector bundle with $Z_{\mathcal{X}}(\mathcal{E})\in \mathbb{H}$.  We say that $\mathcal{E}$ is $Z$-stable if
    \begin{itemize}
    \item[(i)] $\mathcal{E} \in \mathcal{T}$, and
    \item[(ii)] for every exact sequence 
    $
    0\rightarrow \cS\rightarrow \cE \rightarrow \cQ \rightarrow 0
    $
    in ${\rm Coh}^{SU(2)}(\mathcal{X})$ with $\cQ \ne 0$ and $Z_{\cX}(\cS) \in \mathbb{H}$ we have
    \[
    \Theta(\mathcal{Q}) > \Theta(\mathcal{E}) \quad \text{ or, equivalently }\quad \Theta(\cS) < \Theta(\cE).
    \]  
    \end{itemize}
\end{defn}

\begin{rk}
    Note that since $\cE \in \mathcal{T}$, we have both $Z_{\cX}(\cS), Z_{\cX}(\cQ) \in \mathbb{H}$.
\end{rk}

The notion of $Z$-stability defined in Definition~\ref{def: stabOnXxP} differs from the notion of $Z$-stability considered by Dervan-McCarthy-Sektnan \cite{DMS} and Keller-Scarpa \cite{KellerScarpa}.  The notion of stability considered in {\em loc. cit.} is tailored to the  ``large volume limit", while $Z$-stability in the sense of Definition~\ref{def: stabOnXxP} is necessarily not perturbative.  The main differences are: $(1)$ the appearance of the torsion class $\mathcal{T}$, and $(2)$ the restriction to exact sequences 
\begin{equation}\label{eq: stabDiscussionExactSeq}
    0\rightarrow \cS\rightarrow \cE \rightarrow \cQ \rightarrow 0
    \end{equation}
    in ${\rm Coh}^{SU(2)}(\mathcal{X})$ with $Z_{\cX}(\cS), Z_{\cX}(\cQ) \in \mathbb{H}$. These restrictions are not merely cosmetic, and are closely related to the notion of Bridgeland stability; see Section~\ref{sec: comments}.  Indeed, if one allows {\em all} exact sequences in ${\rm Coh}^{SU(2)}(\mathcal{X})$, then it is not hard to check that there are no stable bundles $\cE$ of vortex type with $\Theta(\cE) \in [\frac{\pi}{2}, \pi]$.

\begin{rk}
In Lemma~\ref{lem: StabreplaceTorsion} we show that the notion of $Z$-stability for vortex-type triples can be further refined to consider only exact sequences as in~\eqref{eq: stabDiscussionExactSeq} for which $\cS,\cQ \in \mathcal{T}$.  This refinement is essential for establishing the connection between $Z$-stability and  Bridgeland stability.  One could therefore take this weaker notion of stability as the definition in general. By Lemma~\ref{lem: StabreplaceTorsion}, Theorem~\ref{thm: mainTheorem}  continues to hold under this weaker notion of stability.
\end{rk}

The remainder of this section examines and reformulates $Z$-stability into an equivalent form which is more adapted to our analytic estimates in Section~\ref{sec: Estimates}. It is helpful to introduce some terminology for discussing stability.

\begin{defn}\label{defn: admissible}
    If $\cE \in \mathcal{T}$ and $\cE \twoheadrightarrow \cQ \ne 0$, we say that $\cQ$ is admissible if the exact sequence  $
    0\rightarrow \cS\rightarrow \cE \rightarrow \cQ \rightarrow 0
    $
    in ${\rm Coh}^{SU(2)}(\mathcal{X})$ has $Z_{\cX}(\cS) \in \mathbb{H}$.
\end{defn}

\begin{rk}
    If $\cE\twoheadrightarrow \cQ$ and  ${\rm Im}(Z_{\cX}(\cQ))>0$, then $\Theta(\cQ) >\Theta(\cE)$ is equivalent to ${\rm Im}\left(\frac{Z_{\cX}(\cQ)}{Z_{\cX}(\cE)}\right)>0$, and it is also equivalent to ${\rm Im}\left(\frac{Z_{\cX}(\cS)}{Z_{\cX}(\cE)}\right)<0$.
\end{rk}

\subsection{$Z$-stable triples}
We now discuss the notion of $Z$-stability from the point of view of the holomorphic triple $\mathcal{E}= (E_1,E_2,\Phi)$ over $X$.  This formulation of stability in terms of triples will be indispensable in the analysis to follow.

Recall the following definition, after Bradlow-Garcia-Prada \cite{BG-P}:

\begin{defn}
    Let $(E_1,E_2,\Phi)$ be a holomorphic triple over $X$.  A holomorphic triple $(E_1',E_2',\Phi')$ is said to be a subtriple of $(E_1,E_2,\Phi)$ if 
    \begin{itemize}
        \item[(i)] $E_i'$ is a coherent subsheaf of $E_i$ for $i=1,2$, and
        \item[(ii)] the following diagram commutes
    \[
\begin{tikzcd}
   E_2' \arrow[d, " "] \arrow[r, "\Phi'"] & E_1' \arrow[d, " "] \\
   E_2 \arrow[r, "\Phi "] & E_1 
\end{tikzcd}
\]
\end{itemize}
We say that the holomorphic subtriple $(E_1',E_2',\Phi')$ is saturated if $E_i'$ is a sub-bundle of $E_i$ for $i=1,2$.
\end{defn}

By \cite{BG-P} holomorphic subtriples are in one-to-one correspondence with $SU(2)$-invariant coherent sub-sheaves of $\mathcal{E}\rightarrow \mathcal{X}$.  In our present setting, with ${\rm rk}(E_2)=1$, there are two types of subtriples, depending on whether $\rank(E_2') =0, 1$.  It will be convenient to distinguish the two.
\begin{defn}\label{defn: subObTypeA/B}
    Let $\mathcal{E}= (E_1, E_2, \Phi)$ be a holomorphic triple with ${\rm rk}(E_2)=1$.  We say that a subtriple $\mathcal{S} = (E_1', E_2', \Phi')$ is:
    \begin{enumerate}
        \item[(i)] of type A if ${\rm rk}(E_2')=1$.
        \item[(ii)] of type B if ${\rm rk}(E_2')=0$
    \end{enumerate}
\end{defn}
If $\mathcal{E}= (E_1,E_2, \Phi)$ is a holomorphic triple then the calculation in Section~\ref{sec: dimReduction} yields
\begin{equation}\label{eq: TripleCentralCharge}
Z_{\mathcal{X}}(\cE) =  \frac{\sigma}{2\pi} \left( \left[\frac{V_{X}}{2\pi}\left(\rank(E_1)+\rank(E_2)\right) -\frac{4\pi }{\sigma} \deg(E_2)\right] +i\bigg[\deg(E_1)+ \deg(E_2) +\frac{2 V_{X}}{\sigma}{\rm rk}(E_2)\bigg]\right)
\end{equation}
In particular, if $\mathcal{S} = (E_1', 0, 0)$ is a subtriple of $\cE$ of type B, then
\begin{equation}\label{eq: subTripleCentralChargeTypeB}
Z_{\mathcal{X}}(\cS) = \frac{\sigma}{2\pi}\left( \frac{V_X}{2\pi} {\rm rk}(E_1') +i\deg(E_1')\right)
\end{equation}
With these formulas we can translate stability in the sense of Definition~\ref{def: stabOnXxP} into a definition of stability for triples. 

\begin{defn}\label{defn: stabOfTriple}
We say that a holomorphic triple $\mathcal{E}= (E_1, E_2, \Phi)$ with ${\rm Im}(Z_{\cX}(\cE)) > 0$ is $Z$-stable if for every non-zero proper holomorphic subtriple $\mathcal{S} = (E_1', E_2',\Phi')$ we have:
\begin{enumerate}
\item[(i)] $
{\rm Im}\left(Z_{\mathcal{X}}(\cS)\right) < {\rm Im}\left(Z_{\mathcal{X}}(\cE)\right),
$ and,
\item[(ii)] If $Z_{\cX}(\cS) \in \mathbb{H}$, then ${\rm Im}\left(\frac{Z_{\cX}(\cS)}{Z_{\cX}(\cE)}\right)<0$
\end{enumerate}
\end{defn}

The following trivial lemma proves the equivalence of Definition~\ref{def: stabOnXxP} and Definition~\ref{defn: stabOfTriple}. 

\begin{lem}
   Let $\mathcal{E}\rightarrow \mathcal{X}$ be an $SU(2)$-equivariant vector bundle with ${\rm Im}(Z_{\mathcal{X}}(\mathcal{E}))>0$.  Then $\cE$ is $Z$-stable in the sense of Definition~\ref{def: stabOnXxP} if and only if $\cE$ is stable in the sense of Definition~\ref{defn: stabOfTriple}.
\end{lem}
\begin{proof}
    By additivity of $Z_{\cX}$ on exact sequences, property $(i)$ is equivalent to the statement that $\cE \in \mathcal{T}$.  Similarly, property $(ii)$ is equivalent to the statement that if $\cE\twoheadrightarrow\cE/\cS$ is admissible, then $\Theta(\cE/\cS)> \Theta(\cE)$.
\end{proof}

\subsection{Classification and reduction results}
We prove a lemma classifying stable vortex bundles with large phase.  This can be viewed as the algebro-geometric version of Proposition~\ref{prop: charSolCosNeg}.

\begin{lem}\label{lem: stabForTorsion}
    Suppose $\mathcal{E}\in \mathcal{T}$ is an $SU(2)$ equivariant vector bundle determined by a triple $(E_1, E_2, \Phi)$ with $\rank(E_2)=1$. 
    \begin{itemize}
    \item[(i)] Suppose that $\Theta(\cE) \in(0,\frac{\pi}{2})$.  If $F\simeq \mathbb{P}^1$ is a fiber of $X\times \mathbb{P}^1\rightarrow X$, $Q\rightarrow F$ is a vector bundle, $\cQ=\iota_{*}Q$ and $\iota^*\mathcal{E}\rightarrow \mathcal{Q}\rightarrow 0$, then
    \[
    \Theta(\mathcal{E}) < \frac{\pi}{2} \leq  \Theta(\mathcal{Q}),
    \]
    \item[(ii)]If $\cE$ is stable with $\Theta(\cE) \in [\frac{\pi}{2}, \pi]$, then $\cE$ has rank $2$, and $E_1\simeq E_2$ via the map $\Phi \in H^{0}(X, {\rm Hom}(E_2, E_1))$, which  is nowhere vanishing.  Furthermore, $\Theta(\cE) \in [\frac{\pi}{2}, \frac{\pi}{2} + \arctan(\frac{2\pi}{\sigma}))$.
    \end{itemize}
\end{lem}
\begin{proof}
The main point is to prove $(ii)$, since $(i)$ is trivial. $\mathcal{E}$ is determined by the triple $(E_1,E_2, \Phi)$.  Choose any point $x\in X$, and let $F= p_1^{-1}(x)\simeq \mathbb{P}^1$.  If $\Phi(x) \ne 0$ then
\begin{equation}\label{eq: torsionQuotientPhiNeZero}
\iota^*\mathcal{E} \simeq \mathcal{O}_{\mathbb{P}^1}^{\oplus {\rm rk}(E_1)-1} \oplus \mathcal{O}_{\mathbb{P}^1}(1)^{\oplus 2}.
\end{equation}
If instead $\Phi(x)=0$ then
\begin{equation}
\label{eq: torsionQuotientPhiequalZero}
\iota^*\mathcal{E} \simeq \mathcal{O}_{\mathbb{P}^1}^{\oplus {\rm rk}(E_1)} \oplus \mathcal{O}_{\mathbb{P}^1}(2).
\end{equation}
Therefore if either ${\rm rk}(E_1) >1$ or ${\rm rk}(E_1)=1$ and $x\in \{\Phi=0\}$, then we have a surjection 
\[
\iota^*\mathcal{E} \rightarrow \mathcal{O}_{\mathbb{P}^1} \rightarrow 0.
\]
This yields an exact sequence
\begin{equation}\label{eq: classStableLargePhaseSeq1}
0\rightarrow \cS \rightarrow \mathcal{E}\rightarrow \iota_*\mathcal{O}_{F}\rightarrow 0
\end{equation}
If we knew that $Z_{\cX}(\cS) \in \mathbb{H}$ then we could test the stability of $\cE$ against this exact sequence.  Let us assume $Z_{\cX}(\cS) \in \mathbb{H}$, and we will return to this point at the end of the proof. We compute
\[
Z_{\cX}(\iota_{*}\mathcal{O}_{\mathbb{P}^1}) = \sqrt{-1}\frac{\sigma}{2\pi}
\]
(e.g. by ~\eqref{eq: centralChargeTorsionExpanded}), we get $\Theta(\iota_*\mathcal{O}_{\mathbb{P}^1}) = \frac{\pi}{2}$.  It is straightforward to check, that any other $SU(2)$-equivariant quotient arising from the options~\eqref{eq: torsionQuotientPhiNeZero} and~\eqref{eq: torsionQuotientPhiequalZero} has phase at least $\frac{\pi}{2}$. This finishes the proof of $(i)$.

Proceeding to $(ii)$, we see that if $\mathcal{E}$ is stable and $\Theta(\cE)\in [\frac{\pi}{2}, \pi]$, then necessarily $E_1$ and $E_2$ have  rank $1$, and $\Phi$ is non-vanishing.  In this case we have $\iota^*\cE\simeq \mathcal{O}_{\mathbb{P}^1}(1)^{\oplus 2}$. Now
\[
Z_{\cX}(\iota_{*}\iota^*\mathcal{E}) = \sqrt{-1}\frac{\sigma}{\pi} -2
\]
and so
\[
{\rm Im}\left(\frac{Z_{\cX}(\iota_*\iota^*\cE)}{Z_{\cX}(\mathcal{E})}\right) \propto \frac{\sigma}{2\pi}\cos(\htheta) +\sin(\htheta).
\]
Consider the exact sequence
\begin{equation}\label{eq: classStableLargePhaseSeq2}
0\rightarrow \cS \rightarrow \cE \rightarrow \iota_*\iota^*\cE \rightarrow 0
\end{equation}
If $Z_{\cX}(\cS) \in \mathbb{H}$, then we conclude that the stability of $\mathcal{E}$  implies $\cos(\htheta) +\frac{2\pi}{\sigma}\sin(\htheta) >0$.

We are therefore reduced to proving that, if $\Theta(\cE) \in [\frac{\pi}{2},\pi]$, then in the exact sequences~\eqref{eq: classStableLargePhaseSeq1} and~\eqref{eq: classStableLargePhaseSeq2} we have $Z_{\cX}(\cS)\in \mathbb{H}$.  From~\eqref{eq: sec3CentralCharge} we see that we must have $\deg(E_2)>0$.  We claim that $\deg(E_1)>0$ as well.  First assume that $\Phi \not\equiv 0$.  Let $L= {\rm sat}(\Phi(E_2))$ be the saturation of $\Phi(E_2) \subset E_1$.  We have $L_{\Phi} \simeq E_2(D)$ for an effective divisor $D\subset X$ corresponding to the zeros of $\Phi$.  Consider the $SU(2)$-equivariant subbundle $\cS_{\Phi}\subset \cE$ given by the holomorphic triple $(L,E_2,\Phi)$. If $\mathcal{S}_{\Phi} \ne \cE$ then since $\cE \in \mathcal{T}$ we have
\[
{\rm deg}(E_1/L_{\Phi})>0
\]
but then we have
\[
\deg(E_1) = \deg(E_1/L_{\Phi}) + \deg(L_{\Phi}) >\deg(E_2)+\deg(D)
\]
If $E_1/L_{\Phi}=0$, then $\deg(E_1)= \deg(E_2)+\deg(D)>0$.  Finally, if $\Phi \equiv 0$ then $E_2\otimes \mathcal{O}_{\mathbb{P}^1}(2)$ is a subbundle of $\cE$ and we again conclude from $\cE \in \mathcal{T}$ that $\deg(E_1)>0$.

Now we can prove $Z_{\cX}(\cS) \in \mathbb{H}$ in the exact sequences~\eqref{eq: classStableLargePhaseSeq1} and~\eqref{eq: classStableLargePhaseSeq2}.  In both cases, the additivity of the central charge on exact sequences gives
\[
{\rm Im}\left(Z_{\cX}(\cS)\right) =  \frac{\sigma}{2\pi}\left(\deg(E_1)+\deg(E_2) -r+ \frac{2V_{X}}{\sigma}\right)
\]
where $r\in\{1,2\}$. Since $\deg(E_1),\deg(E_2) \in \mathbb{Z}_{>0}$ we conclude that ${\rm Im}Z_{\cX}(\cS)>0$ as desired. 
\end{proof}

Next we prove a lemma which says that, to check stability in the sense of Definition~\ref{def: stabOnXxP} it suffices to check a small subset of relatively simple coherent sheaves.

\begin{prop}\label{prop: simplifiedStability}
Let $\mathcal{E}$ be an $SU(2)$-equivariant vector bundle defined by a holomorphic triple $(E_1, E_2,\Phi)$ with $\rank(E_2)=1$, and ${\rm Im}(Z_{\cX}(\cE))>0$. Then $\mathcal{E}$ is Z-stable in the sense of Definition~\ref{def: stabOnXxP} if and only if the following hold:
\begin{itemize}
    \item[(i)] If $F\simeq \mathbb{P}^1$ is a fiber of $X\times \mathbb{P}^1\rightarrow X$, $Q\rightarrow F$ is a vector bundle, $\cQ= \iota_{*}Q$ and $\iota^*\mathcal{E}\rightarrow \mathcal{Q}\rightarrow 0$ is admissible then
    \[
    \Theta(\mathcal{E}) < \Theta(\mathcal{Q}),
    \]
    and,
    \item[(ii)] For every $SU(2)$-equivariant quotient bundle $\mathcal{E}\twoheadrightarrow \mathcal{Q} \ne 0$, we have ${\rm Im}(Z_{\cX}(\cQ))>0$ and, if $\cQ$ is admissible, then
    \[
    \Theta(\cE) < \Theta(\mathcal{Q}).
    \]
\end{itemize}
\end{prop}
\begin{proof}
Clearly stability in the sense of Definition~\ref{def: stabOnXxP} implies $(i)$ and $(ii)$.  For the converse, let $\mathcal{Q}$ be any non-zero $SU(2)$-equivariant coherent sheaf and suppose we have
\begin{equation}\label{eq: simpleStabexSeq1}
0\rightarrow \cS \rightarrow \cE\rightarrow \cQ \rightarrow 0
\end{equation}
We have an exact sequence
\[
0 \rightarrow \mathtt{T}\rightarrow \mathcal{Q} \rightarrow \mathcal{V} \rightarrow 0
\]
where $\mathtt{T}$ is an $SU(2)$-equivariant torsion sheaf and $\mathcal{V}$ is an equivariant torsion-free sheaf.  Since the singular set of $\mathcal{V}$ consists of points in $\cX$, but $\mathcal{V}$ is $SU(2)$-equivariant, this implies that $\mathcal{V}$ is a vector bundle.  Thus, by $(ii)$ we have ${\rm Im}\left(Z_{\cX}(\mathcal{V})\right) >0$. Since $\mathtt{T}$ is torsion $\mathtt{T}\in \mathcal{T}$ by Lemma~\ref{lem: torsionIsTorsion}, and so ${\rm Im}(Z_{\cX}(\cQ))>0$.  Hence $\cE \in \mathcal{T}$.  

From now on assume that the kernel $\cS$ in~\eqref{eq: simpleStabexSeq1} has $Z_{\cX}(\cS) \in \mathbb{H}$, so $\cQ$ is admissible. Let $\cS'= {\rm Ker}(\cE\twoheadrightarrow \mathcal{V})$.  Then
\[
Z_{\cX}(\cS')= Z_{\cX}(\cS) +Z_{\cX}(\mathtt{T})
\]
and using again the $\mathtt{T}\in \mathcal{T}$ by Lemma~\ref{lem: torsionIsTorsion} we conclude that $\cE\twoheadrightarrow \mathcal{V}$ is admissible. Therefore, by property $(ii)$ of the proposition we have
\[
\Theta(\cE) < \Theta(\mathcal{V}),
\]
which in particular yields that ${\rm Im}(Z_{\cX}(\mathcal{V}))>0$.
If $\mathtt{T}=0$ we are done, so we assume this is not the case.  By $SU(2)$-equivariance, the reduced support of $\mathtt{T}$ is a disjoint union of $\mathbb{P}^1$ fibers  $F_\beta\subset \mathcal{X}$.  Write
$
\mathtt{T}= \bigoplus_{\beta=1}^n \mathtt{T}_{(\beta)}
$
where the reduced support of $T_{(\beta)}$ is a single fiber $F_\beta$.  Since 
$
Z_{\cX}(\mathtt{T}) = \sum_{\beta=1}^{n} Z(\mathtt{T}_{(\beta)})
$
it suffices to treat each fiber separately.  Thus, for simplicity we assume $n=1$, and let $x\in X$ and $\mathcal{I}$ denote the ideal sheaf of the fiber $F_x = p_1^{-1}(x)$ and let $\iota: F_{x} \hookrightarrow \cX$ be the inclusion.  We consider the $\mathcal{I}$-adic filtration
\[
\mathtt{T}\supset \mathcal{I}\mathtt{T} \supset \mathcal{I}^2\mathtt{T} \cdots \supset \mathcal{I}^N\mathtt{T} =0
\]
Let $\mathtt{T}_k = \mathcal{I}^k\mathtt{T}/\mathcal{I}^{k+1}\mathtt{T}$, and note that $\mathtt{T}_0 = \iota^*\mathtt{T}$.  Since $F_{x}$ is a fiber we have $\mathcal{I}^k/\mathcal{I}^{k+1} = \mathcal{O}_{F_{x}}$, and hence
\[
\mathtt{T}_0 \simeq \mathtt{T}_0 \otimes_{\mathcal{O}_{F_{x}}} \mathcal{I}^k/\mathcal{I}^{k+1} \rightarrow \mathtt{T}_k \rightarrow 0.
\]
Since tensor product is right exact  \cite[\href{https://stacks.math.columbia.edu/tag/00M0}{Tag 00M0}]{stacks-project}, tensoring by $\mathcal{O}_{\mathcal{X}}/\mathcal{I}$ yields the sequence
\[
\iota^*\cE\rightarrow \iota^*\mathcal{Q}\rightarrow 0
\]
On the other hand, since $\mathcal{V}$ is a vector bundle, torsion vanishing \cite[\href{https://stacks.math.columbia.edu/tag/00M5}{Tag 00M5}]{stacks-project} yields 
\[
0\rightarrow \iota^*\mathtt{T} \rightarrow \iota^*\mathcal{Q} \rightarrow \iota^*\mathcal{V} \rightarrow 0
\]
We now consider two cases:

\smallskip\noindent {\bf Case 1}: Assume that $\Theta(\cE) \in (0, \frac{\pi}{2})$.  We have
\[
0\rightarrow \mathcal{O}_{\mathbb{P}^1}^{\oplus r} \rightarrow \iota^*\mathcal{E} \rightarrow \mathcal{O}_{\mathbb{P}^1}(2) \rightarrow 0 
\]
where $r \geq 1$.  Since we have an $SU(2)$-equivariant surjection $\iota^*\cE \rightarrow \iota^*\cQ\rightarrow 0$, we obtain
\[
0\rightarrow \mathcal{O}_{\mathbb{P}^1}^{\oplus r'} \rightarrow \iota^*\mathcal{Q} \rightarrow \mathcal{O}_{\mathbb{P}^1}(2)^{\oplus a} \rightarrow 0 
\]
where $r' \leq r$ and $ a\in \{0,1\}$, and by convention $\mathcal{O}_{\mathbb{P}^1}(2)^{\oplus 0}=0$.  Since $\mathtt{T}_0$ is an equivariant subbundle of $\iota^*\cQ$ we obtain that $\mathtt{T}_0$ sits in an exact sequence
\[
0\rightarrow \mathcal{O}_{\mathbb{P}^1}^{\oplus r_0} \rightarrow \mathtt{T}_0\rightarrow \mathcal{O}_{\mathbb{P}^1}(2)^{\oplus a_0} \rightarrow 0 
\]
where $a_0\in \{ 0,1\}$, and $r_0 \leq r'$.  Thus we have
\[
ch(\mathtt{T}_0) = r_0ch(\mathcal{O}_{\mathbb{P}^1}) + a_0ch(\mathcal{O}_{\mathbb{P}^1}(2))
\]
Furthermore, thanks to the equivariant surjection $\mathtt{T}_0 \twoheadrightarrow \mathtt{T}_k$, the same is true of $\mathcal{T}_k$; namely $\mathtt{T}_k$ sits in an exact sequence
\[
0\rightarrow \mathcal{O}_{\mathbb{P}^1}^{\oplus r_k} \rightarrow \mathtt{T}_k\rightarrow \mathcal{O}_{\mathbb{P}^1}(2)^{\oplus a_k} \rightarrow 0 
\]
where $a_k\in\{ 0,1\}$, and $r_k \leq r_0$.  Since
\[
ch(\mathtt{T}) = \sum_{k=0}^{N-1}ch(\mathtt{T}_k) = m_0ch(\mathcal{O}_{\mathbb{P}^1}) + m_2ch(\mathcal{O}_{\mathbb{P}^1}(2))
\]
for $m_0,m_2 \in \mathbb{Z}_{\geq 0}$ not both zero we obtain
\[
\Theta(\mathtt{T}) \geq  \frac{\pi}{2}
\]
and the result follows.

\bigskip
{\bf Case 2:}  Assume that $\Theta(\cE) \in[\frac{\pi}{2}, \pi]$.  We have already shown that $\cE \in \mathcal{T}$. Recall that the conclusion of Lemma~\ref{lem: stabForTorsion} part $(ii)$ followed from checking stability against the admissible quotients
\[
\begin{aligned}
    &\cE \twoheadrightarrow \iota_{*}\mathcal{O}_{F}\\
    &\cE \twoheadrightarrow \iota_*\iota^*\cE
    \end{aligned}
\]
where $F\simeq \mathbb{P}^1$ is a fiber.  These are precisely the types of objects permitted by $(ii)$ and so we conclude again that $\rank(\cE)=2$, $E_1 \simeq E_2 =:L$ and $\{\Phi=0\} = \emptyset$ and, crucially $\Theta(\cE) \in [\frac{\pi}{2}, \frac{\pi}{2} + \arctan(\frac{2\pi}{\sigma}))$.  The conclusion now follows from Lemma~\ref{lem: stabForVeryRestr} below.
\end{proof}

\begin{lem}\label{lem: stabForVeryRestr}
Suppose that $\cE=(E_1,E_2,\Phi)$ is a vortex type triple with $\rank(E_1)=1$, $E_1 \simeq E_2$, $\{\Phi=0\} = \emptyset$ and $\Theta(\cE) \in [\frac{\pi}{2}, \frac{\pi}{2} + \arctan(\frac{2\pi}{\sigma}))$.  Then $\cE$ is $Z$-stable.
\end{lem}
\begin{proof}
Let $L=E_2\simeq E_1$. In this restrictive setting, the only equivariant coherent sub-sheaves of $\cE$ of type B are given by a triple $(L(-D), 0,0)$ where $D$ is an effective divisor.  The sub-sheaves of type A correspond to holomorphic triples $(L(-D_1), L(-D_2), \Phi')$ where $D_1, D_2$ are effective divisors, $D_2 \geq D_1$ and $\Phi'$ is the defining section of $D_2-D_1$.  We check each case separately.
\smallskip

{\bf Case 1:} Sub-sheaves of type B
\smallskip

\noindent In this case $\mathcal{S}=(L(-D),0,0)$ with $D$ effective so by~\eqref{eq: TripleCentralCharge} we have
\[
Z_{\cX}(\cS) = \frac{\sigma}{2\pi}\left(\frac{V_{X}}{2\pi} + i (\deg(E_1) -\deg(D)\right)
\]
If $Z_{\cX}(\cS) \in \mathbb{H}$ we get $\deg(D) \leq \deg(E_1)$. It suffices to show that
\[
\cos(\Theta(\cE))(\deg(E_1)-\deg(D)) -\frac{V_{X}}{2\pi}\sin(\Theta(\cE)) \leq 0 
\]
but this is clear since $\cos(\Theta(\cE)) \leq 0$ and $\deg(D) \leq \deg(E_1)$.
\smallskip

{\bf Case 2:} Sub-sheaves of type A:
\smallskip

\noindent In this case $\mathcal{S}$ is induced by $(L(-D_1), L(-D_2), \Phi')$ where $D_1, D_2$ are effective divisors, $D_2 \geq D_1$ and $\Phi'$ is the defining section of $D_2-D_1$. Since $\mathcal{S}$ is proper we must have $D_2>0$.  For simplicity let $L=E_1=E_2$. To ease notation let us denote $\ell = \deg(L)$ and $d_i= \deg(D_i)$ and $\htheta =\Theta(\cE)$. In this case, by~\eqref{eq: TripleCentralCharge} we have
\[
Z_{\cX}(\cS) = \frac{\sigma}{2\pi}\left(\frac{V_{X}}{2\pi}\cdot 2 -\frac{4\pi}{\sigma}(\ell-d_2) +i\left(2\ell-d_1-d_2 + \frac{2V_{X}}{\sigma}\right)\right).
\]
It suffices to show that
\begin{equation}\label{eq: typeAsimplestabSuff}
\cos(\htheta)(d_1+d_2) +\frac{4\pi}{\sigma}d_2\sin(\htheta)>0.
\end{equation}
Since $\cos(\Theta(\cE)) \leq 0$, and $d_1 \leq d_2$ we get
\[
\cos(\htheta)(d_1+d_2)  \geq 2\cos(\htheta)d_2
\]
which yields
\[
\cos(\htheta)(d_1+d_2) +\frac{4\pi}{\sigma}d_2\sin(\htheta) \geq 2d_2(\cos(\htheta)+\frac{2\pi}{\sigma}\sin(\htheta))>0.
\]
where we used the key fact that $\htheta \in [\frac{\pi}{2}, \frac{\pi}{2} + \arctan(\frac{2\pi}{\sigma}))$.
\end{proof}

As a consequence of Lemma~\ref{lem: stabForTorsion} and Proposition~\ref{prop: charSolCosNeg} we obtain the equivalence between $Z$-stability and existence of solutions in the ``hypercritical" regime.

\begin{cor}\label{cor: classifySolRealNeg}
    Suppose $\mathcal{E} \rightarrow \mathcal{X}$ has ${\rm Im}(Z_{\mathcal{X}}(E))>0$ and ${\rm Re}(Z_{\mathcal{X}}(\mathcal{E}))\leq 0$.  Then $\mathcal{E}$ is $Z$-stable in the sense of Definition~\ref{def: stabOnXxP} if and only if $\mathcal{E}$ admits a solution of the deformed Hermitian-Yang-Mills system.
\end{cor}

\begin{proof}
    By Lemma~\ref{lem: stabForTorsion} if $\mathcal{E}$ is stable with ${\rm Im}(Z_{\mathcal{X}}(E))>0$ and ${\rm Re}(Z_{\mathcal{X}}(\mathcal{E}))\leq 0$, then $E_1 \simeq E_2$, $\Phi \in H^{0}(X, {\rm Hom}(E_2, E_1))$ is never vanishing and 
    \[
    \cos(\htheta) + \frac{2\pi}{\sigma}\sin(\htheta)>0
    \]
    It follows that 
    \[
    \deg(E_1) = \deg(E_2) =\frac{V_X}{2\pi}\left(\frac{ \sin(\htheta)-\frac{2\pi}{\sigma}\cos(\htheta) }{\cos(\htheta) + \frac{2\pi}{\sigma}\sin(\htheta)}\right)
    \]
    Take $h_1$ on $E_1$ so that $i\Lambda F_1= \frac{2\pi}{V_X} \deg E_1$.  Choose $h_2= e^{c}h_1$ for a suitably chosen constant $c \in \mathbb{R}$ so that $|\Phi|^2 = \frac{2\pi}{\sigma}$. By Proposition~\ref{prop: charSolCosNeg}, $\mathcal{E}$ admits a solution of the deformed Hermitian-Yang-Mills system.

    Conversely, if $\mathcal{E}$ has ${\rm Im}(Z_{\mathcal{X}}(E))>0$ and ${\rm Re}(Z_{\mathcal{X}}(\mathcal{E}))\leq 0$ and admits a solution of the system~\eqref{eq: dimReduceDHYM} +~\eqref{eq: dimReduceDHYM-Pos} then Proposition~\ref{prop: charSolCosNeg} implies that $E_1\simeq E_2$ are line bundles, $\Phi \in H^{0}(X, {\rm Hom}(E_2, E_1))$ is never vanishing, and  $\cos(\htheta) + \frac{2\pi}{\sigma}\sin(\htheta)>0$.  By  Lemma~\ref{lem: stabForVeryRestr} $\cE$ is $Z$-stable.
\end{proof}

We now provide a simplified criterion for stability which applies in our setting. Since the case ${\rm Re}(Z_{\cX}(E_1,E_2,\Phi)) \leq 0$ is completely understood by Lemma~\ref{lem: stabForTorsion}, and Corollary~\ref{cor: classifySolRealNeg}, we shall focus only on the case of triples with ${\rm Re}(Z_{\cX}(E_1,E_2,\Phi)) > 0$.  

\begin{lem}\label{lem: Re>0ReduceStabSaturated}
    Let $\cE= (E_1, E_2, \Phi)$ be a holomorphic triple with ${\rm rk}(E_2)=1$ and $\Theta(\cE) \in (0, \frac{\pi}{2})$.  Then, in order to check the stability of  $\cE$ in the sense of Definition~\ref{defn: stabOfTriple}, it suffices to check the stability of $\cE$ against saturated subtriples.
\end{lem}
\begin{proof}
The proof is a consequence of Lemma~\ref{lem: stabForTorsion} and Proposition~\ref{prop: simplifiedStability}.
Since we are assuming that $\Theta(\cE) \in (0, \frac{\pi}{2})$ we conclude from Lemma~\ref{lem: stabForTorsion} part $(i)$ that the first condition of Proposition~\ref{prop: simplifiedStability} is satisfied.  The second condition of Proposition~\ref{prop: simplifiedStability} is precisely equivalent to checking stability of $\cE$ against saturated subtriples. 
\end{proof}

We end this section by noting the following easy lemma.

\begin{lem}\label{lem: PhiCannotVanishIDentically}
    Suppose $\cE=(E_1,E_2,\Phi)$ is $Z$-stable.  Then $\Phi \not\equiv 0$.
\end{lem}
\begin{proof}
    Suppose $\cE$ is stable and $\Phi \equiv 0$.  One obtains a contradiction by checking stability against the subtriples $(E_1,0,0)$ and $(0,E_2,0)$ and using the additivity of the central charge.
\end{proof}

\subsection{Consequences of stability}
In this section we check stability against several ``canonical" subtriples.  The only result from this section that is required for our later purposes is Proposition~\ref{prop: stableImpliesSimple}. On the other hand, the results here give the flavor of the geometric constraints satisfied by $Z$-stable triples, and may play a role in future questions concerning moduli spaces of $Z$-stable triples; see e.g. \cite{BG-P, GPVortex, BD1, BD2, BDGPW, BDW, Thaddeus} for work on the moduli space of stable triples arising from the Vortex equations.    

\begin{lem}\label{lem: StabE_1CanonSub}
    Suppose $\mathcal{E}=(E_1,E_2,\Phi)$ is $Z$-stable with $\htheta= \Theta(\cE) \in (0,\frac{\pi}{2})$. Let $E_{1}'\subset E_1$ be a non-trivial holomorphic sub-bundle and consider the subtriple $\cS =(E_1',0,0)$.  Then 
    \[
    \cos(\htheta) \deg(E_1') - \sin(\htheta) \frac{V_X}{2\pi} {\rm rk}(E_1')<0.
    \]
\end{lem}
\begin{proof}
Assuming that $\deg(E_1')>0$ we apply the definition of $Z$-stability to the saturated  subtriple $\mathcal{S} = (E_1',0,0)$. The lemma follows from~\eqref{eq: subTripleCentralChargeTypeB}.  On the other hand, if $\deg(E_1')\leq 0$ then the inequality is trivially true.
\end{proof}

Next, let $L= {\rm sat}(\Phi(E_2))$ be the saturation of $\Phi(E_2)$. We consider the subtriples $(L, E_2, \Phi)$ and $(L, 0, 0)$.

\begin{lem}\label{lem: StabE_2CanonSub}
    Suppose that the triple $\mathcal{E} = (E_1, E_2, \Phi)$ is $Z$-stable with $\Theta(\cE) \in (0,\frac{\pi}{2})$.  For each $p \in X$ let $m_p$ denote the multiplicity of the zero of $\Phi$ at $p$.  Then we have:
    \begin{itemize}
    \item[(i)]
    \[
    \deg(E_2) + \frac{2V_X}{\sigma}>0
    \]
        \item[(ii)] If $\Phi(E_2)\neq E_1$, then
        \[
        \deg(E_1) + \frac{2V_{X}}{\sigma} > \sum_p m_p,
        \]
        and
        \[
        \cos(\htheta)(\deg(E_2) +\sum_p m_p) < \sin(\htheta) \frac{V_{X}}{2\pi}.
        \]
        \item[(iii)] Furthermore, if $\Phi(E_2)\neq E_1$, and $2\deg(E_2) +\sum_p m_p + \frac{2V_{X}}{\sigma}>0$, then
        \[
        \cos(\htheta)\left(2\deg(E_2) +\sum_p m_p + \frac{2V_{X}}{\sigma}\right) <\sin(\htheta)\left( \frac{V_{X}}{\pi} - \frac{4\pi}{\sigma}\deg(E_2)\right)
        \]
    \end{itemize}
\end{lem}
\begin{proof}
    To prove $(i)$, consider the subtriple $\cS= (E_1, 0 ,0)$. Since $\cE \in \mathcal{T}$ we have ${\rm Im}(Z_{\cX}(\cE/\cS)) >0$.  We deduce $(i)$ from the additivity of the central charge on exact sequences.  For $(ii)$ and $(iii)$ we may assume that ${\rm Re}(Z_{\cX}(\cE))>0$, for otherwise the statement is vacuous by Lemma~\ref{lem: stabForTorsion}.  Define the line bundle $L={\rm sat}(\Phi(E_2))$, and consider the subtriples $\mathcal{S} = (L, E_2,\Phi)$ and $\mathcal{S}' = (L, 0, 0)$.  By direct calculation we have
    \[
    \deg(L) = \deg(E_2) + \sum_{p \in \{\Phi=0\}} m_p
    \]
    where $m_p$ is the multiplicity of the zero of $\Phi$ at $p$.  More precisely, $m_p$ is chosen so that, if $z$ is a local coordinate on $X$ vanishing at $p$, then $\widetilde{\Phi}:= z^{-m_p}\Phi$ is holomorphic near $p$, and ${\rm rk}(\widetilde{\Phi}) =1$ near $p$.  Since $\cE \in \mathcal{T}$ and the quotient $\cE/\cS'$ is non-trivial, we obtain $(ii)$ from Lemma~\ref{lem: StabE_1CanonSub} applied to the triple $\cS'$.  For $(iii)$, assuming that
    \[
    {\rm Im}\left(Z_{\mathcal{X}}(L, E_2,\Phi)\right) = 2\deg(E_2)+ \sum_{p \in \{\Phi=0\}}  m_p +2\frac{V_{X}}{\sigma}>0,
    \]
    stability applied to $\cS$ implies
    \[
    \cos(\hat{\theta})\left(2\deg(E_2)+ \sum_{p \in \{\Phi=0\}}  m_p +2\frac{V_{X}}{\sigma}\right)-\sin(\hat{\theta})\left(\frac{V_{X}}{2\pi}2 -\frac{4\pi}{\sigma}\deg(E_2))\right) <0
    \]
\end{proof}

Recall \cite{BG-P} that a global holomorphic endomorphism of a triple $\mathcal{E}= (E_1,E_2,\Phi)$ is a pair of maps $\alpha_1 \in H^{0}(X, {\rm End}(E_1))$, and $\alpha_2 \in H^{0}(X, {\rm End}(E_2))$ such that
\begin{equation}\label{eq: commutesWithPhi}
\Phi \circ \alpha_2 = \alpha_1\circ \Phi
\end{equation}
Such pairs $(\alpha_1,\alpha_2)$ are equivalently described in terms of a $SU(2)$-equivariant holomorphic endomorphism of $\mathcal{E}\rightarrow \mathcal{X}$.

\begin{defn}\label{defn: PhiSimple}
    We say that a triple $\mathcal{E} = (E_1,E_2,\Phi)$ is $\Phi$-simple if the only global holomorphic endomorphisms of the triple are multiples of the identity. 
\end{defn}

\begin{prop}\label{prop: stableImpliesSimple}
    Suppose that the triple $\mathcal{E}$ is $Z$-stable in the sense of Definition~\ref{def: stabOnXxP}.  Then $\mathcal{E}$ is simple.  Equivalently, the triple $(E_1, E_2,\Phi)$ is $\Phi$-simple in the sense of Definition~\ref{defn: PhiSimple}.
\end{prop}
\begin{proof}
    By Corollary~\ref{cor: classifySolRealNeg} the result is clear if $\cos(\htheta)\leq 0$.  Therefore we may assume that $\cos(\htheta)>0$.  Suppose $(\alpha_1,\alpha_2) \in H^{0}(X,{\rm End}(E_1) \oplus {\rm End}(E_2))$ satisfies~\eqref{eq: commutesWithPhi}.  Since ${\rm rk}E_2=1$, $\alpha_2= \lambda$ for some $\lambda \in \mathbb{C}$.  Then \eqref{eq: commutesWithPhi} implies that $\alpha_1 \circ \Phi = \lambda \Phi$.  Let 
    \[
    \begin{aligned}
    E_1' &= {\rm Ker}(\alpha_1 -\lambda I_{E_1}) \subset E_1\\
    E_{1}'' &= {\rm Im}(\alpha_1 - \lambda I_{E_1}) \subset E_1.
    \end{aligned}
    \]
    For the sake of obtaining a contradiction, assume that $\alpha_1 \ne \lambda I_{E_1}$, so that $E_1''$ is a non-trivial subbundle of $E_1$. Consider the non-trivial holomorphic triples
   \[
    \mathcal{S'} = (E_1', E_2,\Phi) \qquad \mathcal{S}'' = (E_1'', 0, 0).
    \]
    As usual, we also denote by $\mathcal{S}', \mathcal{S}''$ the $SU(2)$-equivariant holomorphic vector bundles over $\mathcal{X}$ induced by the associated triples.  By additivity of the central charge we have
     \begin{equation}\label{eq: addCentralChargeSubQuotient}
    Z_{\mathcal{X}}(\mathcal{S}') + Z_{X}(\mathcal{S}'') = Z_{X}(\mathcal{E}) 
    \end{equation}
    Since both $\cS',\cS''$ are subtriples, property $(i)$ of Definition~\ref{defn: stabOfTriple} and the additivity of the central charge implies
    \[
    {\rm Im}\left(Z_{\mathcal{X}}(\mathcal{S}')\right)>0, \quad \text{and} \quad {\rm Im}\left( Z_{\mathcal{X}}(\mathcal{S}'')\right)>0.
    \]
    We now apply the definition of stability to get
    \[
    {\rm Im}\left(\frac{Z_{\mathcal{X}}(\mathcal{S}')}{Z_{X}(\mathcal{E}) }\right)<0, \quad \text{ and }\quad  {\rm Im}\left(\frac{Z_{\mathcal{X}}(\mathcal{S}'')}{Z_{X}(\mathcal{E}) }\right)<0 
    \]
    contradicting ~\eqref{eq: addCentralChargeSubQuotient}.
\end{proof}

\subsection{Necessity of stability}

The aim of this section is to prove that $Z$-stability, in the sense of Definition~\ref{def: stabOnXxP} or Definition~\ref{defn: stabOfTriple} is necessary for the existence of a solution to the deformed Hermitian-Yang-Mills system.  By Corollary~\ref{cor: classifySolRealNeg}, it suffices to consider the case when $\cos(\htheta)>0$, which implies, by Lemma~\ref{lem: stabForTorsion} and Proposition~\ref{prop: simplifiedStability}, that we can consider only torsion-free quotients.  Necessity of stability for torsion quotients has been studied by Keller-Scarpa \cite{KellerScarpa} and McCarthy \cite{JBM}.  McCarthy \cite{JBM} proved that if $\cE$ admits a $Z$-positive solution of the dHYM equation in the sense of \cite{DMS,  KellerScarpa, JBM} (see Definition~\ref{defn: Zpositive}) then $\cE$ is necessarily stable for torsion quotients of the forms $\cE \twoheadrightarrow \cE\otimes \mathcal{O}_{\mathcal{D}}$ for $\mathcal{D}\subset \cX$ a divisor.  Keller-Scarpa \cite{KellerScarpa} considered the case when ${\rm dim}_{\mathbb{C}} \cX=2$; assuming $Z$-positivity, they proved the necessity of stability for surjections $\cE\twoheadrightarrow \iota_*\cQ$ where $\iota:\mathcal{V}\hookrightarrow \cX$ is a reduced and irreducible analytic subset of dimension $1$, and $\cQ\rightarrow \mathcal{V}$ is a rank $1$ quotient of $\iota^*\cE$. 

\begin{lem}\label{lem: torsionClassNecessary}
Suppose $\cE=(E_1,E_2,\Phi)$ is a holomorphic triple with $\rank(E_2)=1$, and suppose that ${\rm Im}(Z_{\cX}(\cE))>0$, ${\rm Re}(Z_{\cX}(\cE))>0$.  Suppose that $\cE$ admits a solution of the system~\eqref{eq: dimReduceDHYM}+~\eqref{eq: dimReduceDHYM-Pos}.  Then $\cE \in \mathcal{T}$.  That is, for every surjection $\cE \twoheadrightarrow \cQ$ in ${\rm Coh}^{SU(2)}(\cX)$ we have ${\rm Im}(Z_{\cX}(\cQ))>0$.
\end{lem}
\begin{proof}
    By Lemma~\ref{lem: simplifiedPos}, when~\eqref{eq: dimReduceDHYM} holds, the calibration constraint~\eqref{eq: dimReduceDHYM-Pos} is equivalent to
    \[
    \begin{aligned}
i\Lambda F_1 + \Phi\Phi^{\dagger} &>0\\
i\Lambda F_2 + \frac{4\pi}{\sigma} - |\Phi|^2 &>0
\end{aligned}
    \]
    Let $\cS\subset \cE$ be a subtriple and $\cQ=\cE/\cS$ be the quotient sheaf.  It suffices to prove that ${\rm Im}(Z_{\cX}(\cS))< {\rm Im}(Z_{\cX}(\cE))$.  Recall that if $\cS= (E_1',E_2',\Phi')$, then
    \[
    {\rm Im}(Z_{\cX}(\cS)) \propto \deg(E_1')+\deg(E_2')+ \frac{2V_{X}}{\sigma}\rank(E_2').
    \]
    Since this expression increases by passing to the saturation, we can assume that $\cS$ is a saturated subtriple of $\cE$.  Suppose $\cS$ is a saturated subtriple of type $A$, so that $E_2'=E_2$.  Let $Q=E_1/E_1'$.   Let $\pi$ denote the orthogonal projection to $E_1'$ and let $\pi^{\perp}$ be the projection to $Q$.  Applying $\pi^{\perp}$ to the above inequality for $F_1$, taking the trace and integrating yields
    \[
    \deg Q - \frac{1}{2\pi}\int_{X}|\beta_{Q}|^2>0
    \]
    where $\beta_{Q}$ is the second fundamental form of $Q$ and we used that $\Phi(E_2) \subset E_1'$ and so $\pi^{\perp}\Phi=0$.  Thus, $\deg(Q)>0$.  Since  ${\rm Im}(Z_{\cX}(\cQ)) \propto \deg Q$ when $\cS$ is of type A, the result follows.

    Next assume $\cS$ is of type B, so $\cS= (E_1',0,0)$.  Let $\pi^{\perp}$ again denote the projection to $Q= E_1/E_1'$.  Then we have
    \[
    i\Lambda F_{Q} - |\beta|^2 + |\pi^{\perp}\Phi|^2 + i\Lambda F_2 + \frac{4\pi}{\sigma} - |\Phi|^2 >0
    \]
    In particular we obtain
    \[
    \deg Q + \deg E_2 + \frac{2V_{X}}{\sigma} > \frac{1}{2\pi}\int_{X} |\pi^{\perp}\Phi|^2+|\beta_{Q}|^2
    \]
    Since ${\rm Im}(Z_{\cX}(\cQ)) \propto \deg Q + \deg E_2 + \frac{2V_{X}}{\sigma} $, the proof is complete.
\end{proof}

Before proceeding to consider stability for sub-objects, we prove a lemma which gives an equivalent formulation of stability in the case that $\cos(\htheta)>0$.

\begin{lem}\label{lem: phaseToRatioStab}
Suppose $\cE=(E_1,E_2,\Phi)$ is a holomorphic triple of vortex type with ${\rm Im}Z_{\cX}(\cE)>0$ and ${\rm Re}(Z_{\cX}(\cE))>0$.  Then $\cE$ is $Z$-stable if and only if:
\begin{itemize}
    \item[(i)]$\cE \in \mathcal{T}$ and,
    \item[(ii)]for every non-zero saturated subtriple $\cS=(S_1,S_2,\Phi')$ we have
\begin{equation}\label{eq: slopeTypecondition}
{\rm Im}\left(\frac{Z_{X}(\cS)}{Z_{\cX}(\cE)}\right)<0.
\end{equation}
\end{itemize}
\end{lem}
\begin{proof}
We first prove that the assumptions $(i)$ and $(ii)$ imply $Z$-stability.  Since $(i)$ gives $\cE \in \mathcal{T}$, it suffices to observe that~\eqref{eq: slopeTypecondition} implies the angle condition of Definition~\ref{def: stabOnXxP}.

Showing the angle condition of Definition~\ref{def: stabOnXxP} implies ~\eqref{eq: slopeTypecondition} requires more work.  The concern is that, in principle, $\cE$ could admit subtriples with $Z_{X}(\cS)$ having large negative real component and bounded imaginary component. These objects are not valid objects for testing stability in Definition~\ref{def: stabOnXxP}, and they may not satisfy ~\eqref{eq: slopeTypecondition}.  We show that these objects cannot exist.

Precisely, suppose that $\cE$ satisfies the angle inequality of Definition~\ref{def: stabOnXxP}. Let $\cS=(S,0,0)$ be a subtriple of type $B$. We need to show that
\[
{\rm Im}\left(\frac{Z_{X}(\cS)}{Z_{\cX}(\cE)}\right)= \cos(\htheta) \deg(S)-\sin(\htheta)\frac{V_{X}}{2\pi}\rank(S)<0.
\]
If $\deg(S)\leq 0$ this is clear, while if $\deg(S)>0$ then it follows from $Z$-stability.

Now consider a saturated subtriple $\cS=(S,E_2,\Phi)$ of type A. The desired inequality is a consequence of $Z$-stability if $Z_{\cX}(\cS)\in \mathbb{H}$, while if ${\rm Re}(Z_{\cX}(\cS))>0$, then~\eqref{eq: slopeTypecondition} holds , even when ${\rm Im}(Z_{\cX}(\cS))<0$.  We only need to address the case
\[
{\rm Re}(Z_{\cX}(\cS))<0 \quad \text{ and }\quad  {\rm Im}(Z_{\cX}(\cS))<0.
\]
From~\eqref{eq: TripleCentralCharge} we see that ${\rm Re}(Z_{\cX}(\cS))<0$ implies
\begin{equation}\label{eq: degE2hugecontr}
\frac{4\pi}{\sigma} \deg(E_2) > \frac{V_{X}}{2\pi}(\rank(S)+1) > \frac{V_{X}}{\pi}.
\end{equation}
Now Let $L= {\rm sat}(\Phi(E_2))= E_2(D)$ for $D$ and effective divisor. Consider the subtriple $\cS_{\Phi}:= (L, E_2,\Phi)$.  Then
\[
Z_{\cX}(\cS_{\Phi}) = \left(\frac{V_{X}}{\pi} - \frac{4\pi}{\sigma}\deg(E_2)\right) +i \left(2\deg(E_2) + \deg(D)+  \frac{2V_{X}}{\sigma}\right).
\]
But since $\deg(E_2)> \frac{V_{X}}{\pi}>0$ by~\eqref{eq: degE2hugecontr} we see that $Z_{\cX}(\cS_{\Phi}) \in \mathbb{H}$ and has $\Theta(\cS_{\Phi}) \in (\frac{\pi}{2}, \pi]$.  This contradicts the $Z$-stability of $\cE$.
\end{proof}

Next we consider condition $(ii)$ of Definition~\ref{def: stabOnXxP}. We prove the necessity of stability for bundles of vortex type.  Perhaps surprisingly, the proofs are not local, depending in an essential way on the global structure via integration-by-parts.  In particular, our arguments are rather different than those of Keller-Scarpa \cite{KellerScarpa}.  We split our analysis into two cases, depending on whether we consider sub-objects of type A or type B, in the sense of Definition~\ref{defn: subObTypeA/B}.

\begin{prop}\label{prop: stabTypeANecessary}
    Suppose that $\cE= (E_1, E_2, \Phi)$ admits a solution of the system~\eqref{eq: dimReduceDHYM}+~\eqref{eq: dimReduceDHYM-Pos} with $\cos(\htheta)>0$.  Let $\cS = (S,E_2,\Phi')$ be a saturated subtriple of $(E_1,E_2,\Phi)$ of type A.  Then
        \[
         {\rm Im}\left(\frac{Z_{X}(\cS)}{Z_{\cX}(\cE)}\right)\leq 0
        \]
    with equality if and only if $E_{1}$ splits orthogonally and holomorphically as a direct sum $E_{1} = S \oplus (E_1/S)$ with $\Phi= \Phi' \oplus \{0\}$. In this case we have a splitting of $SU(2)$-equivariant holomorphic vector bundles $\mathcal{E} = \cS \oplus \cE/\cS$, where $\cS, \cE/\cS'$ admit $SU(2)$-equivariant solutions of the dHYM system~\eqref{eq: dHYMIntroduction}+~\eqref{eq: dHYMPosIntroduction} with
        \[
        \Theta(\cS)=\Theta(\cE/\cS)
        \]
\end{prop}

\begin{proof}
Let $\pi$ denote the $h_{1}$ orthogonal projection to $S\subset E_1$, and let $\pi^{\perp}$ denote the orthogonal projection to $S^{\perp}\sim Q:= E_1/S$. Let $h_{1,S}, h_{1,Q}$ denote the induced metrics on $S, Q$ respectively and let $\beta$ denote the second fundamental form of $S\subset E_1$.  Note that by additivity of the central charge, the desired result is equivalent to proving
\[
 {\rm Im}\left(\frac{Z_{X}(\cQ)}{Z_{\cX}(\cE)}\right)\geq 0
 \]
 where $\cQ=(E_1/S,0,0)$ is the quotient triple. We apply $\pi^{\perp}$ to both sides of~\eqref{eq: dimReduceDHYM} to obtain
\[
\begin{aligned}
    &\cos(\hat{\theta})( i\Lambda F_{h_{1,Q}} - i\Lambda \beta^{\dagger}\wedge \beta ) \\
    &= \sin(\hat{\theta})(I_{Q} + \frac{1}{2}g^{p\bar{j}}\left(\pi^{\perp}\nabla_{p}\Phi \right)\left(\pi^{\perp}\nabla_{j}\Phi\right)^{\dagger})
\end{aligned}
\]
where we used that $\pi^{\perp}\Phi=0$.  Now since $\pi\Phi=\Phi$ we have $
\pi^{\perp}\nabla\Phi = -\beta^{\dagger}\Phi
$
and hence, upon tracing and integrating we obtain
\[
\cos(\htheta)(2\pi \deg(Q) - \int_{X}|\beta|^2)=\sin(\htheta)\left({\rm rk}(Q) V_{X} + \frac{1}{2}\int_{X}|\beta^{\dagger}\Phi|^2\right).
\]
Therefore,
\begin{equation}\label{eq: typeAquotientDelta}
\cos(\htheta)\deg(Q) - \frac{V_{X}}{2\pi}\sin(\htheta){\rm rk}(Q) = \frac{\cos(\htheta)}{2\pi} \int|\beta|^2 + \frac{1}{4\pi}\sin(\htheta)\int_{X}|\beta^{\dagger}\Phi|^2
\end{equation}
and so, since $\sin(\htheta), \cos(\htheta)>0$ we get
\[
\cos(\htheta)\deg(Q) - \frac{V_{X}}{2\pi}\sin(\htheta)\rank(Q)  \geq 0
\]
with equality if and only if $\beta \equiv 0$.  By the additivity of the central charge, this yields $\Theta(S) \leq \Theta(\cE)$ with equality if and only if $\beta \equiv 0$, in which case $E_1 = S \oplus S^{\perp}$ is a holomorphic and orthogonal splitting. It follows easily that $h_{1,S}$ and $h_{1,Q}$ both solve the deformed Hermitian-Yang-Mills system ~\eqref{eq: dimReduceDHYM}+~\eqref{eq: dimReduceDHYM-Pos} with the same phase $\hat{\theta}$.
\end{proof}

We now prove necessity for sub-bundles of type $B$.

 \begin{lem}\label{lem: typeBStabNecessary}
    Suppose that $\cE= (E_1, E_2, \Phi)$ admits a solution of the system~\eqref{eq: dimReduceDHYM}+~\eqref{eq: dimReduceDHYM-Pos} with $\cos(\htheta)>0$.  Let $\cS =(S,0,0)$ be a saturated subtriple of $(E_1,E_2,\Phi)$.  Suppose that either:
    \begin{itemize}
        \item[(i)]  $\Phi(E_2) \subset S$, or
        \item[(ii)] $\cos(\htheta) \geq \frac{4\pi}{\sigma}\sin(\htheta)$
        \end{itemize}
     Then we have
    \[
   {\rm Im}\left(\frac{Z_{\cX}(\cS)}{Z_{\cX}(\cE)}\right) \leq 0
    \]
    with equality if and only if $E_1$ splits metrically and holomorphically as $S\oplus E_1/S$ and $\Phi(E_2)\subset E_1/S$.
\end{lem}
\begin{proof}
    Let $\pi$ denote the projection to $S\subset E_1$.  Applying $\pi$ to~\eqref{eq: dimReduceDHYM}, taking the trace and integrating yields
    \[
    \cos(\htheta)\left(2\pi \deg(S) +\int_{X} |\beta|^2+ \int_{X}|\pi\Phi|^2\right) = \sin(\htheta)\left({\rm rk}(S)V_{X} - \int_{X}{\rm Re}(\langle i\Lambda F_1 \Phi, \pi \Phi \rangle) + \frac{1}{2}|\pi \nabla \Phi|^2 \right)
    \]
    We need to simplify the term $\int_{X}{\rm Re}(\langle i\Lambda F_1 \Phi, \pi \Phi \rangle$.  By integration by parts we have
    \[
    \begin{aligned}
        \int_{X}{\rm Re}(\langle i\Lambda F_1 \Phi, \pi \Phi \rangle&= \int_{X}i\Lambda F_2|\pi\Phi|^2 +\int_{X} {\rm Re}\langle \nabla \Phi, \nabla(\pi \Phi) \rangle\\
        &= \int_{X}i\Lambda F_2|\pi\Phi|^2 +\int_{X}|\pi \nabla \Phi|^2- {\rm Re}\langle \pi^{\perp}\nabla \Phi, \beta^{\dagger}\pi \Phi \rangle.
    \end{aligned}
    \]
    Therefore, we have
    \begin{equation}\label{eq: specCasetypeBhardPart}
    \begin{aligned}
    2\pi \left(\cos(\htheta)\deg(S)- \sin(\htheta){\rm rk}(S)\frac{V_X}{2\pi}\right) &=  \sin(\htheta)\left(  \int_{X} {\rm Re}\langle \pi^{\perp}\nabla \Phi, \beta^{\dagger}\pi \Phi \rangle - \frac{1}{2}|\pi \nabla \Phi|^2 \right)\\&\quad - \cos(\htheta)\int|\beta|^2
     -\int_{X} \left(\sin(\htheta)i\Lambda F_2 + \cos(\htheta)\right)|\pi\Phi|^2 
    \end{aligned}
    \end{equation}
    Now we observe that by~\eqref{eq: dimReduceDHYM}, or Lemma~\ref{lem: H2CurvLem} we have
    \begin{equation}\label{eq: locF2eqStabNec}
    \sin(\htheta) i\Lambda F_2 + \cos(\htheta) =\frac{\left( 1+ \frac{\sin^2(\htheta)}{2}|\nabla \Phi|^2\right)}{\cos(\htheta) + \left(\frac{4\pi}{\sigma}-|\Phi|^2\right)\sin(\htheta)}
    \end{equation}
    Therefore, the only term for which the sign is unclear is the term ${\rm Re}\langle \pi^{\perp}\nabla \Phi, \beta^{\dagger}\pi \Phi \rangle$.  We now consider case $(i)$.  If $\Phi(E_2) \subset S$, then $\pi \Phi = \Phi$, and $\pi^{\perp}\Phi = 0$.  In this case
    \[
    \pi^{\perp}\nabla \Phi = -\beta^{\dagger}\pi\Phi
    \]
    and so
    \begin{equation}\label{eq: DeltaBEasyCase}
    \begin{aligned}
    2\pi \left(\cos(\htheta)\deg(S)- \sin(\htheta){\rm rk}(S)\frac{V_X}{2\pi}\right) &=  -\sin(\htheta)\left(  \int_{X} |\beta^{\dagger}\pi \Phi|^2 + \frac{1}{2}|\pi \nabla \Phi|^2 \right)- \cos(\htheta)\int|\beta|^2\\
    &\quad -\int_{X}\left(\sin(\htheta)i\Lambda F_2 +\cos(\htheta)\right)|\pi\Phi|^2
    \end{aligned}
    \end{equation}
    from which the result follows using~\eqref{eq: locF2eqStabNec}.  Since $\pi\Phi=\Phi$ we see that equality is never achieved unless $\Phi=0$ and $\beta=0$.  Now we suppose that $\pi^{\perp}\Phi \ne 0$.  We bound  
    \[
    \sin(\htheta)\big|{\rm Re}\langle \pi^{\perp}\nabla \Phi, \beta^{\dagger}\pi \Phi \rangle\big| \leq \frac{\epsilon^2\sin^2(\htheta)}{2}|\pi^{\perp}\nabla \Phi|^2|\pi\Phi|^2 + \frac{1}{2\epsilon^2}|\beta|^2
    \]
    Choose $\frac{1}{2\epsilon^2} = \cos(\htheta)$.  Then
  \[
    \sin(\htheta)\big|{\rm Re}\langle \pi^{\perp}\nabla \Phi, \beta^{\dagger}\pi \Phi \rangle\big| \leq \frac{\sin^2(\htheta)}{4\cos(\htheta)}|\pi^{\perp}\nabla \Phi|^2|\pi\Phi|^2 + \cos(\htheta)|\beta|^2
    \]  
    We attempt to absorb the $\nabla \Phi$ term using the corresponding term from~\eqref{eq: locF2eqStabNec}.  In order for this to be possible, we need
    \[
    \frac{1}{\cos(\htheta) + (\frac{4\pi}{\sigma}-|\Phi|^2)\sin(\htheta)} \geq  \frac{1}{2\cos(\htheta)}
    \]
    or in other words,
    \[
    \cos(\htheta) \geq (\frac{4\pi}{\sigma}-|\Phi|^2)\sin(\htheta),
    \]
    which holds unconditionally provided $\cos(\htheta) \geq \frac{4\pi}{\sigma}\sin(\htheta)$.  Note that equality cannot hold unless $\pi \Phi =0$ and $\beta=0$. This establishes $(ii)$.
\end{proof}

The difficulty in extending the proof of necessity in Lemma~\ref{lem: typeBStabNecessary} to the full range of parameters is controlling the term
\begin{equation}\label{eq: typeBbadTerm}
{\rm Re}\langle \pi^{\perp}\nabla \Phi, \beta^{\dagger}\pi \Phi \rangle
\end{equation}
appearing in~\eqref{eq: specCasetypeBhardPart}.  The challenge is that this term has no obvious sign and is {\em linear} in $\pi\Phi$, while the helpful negative terms are quadratic in $\pi\Phi$.  Thus the negative terms cannot control~\eqref{eq: typeBbadTerm} in the region where $|\pi\Phi|$ is small.  In the proof of Lemma~\ref{lem: typeBStabNecessary}, we used a pointwise argument via Cauchy-Schwarz, which requires $\cos(\htheta) \geq \frac{4\pi}{\sigma}\sin(\htheta)$.  Beyond this regime we expect that proving the necessity of stability may require a more elaborate approach.  Indeed, the natural approach to controlling~\eqref{eq: typeBbadTerm} is to integrate by parts to obtain curvature terms and then invoke the system~\eqref{eq: dimReduceDHYM}.  Unfortunately, this introduces further second fundamental form terms with unfavorable signs and the resulting off-diagonal curvature terms also do not have any clear sign.  Instead, extending Lemma~\ref{lem: typeBStabNecessary} to the full range of parameters may require developing an infinite-dimensional GIT approach along the lines of Collins-Yau \cite{CY, CYarX}; we explain this briefly in Section~\ref{sec: comments}.  It is interesting to note that a similar issue arises in the estimates in Section~\ref{sec: Estimates} in the blow-up argument of Proposition~\ref{prop: UYargumentH} in which one extracts, by a contradiction argument, a subtriple of type $A$.

\begin{conj}\label{conj: stabTypeB}
    Suppose that $\cE = (E_1, E_2, \Phi)$ admits a solution of the deformed Hermitian-Yang-Mills system with $\cos(\htheta)>0$.  Suppose that $\cS= (S,0,0)$ is a saturated subtriple of type B.  Then
    \[
    \Theta(\cS) \leq \Theta(\cE)
    \]
    with equality if and only if $E_1$ splits orthogonally and holomorphically as $E_1= S \oplus Q$, and $\Phi \in H^{0}(X,Q\otimes E_2^{\vee})$.
\end{conj}

Finally, we can prove the necessity of stability.

\begin{prop}\label{prop: stabNecessaryCosPos}
    Suppose that $\cE= (E_1, E_2, \Phi)$ is equivariantly irreducible and admits a solution of the system~\eqref{eq: dimReduceDHYM}+~\eqref{eq: dimReduceDHYM-Pos} with $\hat{\theta}\in(0, \arctan(\frac{\sigma}{4\pi})] \cup[\frac{\pi}{2}, \pi]$.  Then $\cE$ is $Z$-stable in the sense of Definition~\ref{def: stabOnXxP}.
\end{prop}

\begin{proof}
If $\cos(\htheta)\leq 0$ then the conclusion follows from the stronger result, Corollary~\ref{cor: classifySolRealNeg}.  We may therefore assume that $\cos(\htheta)>0$. That $\cE\in \mathcal{T}$ follows from Lemma~\ref{lem: torsionClassNecessary}.  By Proposition~\ref{prop: simplifiedStability} it suffices to check stability for surjections $\cE\twoheadrightarrow \cQ$ in ${\rm Coh}^{SU(2)}(\cX)$ where $\cQ$ is either a vector bundle, or a vector bundle supported over a $\mathbb{P}^1$-fiber.  By Lemma~\ref{lem: stabForTorsion}, the latter objects can never destabilize.  Thus it suffices to consider saturated subtriples.  By Lemma~\ref{lem: phaseToRatioStab} $Z$-stability is equivalent to proving $ {\rm Im}\left(\frac{Z_{\cX}(\cS)}{Z_{\cX}(\cE)}\right) <0$ for every saturated subtriple. The result now follows from Proposition~\ref{prop: stabTypeANecessary}, and Lemma~\ref{lem: typeBStabNecessary} since irreducibility implies that the equality cases of these results cannot occur.
\end{proof}

\section{Solving the initial equation}\label{sec: initialEquation}

Sections~\ref{sec: initialEquation}-\ref{sec: fixedPoint} are dedicated to proving the following theorem:

\begin{thm}\label{thm: existenceCos>0}
    Suppose $\cE=(E_1,E_2,\Phi)$ is an irreducible holomorphic triple of vortex type.  Suppose that ${\rm Re}(Z_{\cX}(\cE))>0$ and ${\rm Im}(Z_{\cX}(\cE))>0$ and $\tan(\htheta)\leq \frac{\sigma}{4\pi}$.  If $\cE$ is $Z$-stable in the sense of Definition~\ref{def: stabOnXxP}, then $\cE$ admits a solution of the deformed Vortex equations~\eqref{eq: dimReduceDHYM}+~\eqref{eq: dimReduceDHYM-Pos}.
\end{thm}

We approach this via the method of continuity introduced in Section~\ref{sec: dimReduction}. We comment that the restriction $\tan(\htheta)\leq \frac{\sigma}{4\pi}$ is only used in Section~\ref{subsec: EstFromStab}.  In this section we briefly explain how to solve the equations~\eqref{eq: dimReduceDHYM-MOC} when $\cos(\htheta)>0$ and $t=0$.  This follows from arguments of Bradlow \cite{Bradlow}, building on works of Simpson \cite{Simpson}, Donaldson \cite{Do83, Do85, Do87} and Uhlenbeck-Yau \cite{UY}.  

\begin{thm}\label{thm: solveInitialEquation}
Suppose $\cE= (E_1,E_2,\Phi)$ is an irreducible triple of vortex type with ${\rm Re}(Z_{\cX}(\cE))>0$, and $ {\rm Im}(Z_{\cX}(\cE)) >0$.  Then there exists metrics $(h_{1},h_{2})$ on $(E_1,E_2)$ solving 
\begin{equation}\label{eq:t0-vortex}
\begin{aligned}
(i\Lambda F_{1} + \Phi\Phi^{\dagger}) &= \tan(\htheta)I_{E_1}\\
(1 + \frac{4\pi}{\sigma}\tan(\htheta))i\Lambda F_2 - |\Phi|^2 &= \tan(\htheta)-\frac{4\pi}{\sigma}
\end{aligned}
\end{equation}
if and only if $(E_1,E_2,\Phi)$ is $Z$-stable in the sense of Definition~\ref{def: stabOnXxP} or equivalently Definition~\ref{defn: stabOfTriple}.  Furthermore, the solution is unique up to rescaling $(h_{1},h_{2})\mapsto (e^{c}h_{1},e^{c}h_{2}) $.
\end{thm}

The necessity of stability follows easily from the standard second fundamental form calculations.

\begin{lem}
    Suppose $\cE= (E_1,E_2,\Phi)$ is a holomorphic triple with $\rank(E_2)=1$ and ${\rm Re}(Z_{\cX}(\cE))>0$, and $ {\rm Im}(Z_{\cX}(\cE)) >0$.  Suppose there exist Hermitian metrics $(h_1,h_2)$ solving~\eqref{eq:t0-vortex}.  Then for any saturated subtriple $\cS\subset \cE$ we have
    \[
    {\rm Im}\left(\frac{Z_{\cX}(\cS)}{Z_{\cX}(\cE)}\right)<0.
    \]
\end{lem}
\begin{proof}
    We sketch the proof since it is very similar to the calculation in the previous section.  If $\cS=(S,E_2,\Phi)$ is of type $A$, then let $\pi^{\perp}$ denote the $h_1$-orthogonal projection to $S^{\perp}\subset E_1$.  Applying $\pi^{\perp}$ to the equation for $F_1$, tracing and integrating yields
    \[
    2\pi \deg(Q)-\int_{X}|\beta|^2 = V_{X}\tan(\htheta)\rank(Q)
    \]
    which yields the desired result.  If $\cS=(S,0,0)$ is a subtriple of type $B$ then applying $\pi$ to the first equation yields
    \[
    2\pi\deg(S) + \int_{X}|\beta|^2 +\int_{X}|\pi\Phi|^2= \tan(\htheta)V_{X}\rank(S)
    \]
    which again yields the desired result thanks to Lemma~\ref{lem: phaseToRatioStab}.
\end{proof}

By Lemma~\ref{lem: phaseToRatioStab} we see that $Z$-stability is a necessary condition for solving ~\eqref{eq:t0-vortex}. We briefly outline the proof of sufficiency, which follows essentially verbatim from Bradlow's work \cite{Bradlow}.  For more details, we refer the reader to the second author's Ph.D. thesis \cite{YukaiThesis}.  Fix background Hermitian metrics \(h_{1,0},h_{2,0}\), and for simplicity, denote by $\lambda =\tan(\htheta)>0$.  Let
\(\mathcal D_i(h_{i,0},h_i)\) denote the Donaldson functional on \(E_i\). If \(h(t)\) is any smooth path from \(h_{i,0}\) to \(h_i\), then
\begin{equation}
\label{eq:t0-donaldson-path-integral}
  \mathcal D_i(h_{i,0},h_i)
  :=
  \int_0^1
  \int_X\operatorname{tr}\bigl(u(t)\sqrt{-1}\Lambda F_{h(t)}\bigr)\,\omega\,dt,
  \qquad
  \dot h(t)=h(t)u(t).
\end{equation}
This definition is independent of the chosen path. Since \(E_2\) is a line bundle, \(u_2\) is a real valued function.  We define
\begin{equation}
\label{eq:t0-functional}
\begin{aligned}
  \mathcal M(h_1,h_2)
  :=&\ \mathcal D_1(h_{1,0},h_1)
      -\lambda\int_X \log\det(h_{1,0}^{-1}h_1)\,\omega  \\
    &+(1+\frac{4\pi}{\sigma}\lambda)\,\mathcal D_2(h_{2,0},h_2)
      -(\lambda-\frac{4\pi}{\sigma})\int_X \log(h_{2,0}^{-1}h_2)\,\omega
      +\int_X |\Phi|^2_{h_1h_2^{-1}}\,\omega .
\end{aligned}
\end{equation}

A straightforward calculation shows that the first variation of $\mathcal{M}$ satisfies
\begin{equation}
\label{eq:t0-first-variation}
\begin{aligned}
  \frac{d}{ds}\mathcal M(h_1(s),h_2(s))\bigg|_{s=0}
  =&\int_X \operatorname{tr}
  \bigl(u_1(\sqrt{-1}\Lambda F_{1}+\Phi\Phi^\dagger-\lambda I_{E_1})\bigr)\,\omega  \\
  &+\int_X u_2
  \bigl((1+\frac{4\pi}{\sigma}\lambda)\sqrt{-1}\Lambda F_{2}-|\Phi|^2-(\lambda-\frac{4\pi}{\sigma})\bigr)\,\omega .
\end{aligned}
\end{equation}
It follows that sufficiently regular critical points of \(\mathcal M\) are exactly the solutions of
\eqref{eq:t0-vortex}.
This variational identification allows us to approach existence by a variational approach.  An important feature of $\mathcal{M}$ is that it is convex along $1$-parameter subgroups.  

\begin{lem}
\label{lem:t0-functional-convexity}
Let $ h_i(s)=h_i e^{s u_i}$ be a geodesic in the space of Hermitian metrics on $(E_1,E_2)$, where \(u_i\) is Hermitian with respect to
\(h_i\).  Then, for every \(s\),
\begin{equation}
\label{eq:t0-second-variation}
\begin{aligned}
  \frac{d^2}{ds^2}\mathcal M(h_1(s),h_2(s))
  =
  &\int_X |\bar\partial_{E_1}u_1|^2_{h_1(s)}\,\omega
  +\bigl(1+\frac{4\pi}{\sigma}\lambda\bigr)\int_X |\bar\partial u_2|^2\,\omega   \\
  &+\int_X |u_1\Phi-\Phi u_2|^2_{h_1(s)h_2(s)^{-1}}\,\omega .
\end{aligned}
\end{equation}
In particular, since \(1+\frac{4\pi}{\sigma}\lambda>0\), the functional
\(\mathcal M\) is convex along geodesics.
\end{lem}

It is well-known that the functional $\mathcal{M}$ is the Kempf-Ness functional associated to a certain infinite dimensional GIT problem, though we will not need this point of view.  The key result is the following, which is adapted from \cite[Proposition 3.6.1]{Bradlow}

\begin{prop}\label{prop: properness}
    Suppose that $\cE=(E_1,E_2,\Phi)$ is stable in the sense of Definition~\ref{def: stabOnXxP}.  Then there are constants $C,\epsilon>0$ so that, for any $W^{2,p}$ $(p>2)$ $h_{i,0}$-Hermitian endomorphisms $u_i$ of $E_i$ satisfying $\int_{X}{\rm Tr}(u_1) + {\rm Tr}(u_2)=0$ we have
    \begin{equation}\label{eq: propernessEstimate}
    \mathcal{M}(h_{1,0}e^{u_1}, h_{2,0}e^{u_2}) \geq \epsilon \left(\|u_1\|_{L^{1}} + \left( 1 + \frac{4\pi}{\sigma}\lambda\right)\|u_{2}\|_{L^{1}}\right) -C
    \end{equation}
\end{prop}

We refer the reader to the second author's Ph.D. thesis \cite{YukaiThesis} for complete details.  The idea, following the philosophy of GIT, is that if the properness estimate~\eqref{eq: propernessEstimate} does not hold, then one can produce a destabilizing subtriple of $\cE=(E_1,E_2,\Phi)$.  With Proposition~\ref{prop: properness} in hand, Theorem~\ref{thm: solveInitialEquation} follows easily by standard arguments.

\section{A priori estimates}\label{sec: Estimates}

In this section we prove the a priori estimates which will lead to our main existence result, Theorem~\ref{thm: existenceCos>0}.  Let us record the setting in which our estimates take place.  We assume we have metrics $h_1$ on $E_1$, and $h_2$ on $E_2$ and denote by $h= h_1h_2^{-1}$ the metric on $E_1\otimes E_2^{\vee} = {\rm Hom}(E_2, E_1)$.  We fix $\Phi \in H^{0}(X, {\rm Hom}(E_2, E_1))$ and let $\Phi^\dagger$ be the smooth section of ${\rm Hom}(E_1, E_2)$ induced by $\Phi$ and $h$.  We suppose that this data satisfies~\eqref{eq: dimReduceDHYM-MOC} for some $t\in[0,1]$. We assume throughout that $\cos(\htheta)>0$.  The estimates proceed as follows:

\begin{itemize}
    \item Step 1: Establish a bound for $\sup_{X}|\Phi|^2$ by the maximum principle.
    \\
    \item Step 2: Establish a bound for $\sup_{X}|\nabla \Phi|^2$ by the maximum principle.  As a consequence, we deduce uniform curvature bounds.
    \\
    \item Step 3: Using stability, prove a uniform upper bound for the  metric $h=h_1h_{2}^{-1}$ by a contradiction argument and the Uhlenbeck-Yau technique.  This step requires $\tan(\htheta)\leq \frac{\sigma}{4\pi}$.
    \\
    \item Step 4: Using another contradiction argument prove uniform upper and lower bounds for the  metrics $(h_1,h_{2})$ subject to a suitable normalization.
\end{itemize}

\subsection{Bounds for $|\Phi|$ and $|\nabla \Phi|$}
As a first step we prove a strong $C^0$ bound for $\Phi$.

\begin{lem}\label{lem: conditionalPhiBound}
    Suppose $(h_1, h_2, \Phi)$ solve the system~\eqref{eq: dimReduceDHYM-MOC} for $t\in[0,1]$.  Assume that $\cos(\hat{\theta}) >0$ and $\inf_{X}|\Phi|^2 \leq \frac{4\pi}{\sigma t}$. Then we have the following estimates
    \begin{itemize}
        \item[(i)]
    \[
    \sup_{X} |\Phi|^2 \leq \frac{4\pi}{\sigma f(\htheta, t)} \leq  \begin{cases} \frac{2\pi}{\sigma} & \text{ if } \tan(\htheta) \leq \frac{2\pi}{\sigma}\\
    \frac{2\pi}{\sigma \cos(\htheta)}\frac{1}{\cos(\htheta) + \frac{2\pi}{\sigma}\sin(\htheta)} & \text{ if }\tan(\htheta) > \frac{2\pi}{\sigma}
    \end{cases}
    \]
    where
    \[
    f(\htheta, t)= \left(2\cos(\hat{\theta})\left(\cos(\hat{\theta}) +\frac{2\pi}{\sigma}\sin(\hat{\theta})\right) + 2t\sin(\hat{\theta})\left(\sin(\hat{\theta})- \frac{2\pi}{\sigma}\cos(\hat{\theta})\right)\right).
    \]
    \item[(ii)] Furthermore, we have 
    \[
        \cos(\hat{\theta}) + \left(\frac{4\pi}{\sigma} -t|\Phi|^2\right) \sin(\hat{\theta}) > \cos(\hat{\theta}) + \frac{2\pi}{\sigma}\sin(\hat{\theta}) >0.
        \]
        \end{itemize}
\end{lem}
\begin{proof}
    We compute
    \[
    \Delta|\Phi|^2 = |\nabla \Phi|^2 -\langle \Phi, i\Lambda F \Phi\rangle
    \]
    where $F$ is the curvature of $h=h_1h_2^{-1}$. Assume that $|\Phi|^2$ attains a maximum at $x_*\in X$.   We combine Lemma~\ref{lem: H1CurvLem} with Lemma~\ref{lem: H2CurvLem} to obtain a formula for $\langle \Phi, i\Lambda F \Phi\rangle$.  Note that since $x_*$ is the maximum we have that $\langle \nabla \Phi, \Phi \rangle =0$ at $x_*$, so Lemma~\ref{lem: H1CurvLem} (or, more straightforwardly, equation~\eqref{eq: curvH_1PhiPhi}) together with ~\eqref{lem: H2CurvLem} yields
    \[
    \begin{aligned}
    0 &\geq |\nabla \Phi|^2+ \frac{\cos(\hat{\theta})|\Phi|^4 -\sin(\hat{\theta})|\Phi|^2 }{\left( \cos(\hat{\theta})+\sin(\hat{\theta})t|\Phi|^2\right)} +\frac{\sin(\hat{\theta})(1+\frac{t}{2}|\nabla \Phi|^2)|\Phi|^2 -\cos(\hat{\theta})\left(\frac{4\pi}{\sigma}-|\Phi|^2\right)|\Phi|^2}{\cos(\hat{\theta} )+\left(\frac{4\pi}{\sigma}-t|\Phi|^2\right)\sin(\hat{\theta})}\\
    &=|\nabla \Phi|^2\left( \frac{\cos(\hat{\theta}) + (\frac{4\pi}{\sigma} -\frac{t}{2}|\Phi|^2)\sin(\hat{\theta})}{\cos(\hat{\theta}) + (\frac{4\pi}{\sigma}-t|\Phi|^2)\sin(\hat{\theta})}\right)\\
    &+|\Phi|^2\frac{\left(2\cos(\hat{\theta})\left(\cos(\hat{\theta}) +\frac{2\pi}{\sigma}\sin(\hat{\theta})\right) + 2t\sin(\hat{\theta})\left(\sin(\hat{\theta})- \frac{2\pi}{\sigma}\cos(\hat{\theta})\right)\right)|\Phi|^2-\frac{4\pi}{\sigma}}{\left(\cos(\hat{\theta} )+\left(\frac{4\pi}{\sigma}-t|\Phi|^2\right)\sin(\hat{\theta})\right)\left(\cos(\hat{\theta})+\sin(\hat{\theta})t|\Phi|^2\right)}\\
    \end{aligned}
    \] 
  Since $\cos(\htheta)>0$ and $\inf_{X}|\Phi|^2 \leq \frac{4\pi}{\sigma t}$ we know that $\cos(\hat{\theta})+\left(\frac{4\pi}{\sigma}-t|\Phi|^2\right)\sin(\hat{\theta}) >0$ somewhere on $X$ and hence by Lemma~\ref{lem: H2CurvLem} it must be positive everywhere on $X$.  Therefore
  \[
  \left( \frac{\cos(\hat{\theta}) + (\frac{4\pi}{\sigma} -\frac{t}{2}|\Phi|^2)\sin(\hat{\theta})}{\cos(\hat{\theta}) + (\frac{4\pi}{\sigma}-t|\Phi|^2)\sin(\hat{\theta})}\right)> 0.
  \]
  Furthermore, we have that
    \[
    f(\htheta,t):= \left(2\cos(\hat{\theta})\left(\cos(\hat{\theta}) +\frac{2\pi}{\sigma}\sin(\hat{\theta})\right) + 2t\sin(\hat{\theta})\left(\sin(\hat{\theta})- \frac{2\pi}{\sigma}\cos(\hat{\theta})\right)\right)
    \]
    satisfies
    \[
    f(\htheta,t)\geq \begin{cases} 2 & \text{ if } \tan(\htheta) \leq \frac{2\pi}{\sigma}\\
    2\cos(\htheta)(\cos(\htheta)+\frac{2\pi}{\sigma}\sin(\htheta) & \text { if } \tan(\htheta) > \frac{2\pi}{\sigma}.
    \end{cases}
    \]
 The maximum principle now leads to the desired estimate for $\sup_{X}|\Phi|^2$.  Finally, we bound
 \[
 \cos(\htheta) + (\frac{4\pi}{\sigma}-t|\Phi|^2)\sin(\htheta) \geq \cos(\htheta) +\frac{4\pi}{\sigma} \left(1 -\frac{t}{ f(\htheta,t)}\right)\sin(\htheta)
 \]
 Using calculus one shows that $\frac{t}{f(\htheta,t)}$ is maximized when $t=1$, and $f(\htheta,1)=2$.  The final claim follows.
\end{proof}

As a consequence we obtain the following simplification, which is special to the case $\cos(\htheta)>0$; the calibration constraint~\eqref{eq: dimReduceDHYM-Pos} is a consequence of the equation~\eqref{eq: dimReduceDHYM} provided $\inf_{X}|\Phi|^2 < \frac{4\pi}{\sigma}$.  This is very convenient for the analysis as we do not need to control the calibration constraint along the method of continuity.

\begin{lem}\label{lem: positivityPreservedAlongMOC}
    If $\cos(\htheta)>0$ and $(h_1,h_2)$ solve~\eqref{eq: dimReduceDHYM}, with $\inf_{X}|\Phi|^2 < \frac{4\pi}{\sigma}$ then~\eqref{eq: dimReduceDHYM-Pos} holds.
\end{lem}
\begin{proof}
    Recall that by Lemma~\ref{lem: simplifiedPos}, when ~\eqref{eq: dimReduceDHYM} holds, the calibration constraint \eqref{eq: dimReduceDHYM-Pos} simplifies to
    \[
    \begin{aligned}
        i\Lambda F_1 + \Phi\Phi^{\dagger}&>0\\
        i\Lambda F_2 + \frac{4\pi}{\sigma} - |\Phi|^2 &>0
    \end{aligned}
    \]
    By Lemma~\ref{lem: H2CurvLem} we have
    \[
    i\Lambda F_2 + \frac{4\pi}{\sigma} - |\Phi|^2= \sin(\htheta)\frac{ (\frac{4\pi}{\sigma}-|\Phi|^2)^2 + (1+\frac{1}{2}|\nabla \Phi|^2)}{\cos(\htheta)+(\frac{4\pi}{\sigma} -|\Phi|^2)\sin(\htheta)}.
    \]
    Thanks to Lemma~\ref{lem: conditionalPhiBound} part $(ii)$ and the positivity of $\cos(\htheta), \sin(\htheta)$, this expression is plainly positive.  To analyze the positivity of the first expression we use the linear operator $T_{\Phi,1}:= T_{\Phi}$ introduced in equation~\eqref{eq: TPhiOperator}.  Since $T_{\Phi}$ preserves the space of Hermitian, positive definite endomorphisms, it suffices to prove that $T_{\Phi}(i\Lambda F_1 +\Phi\Phi^{\dagger})>0$.  Writing ~\eqref{eq: dimReduceDHYM-MOC} in terms of $T_{\Phi}$ as in~\eqref{eq: dimReduceMOCUsingToperator} we have
    \[
    \begin{aligned}
    T_{\Phi}(i\Lambda F_1)&= \sin(\htheta)I-\cos(\htheta)\Phi\Phi^{\dagger} + \frac{1}{2}\sin(\htheta)\nabla \Phi (\nabla \Phi)^{\dagger}\\
    T_{\Phi}(\Phi\Phi^{\dagger}) &=  \cos(\htheta)\Phi\Phi^{\dagger} + \sin(\htheta)|\Phi|^2\Phi\Phi^{\dagger}
    \end{aligned}
    \]
   Observing that the only negative term, $\cos(\htheta)\Phi\Phi^{\dagger}$ cancels in the sum, we get the desired positivity.
\end{proof}

\begin{rk}
    Throughout the estimates of this section we shall need to assume $\inf_{X}|\Phi|^2 < \frac{4\pi}{t\sigma}$ in order to apply Lemma~\ref{lem: conditionalPhiBound}. This is a harmless assumption, as it clearly holds when $t=0$ (or if $\Phi$ has a zero), and is preserved along the method of continuity by the conclusion of Lemma~\ref{lem: conditionalPhiBound}.
\end{rk}

An immediate corollary of Lemmas~\ref{lem: H1CurvLem} and~\ref{lem: H2CurvLem} and Lemma~\ref{lem: conditionalPhiBound} is the following bound for the curvature of $E_1\otimes E_2^{\vee}$, which will play a key role in obtaining the estimates of $|\nabla \Phi|$ below.

\begin{cor}\label{cor: curvUpperBound}
Assume that $\cos(\hat{\theta})>0$ and $(h_1,h_2)$ solve ~\eqref{eq: dimReduceDHYM-MOC} with $\inf_{X}|\Phi|^2 <\frac{4\pi}{t\sigma}$.  Then there is a uniform constant $C$ so that the following bound holds for the curvature of the metric $h_1h_2^{-1}$ on $E_1\otimes E_2^{\vee}$
\[
-C(1+|\nabla \Phi|^2) \leq \langle i\Lambda F \Phi, \Phi \rangle \leq C + \frac{1}{2}|\nabla \Phi|^2.
\]
\end{cor}
\begin{proof}
We have
\[
i\Lambda F= i\Lambda F_{1}\otimes I_{E_2^{\vee}} - I_{E_1} \otimes i\Lambda F_{h_2},
\]
and so, by Lemmas~\ref{lem: H1CurvLem} and~\ref{lem: H2CurvLem} we have
\[
\begin{aligned}
\langle i\Lambda F \Phi, \Phi \rangle &=\frac{-\cos(\hat{\theta})|\Phi|^4 +\sin(\hat{\theta})\left( |\Phi|^2 + \frac{t}{2}|\langle \nabla\Phi, \Phi \rangle|^2 \right) }{\left( \cos(\hat{\theta})+\sin(\hat{\theta})t|\Phi|^2\right)}\\
&\quad -\left(\frac{ -\cos(\hat{\theta})\left(\frac{4\pi}{\sigma} - |\Phi|^2\right) +\sin(\hat{\theta})\left(1 + \frac{t}{2}|\nabla \Phi|^2\right)}{\left(\cos(\hat{\theta}) + \left(\frac{4\pi}{\sigma}-t|\Phi|^2\right)\sin(\hat{\theta})\right)} \right) |\Phi|^2
\end{aligned}
\]
Now by Lemma~\ref{lem: conditionalPhiBound} the denominators are bounded uniformly away from zero and $|\Phi|$ is bounded uniformly from above, so we only need to consider the gradient terms;
\[
\frac{t}{2}\sin(\hat{\theta})\frac{|\langle \nabla \Phi, \Phi \rangle|^2}{\cos(\hat{\theta})+t\sin(\hat{\theta})|\Phi|^2} - \frac{t}{2}\sin(\hat{\theta})\frac{|\nabla \Phi|^2|\Phi|^2}{\cos(\hat{\theta}) + \left(\frac{4\pi}{\sigma} -t|\Phi|^2\right)\sin(\hat{\theta})}.
\]
The lower bound is clear.  For the upper bound, we observe that second term is negative along the continuity path~\eqref{eq: dimReduceDHYM-MOC}, and hence it suffices to bound the first term from above.  Since $\cos(\htheta)>0$ Cauchy-Schwarz yields
\[
\frac{t}{2}\sin(\hat{\theta})\frac{|\langle \nabla \Phi, \Phi \rangle|^2}{\cos(\hat{\theta})+t\sin(\hat{\theta})|\Phi|^2} < \frac{1}{2}|\nabla\Phi|^2.
\]
\end{proof}

Our next goal is to prove a $C^1$ bound for $\Phi$.  
\begin{prop}\label{prop: gradientBound}
Suppose that $\cos(\hat{\theta})>0$, and that $(h_1, h_2)$ solve ~\eqref{eq: dimReduceDHYM-MOC} for some $t\in[0,1]$ with $\inf_{X} |\Phi|^2 < \frac{4\pi}{t\sigma}$.  There is a uniform constant $C>0$, depending only on $\hat{\theta}, \sigma$ and $(X,g)$ so that
\[
\sup_{X}|\nabla \Phi|^2 \leq C.
\]
\end{prop}

\begin{proof}
To obtain the bound for $|\nabla\Phi|$, we argue by the maximum principle. Consider the quantity
\[
Q := \log |\nabla \Phi|^2 + H(|\Phi|^2)
\]
where $H(y) : \mathbb{R}_{\geq 0} \rightarrow \mathbb{R}$ is an increasing function which will be determined in the course of the proof.  Suppose that $x_*\in X$ is the point where $Q$ achieves its maximum and assume, without loss of generality, that $|\nabla \Phi|(x_*) >1$.  We fix coordinates near $x_*$ so that $g(x_*) = dz\otimes d\bar{z}$.  We compute
\[
\begin{aligned}
\Delta |\nabla \Phi|^2 &= g^{p\bar{j}}\del_p \left(g^{k\bar{\ell}}\langle \nabla_{\bar{j}} \nabla_k \Phi, \nabla_{\ell} \Phi \rangle + g^{k\bar{\ell}}\langle \nabla_k \Phi, \nabla_{j} \nabla_{\ell} \Phi\rangle \right)\\
&= -g^{p\bar{j}}\del_p \langle g^{k\bar{\ell}}F_{k\bar{j}} \Phi, \nabla_{\ell} \Phi \rangle +  |\nabla \nabla \Phi|^2 + g^{p\bar{j}}g^{k\bar{\ell}} \langle \nabla_k \Phi, \nabla_{\bar{p}} \nabla_j \nabla_{\ell} \Phi \rangle \\
\end{aligned}
\]
We will simplify the final term.  In the following calculation, $R$ denotes the curvature of the holomorphic cotangent bundle.  We have
\[
\begin{aligned}
g^{p\bar{j}}g^{k\bar{\ell}} \langle \nabla_k \Phi, \nabla_{\bar{p}} \nabla_j \nabla_{\ell} \Phi \rangle  &= g^{p\bar{j}}g^{k\bar{\ell}} \langle \nabla_k \Phi, -(F_{j\bar{p}}+R_{j\bar{p}})\nabla_{\ell} \Phi \rangle  +\langle \nabla_k \Phi, \nabla_j \nabla_{\bar{p}} \nabla_{\ell} \Phi\rangle \\
&= -g^{p\bar{j}}g^{k\bar{\ell}} \langle \nabla_k \Phi, (F_{j\bar{p}}+R_{j\bar{p}})\nabla_{\ell} \Phi \rangle  - g^{p\bar{j}}g^{k\bar{\ell}}\langle \nabla_k \Phi, \nabla_j (F_{\ell \bar{p}}\Phi)\rangle \\
&= -\langle \nabla \Phi, (i\Lambda F+i\Lambda R)\nabla \Phi \rangle  -g^{p\bar{j}}\del_{\bar{j}} \langle \nabla_{k} \Phi,  g^{\ell \bar{k}}F_{\ell\bar{p}}\Phi\rangle\\
&\quad  +  g^{p\bar{j}}g^{k\bar{\ell}}\langle \nabla_{\bar{j}} \nabla_k \Phi,  F_{\ell \bar{p}}\Phi\rangle  \\
&=-\langle \nabla \Phi, (i\Lambda F+i\Lambda R)\nabla \Phi \rangle  -g^{p\bar{j}}\del_{\bar{j}} \langle \nabla_{k} \Phi,  g^{\ell \bar{k}}F_{\ell\bar{p}}\Phi\rangle\\
&\quad  -g^{p\bar{j}}g^{k\bar{\ell}}\langle F_{k\bar{j}} \Phi, F_{\ell \bar{p}}\Phi\rangle.
\end{aligned}
\]
Since $X$ is a Riemann surface, we can write the final expression as
\[
\begin{aligned}
\Delta |\nabla \Phi|^2 &= -2{\rm Re}\left(g^{p\bar{j}}\del_p \langle g^{k\bar{\ell}}F_{k\bar{j}} \Phi, \nabla_{\ell} \Phi \rangle\right) +  |\nabla \nabla \Phi|^2 -|i\Lambda F\Phi|^2\\
&\quad -\langle \nabla \Phi, (i\Lambda F+i\Lambda R)\nabla \Phi \rangle.
\end{aligned}
\]
We now estimate the curvature terms, using Lemma~\ref{lem: H1CurveLemSimplfiedPositiveCos}, Lemma~\ref{lem: H1CurvLem}, Lemma~\ref{lem: H2CurvLem} and Lemma~\ref{lem: conditionalPhiBound}.  For the remainder of the proof we let $C$ denote a constant that depends only $(X,g), \hat{\theta}, \sigma$, but which can increase from line to line.  Since $|\nabla \Phi| \geq 1$ by assumption we have the crude bounds
\[
|i\Lambda F\Phi|^2 \leq C|\nabla \Phi|^4, \qquad 
|\langle i\Lambda F\nabla \Phi, \nabla \Phi\rangle| \leq C|\nabla \Phi|^4
\]
We need to bound the term involving $\nabla \langle \nabla \Phi, i\Lambda F \Phi \rangle$.  To do this we observe that from Lemma~\ref{lem: H1CurveLemSimplfiedPositiveCos} and Lemma~\ref{lem: H2CurvLem} we have
\begin{equation}\label{eq: curvatureCrossTermSpecialForm}
\langle i\Lambda F \Phi, \nabla \Phi \rangle = \langle \Phi, \nabla \Phi \rangle\left(f_1(|\Phi|^2,t) + f_2(|\Phi|^2,t)|\langle \nabla \Phi, \Phi \rangle |^2+f_3(|\Phi|^2,t)|\nabla \Phi|^2\right)
\end{equation}
where $f_1(y,t), f_2(y,t), f_3(y,t)$ are smooth, real-valued functions with bounded derivatives on the range of $|\Phi|^2$, as determined in Lemma~\ref{lem: conditionalPhiBound}.  Write $f_i:=f_i(|\Phi|^2,t)$ for simplicity and denote $f_i' = \frac{\del f_i}{\del y}$, then
\[
\begin{aligned}
\dbar \langle i\Lambda F \Phi, \nabla \Phi \rangle &= \langle \Phi, \nabla \nabla \Phi \rangle \left(f_1 + f_2|\langle \nabla \Phi, \Phi \rangle |^2+f_3|\nabla \Phi|^2\right)\\
&\quad +\langle \Phi, \nabla \Phi \rangle^2 \left(f_1'   + f_2' |\langle \nabla \Phi, \Phi \rangle |^2  + f_3'|\nabla \Phi|^2 \right)\\
&\quad +\langle \Phi, \nabla \Phi \rangle \left(f_2 \left(\langle \dbar \nabla \Phi, \Phi \rangle \langle \Phi, \nabla \Phi \rangle\right)  + f_2|\nabla \Phi|^2 \langle \Phi, \nabla \Phi \rangle + f_2 \langle \nabla \Phi, \Phi \rangle\langle \Phi, \nabla\nabla \Phi \rangle\right)\\
&\quad+ \langle \Phi, \nabla \Phi \rangle f_3 \dbar |\nabla \Phi|^2 \\
\end{aligned}
\]

To simplify the term on the third line we use $\langle \dbar \nabla \Phi, \Phi \rangle=-\langle i\Lambda F \Phi, \Phi \rangle$. To simplify the final term we recall that we are working at a maximum of $Q$ so that
\begin{equation}\label{eq: gradZeroMax}
\dbar|\nabla \Phi|^2 =- |\nabla \Phi|^2H'(|\Phi|^2) \langle \Phi, \nabla \Phi\rangle.
\end{equation}
All together we obtain
\[
\begin{aligned}
    -2{\rm Re}\left(g^{p\bar{j}}\del_p \langle g^{k\bar{\ell}}F_{k\bar{j}} \Phi, \nabla_{\ell} \Phi \rangle\right) &= -2{\rm Re}(\langle \Phi, \nabla \nabla \Phi \rangle)\left(f_1 + f_2|\langle \nabla \Phi, \Phi \rangle |^2+f_3|\nabla \Phi|^2\right)\\
    &\quad -2{\rm Re}(\langle \Phi, \nabla \Phi \rangle^2)\left(f_1'   + f_2' |\langle \nabla \Phi, \Phi \rangle |^2  + f_3'|\nabla \Phi|^2 \right)\\
    &\quad +2f_2\left({\rm Re}\langle \Phi, \nabla \Phi \rangle\right)^2\langle i\Lambda F\Phi, \Phi \rangle\\
    &\quad -2f_2|\nabla\Phi|^2{\rm Re}(\langle \Phi, \nabla\Phi\rangle^2)-2f_2|\langle \Phi, \nabla \Phi \rangle|^2{\rm Re}\left(\langle \Phi, \nabla \nabla \Phi \rangle \right)\\
    &\quad +2f_3|\nabla \Phi|^2H' {\rm Re}\left(\langle \Phi, \nabla \Phi\rangle^2\right)\\
    & \geq -C(|\nabla \nabla \Phi||\nabla \Phi|^2 + |\nabla \Phi|^4)-C|\nabla \Phi|^2|\langle \Phi, \nabla \Phi\rangle|^2H'\\
    &\geq -\epsilon|\nabla \nabla \Phi|^2 -C\epsilon^{-1}|\nabla \Phi|^4 -C|\nabla \Phi|^2|\langle \Phi, \nabla \Phi\rangle|^2H'
\end{aligned}
\]

for any $0<\epsilon<1$.  In total, we have that, at $x_*$ there holds
\[
\begin{aligned}
\Delta \log|\nabla \Phi|^2 &\geq \frac{(1-\epsilon)|\nabla \nabla \Phi|^2}{|\nabla \Phi|^2} - \frac{|\nabla|\nabla \Phi|^2|^2}{|\nabla \Phi|^4}  - C\epsilon^{-1}|\nabla \Phi|^2 -C|\langle \Phi, \nabla \Phi\rangle|^2H'\\
& = \frac{(1-\epsilon)|\nabla \nabla \Phi|^2}{|\nabla \Phi|^2} - (H')^2|\langle \nabla \Phi ,\Phi \rangle|^2 - C\epsilon^{-1}|\nabla \Phi|^2 -C|\langle \Phi, \nabla \Phi\rangle|^2H'
\end{aligned}
\]
where we used ~\eqref{eq: gradZeroMax} in passing from the first line to the second line.  We need to further analyze the resulting term involving $(H')^2$.  Expanding~\eqref{eq: gradZeroMax} we have
\[
- H'(|\Phi|^2) \langle \Phi, \nabla \Phi\rangle = \frac{\dbar |\nabla \Phi|^2}{|\nabla \Phi|^2} =\frac{- \langle i\Lambda F \Phi, \nabla \Phi \rangle + \langle \nabla \Phi, \nabla \nabla \Phi \rangle}{|\nabla \Phi|^2}
\]
We take the norm of both sides, and use~\eqref{eq: curvatureCrossTermSpecialForm} (recalling that $|\nabla \Phi|(x_*) \geq 1$) and Lemma~\ref{lem: conditionalPhiBound} to bound $|\langle i\Lambda F\Phi, \nabla \Phi\rangle| \leq C|\nabla \Phi|^2|\langle \Phi, \nabla \Phi\rangle|$, which leads to the estimate 
\[
\begin{aligned}
(H')|\langle \Phi, \nabla \Phi \rangle| &\leq \frac{|\langle i\Lambda F\Phi, \nabla \Phi\rangle| + |\nabla \Phi||\nabla \nabla \Phi|}{|\nabla \Phi|^2}\\
& \leq C|\langle\Phi, \nabla \Phi\rangle| + \frac{|\nabla \nabla \Phi|}{|\nabla \Phi|}.
\end{aligned}
\]
Therefore, for any $\delta>0$ we have
\[
(H')^2|\langle \Phi, \nabla \Phi \rangle|^2 \leq C^2(1+\delta^{-1})|\langle\Phi, \nabla \Phi\rangle|^2 +  (1+\delta)\frac{|\nabla \nabla \Phi|^2}{|\nabla \Phi|^2}.
\]
and so, for any $\gamma \in (0,1)$
\[
\begin{aligned}
(H')^2|\langle \Phi, \nabla \Phi \rangle|^2 &\leq \gamma(H')^2|\langle \Phi, \nabla \Phi \rangle|^2  + C^2(1-\gamma)(1+\delta^{-1})|\langle\Phi, \nabla \Phi\rangle|^2 \\
&\quad +  (1-\gamma)(1+\delta)\frac{|\nabla \nabla \Phi|^2}{|\nabla \Phi|^2}.
\end{aligned}
\]
Choose $\gamma =2\epsilon$ and $\delta =\epsilon$ so that $(1-\gamma)(1+\delta)<1-\epsilon$. We get the following bound, valid for any $0<\epsilon <1$:
\[
(1-\epsilon)\frac{|\nabla \nabla \Phi|^2}{|\nabla \Phi|^2} -(H')^2|\langle \Phi, \nabla \Phi \rangle|^2 \geq -C\epsilon^{-1}|\langle \Phi, \nabla \Phi \rangle|^2 - 2\epsilon(H')^2|\langle \Phi, \nabla \Phi \rangle|^2.
\]
Summarizing, we have obtained the estimate at $x_*$:
\[
\Delta \log|\nabla \Phi|^2 \geq - 2\epsilon(H')^2|\langle \nabla \Phi ,\Phi \rangle|^2 - C\epsilon^{-1}|\nabla \Phi|^2 -C|\langle \Phi, \nabla \Phi\rangle|^2H'.
\]
It only remains to compute the Laplacian of $H(|\Phi|^2)$.  This is straightforward
\[
\begin{aligned}
\Delta H(|\Phi|^2) &= g^{p\bar{j}} \del_p \left(H'(|\Phi|^2) \langle \Phi, \nabla_j \Phi \rangle\right)\\
&=H'' |\langle \nabla \Phi, \Phi\rangle|^2+H' |\nabla \Phi|^2 + H'g^{p\bar{j}} \langle \Phi, \nabla_{\bar{p}}\nabla_{j} \Phi\rangle\\
&=H'' |\langle \nabla \Phi, \Phi\rangle|^2+H' |\nabla \Phi|^2 - H' \langle \Phi, i\Lambda F \Phi\rangle 
\end{aligned}
\]
Recall that by Corollary~\ref{cor: curvUpperBound} we have $\langle \Phi, i\Lambda F \Phi\rangle \leq C + \frac{1}{2}|\nabla \Phi|^2$ 
for a uniform constant $C$, and so, since $H'>0$ we get
\[
\begin{aligned}
0 \geq \Delta Q(x_*) &\geq - 2\epsilon(H')^2|\langle \nabla \Phi ,\Phi \rangle|^2  - C\epsilon^{-1}|\nabla \Phi|^2 -C|\langle \Phi, \nabla \Phi\rangle|^2H'\\
&\quad +H'' |\langle \nabla \Phi, \Phi\rangle|^2+H'\left( |\nabla \Phi|^2 -(\frac{1}{2}|\nabla \Phi|^2 +C)\right)
\end{aligned}
\]
Choose $H$ so that $ H'>8C$, and take $2\epsilon = \frac{8C}{H'} \in(0,1)$.  Then we have 
\[
\begin{aligned}
0 \geq \Delta Q &\geq - 8C H'|\langle \nabla \Phi ,\Phi \rangle|^2  - \frac{1}{4}H'|\nabla \Phi|^2 -C|\langle \Phi, \nabla \Phi\rangle|^2H'\\
&\quad +H'' |\langle \nabla \Phi, \Phi\rangle|^2+H'\left( \frac{1}{2}|\nabla \Phi|^2 -C\right).
\end{aligned}
\]
Now suppose that we have chosen $H$ such that
\[
\begin{aligned}
H' &\geq 8C +2\\
H''&\geq 10CH'.
\end{aligned}
\]
in particular we can take $H(y)=e^{Ay}$ for $A$ large depending only on $C$.  Then at $x_*$ we have
\[
0 \geq C|\langle \nabla \Phi, \Phi \rangle |^2 + (\frac{1}{4}|\nabla \Phi|^2-C)  
\]
which yields $|\nabla \Phi|^2(x_*) \leq 4C$.  Hence $Q$ is uniformly bounded, and the bound for $|\nabla \Phi|$ follows.
\end{proof}

Finally, we have obtained the following bound:

\begin{cor}\label{cor: totalCurvBound}
    Suppose that $\cos(\hat{\theta})>0$, and that $(h_1, h_2)$ solve ~\eqref{eq: dimReduceDHYM-MOC} for some $t\in[0,1]$ with $\inf_{X}|\Phi|^2 < \frac{4\pi}{\sigma t}$.  There is a uniform constant $C>0$, depending only on $\hat{\theta}, \sigma$ and $(X,g)$ so that
    \[
    -CI \leq i\Lambda F_{1} \leq CI \quad \text{ and } -C \leq i\Lambda F_2 \leq C.
    \]
\end{cor}
\begin{proof}
    The bound for $F_1$ follows from Lemma~\ref{lem: H1CurveLemSimplfiedPositiveCos} together with the bounds for $|\Phi|$ and $|\nabla \Phi|$ obtained in Lemma~\ref{lem: conditionalPhiBound} and Proposition~\ref{prop: gradientBound} respectively.  Similarly, the bound for $F_2$ follows from the formula in Lemma~\ref{lem: H2CurvLem} together with the estimates in  Lemma~\ref{lem: conditionalPhiBound} and Proposition~\ref{prop: gradientBound}.
\end{proof}

\subsection{Estimates from stability}\label{subsec: EstFromStab}

We now prove upper and lower bounds for $h_1, h_2$ using stability.  This requires two arguments by contradiction. 

\begin{prop}\label{prop: UYargumentH}
Suppose that the triple $\cE= (E_1, E_2, \Phi)$ has $\cos(\htheta)>0$, $\tan(\htheta)\leq \frac{\sigma}{4\pi}$ and is stable in the sense of Definition~\ref{def: stabOnXxP} or, equivalently, Definition~\ref{defn: stabOfTriple}.  Suppose the $(h_1, h_2)$ are smooth solutions of~\eqref{eq: dimReduceDHYM-MOC} for $t\in [0, 1]$ satisfying $\inf_{X} |\Phi|^2 < \frac{4\pi}{\sigma t}$.  There is a uniform constant $C>0$ so that
\[
 h_t < Ch_0
\]
where $h_{t}= h_{1,t}h_{2,t}^{-1}$ is the metric on ${\rm Hom}(E_2, E_1)$.
\end{prop}
\begin{proof}
    Since the proof is quite long, we begin by outlining the idea. Arguing by contradiction, following the ideas of  Uhlenbeck-Yau \cite{UY}, we show that if there exists a sequence of metrics $h_{\ell}$ satisfying the assumptions but for which the bound $h_{0}^{-1}h_{\ell}$ blows up at some point of $X$, then one can produce a holomorphic subbundle $\cS\subset \cE$ of type A in the sense of Definition~\ref{defn: subObTypeA/B}.  To obtain a contradiction, it suffices to show that the subbundle $\cS$ destabilizes $\cE$.  This requires a rather intricate argument depending on some subtle cancellation together with integration by parts.  One of the main difficulties is that, while $|\Phi|_{h_{\ell}}$ and $|\nabla^{\ell}\Phi|_{h_{\ell}}$ are bounded independent of $\ell$, this does not imply the convergence to zero of $\Phi^{{\dagger}_{\ell}}$ or $(\nabla^{\ell}\Phi)^{\dagger_{\ell}}$.

    Let $h_{1,t}, h_{2,t}$ be metrics on $E_1$ and $E_2$ respectively solving~\eqref{eq: dimReduceDHYM-MOC}, and let $h_t= h_{1,t}h_{2,t}^{-1}$ be the metric on $E_1\otimes E_2^{\vee}$.  Denote by $F_{1,t}, F_{2,t}$ and $F_{t}$
    the curvatures of $h_{1,t}, h_{2,t}$ and $h_{t}= h_{1,t}h_{2,t}^{-1}$, respectively.  By Lemma~\ref{lem: conditionalPhiBound}, and Proposition~\ref{prop: gradientBound} each of these curvatures is uniformly bounded. Define the relative endomorphism of $E_1$ given by $u:=h_0^{-1}h_t$, which is positive definite and Hermitian with respect to both $h_0$ and $h_t$.  From now on we let $\hat{\nabla}$ denote the $(1,0)$ part of the Chern connections induced by $h_{1,0}, h_{2,0}$, and denote by $\nabla^t$ the Chern connections of $h_{1,t}, h_{2,t}$.  Then we have
    \begin{equation}\label{eq: curvatureRelation}
   i\Lambda \dbar(u^{-1}\hat{\nabla} u) = i\Lambda F_{0} - i\Lambda F_{t} = i\Lambda \dbar\left((\nabla^tu)u^{-1}\right) 
    \end{equation}
    
    Since $u$ is positive definite, for any $\mathtt{s} \in (0,1]$ we can define $u^{\mathtt{s}}$.  We have the following key formulas, due to Uhlenbeck-Yau \cite{UY} 
    \begin{equation}\label{eq: UYFormula1}
     \langle u^{-1}\hat{\nabla} u, \hat{\nabla}u^{\mathtt{s}}\rangle_{h_0} \geq |u^{-\mathtt{s}/2}\hat{\nabla}u^{\mathtt{s}}|^2_{h_0}
    \end{equation}
    \begin{equation}\label{eq: UYFormula2}
    \langle (\nabla^tu)u^{-1},\nabla^tu^{\mathtt{s}}\rangle_{h_t} \geq |(\nabla^t u^{\mathtt{s}})u^{-\mathtt{s}/2}|_{h_t}^2
    \end{equation}
    \begin{equation}\label{eq: UYFormula3}
    g^{j\bar{k}}\del_{\bar{k}}\langle u^{-1}\hat{\nabla}_j u, u^{\mathtt{s}} \rangle_{h_0} = \frac{1}{\mathtt{s}}\Delta {\rm Tr}(u^{\mathtt{s}}) = g^{j\bar{k}}\del_{\bar{k}}\langle u^{-1}\nabla^t_j u, u^{\mathtt{s}} \rangle_{h_t}
    \end{equation}
    Equations~\eqref{eq: UYFormula1} and ~\eqref{eq: UYFormula3} appear in \cite[Lemma 4.1]{UY}, while~\eqref{eq: UYFormula2} is essentially identical to~\eqref{eq: UYFormula1}. Since we will need to exploit some intermediate formulas obtained by Uhlenbeck-Yau, we shall recall the proof of~\eqref{eq: UYFormula1} and~\eqref{eq: UYFormula2} in Lemma~\ref{lem: UYFormulaImprovement} below. Combining these formulas we get
\[
\begin{aligned}
   \frac{1}{\mathtt{s}}\Delta {\rm Tr}(u^{\mathtt{s}})&\geq g^{j\bar{k}} \langle \del_{\bar{k}}(u^{-1}\hat{\nabla}_j u), u^{\mathtt{s}}\rangle_{h_0} +|u^{-\mathtt{s}/2}\hat{\nabla}u^{\mathtt{s}}|^2_{h_0}\\
   &= {\rm Tr}\left( \left(i\Lambda F_{0}-i\Lambda F_{t}\right) u^{\mathtt{s}}\right) +|u^{-\mathtt{s}/2}\hat{\nabla}u^{\mathtt{s}}|^2_{h_0} 
\end{aligned}
\]
    Let $m$ denote the maximum of the largest eigenvalue of $u$ and  set $\tilde{u}^{\mathtt{s}} =m^{-\mathtt{s}}u^{\mathtt{s}}$.  Then
    \begin{equation}\label{eq: LaplaceUsigmaLB}
    \begin{aligned}
    \frac{1}{\mathtt{s}}\Delta {\rm Tr}(\tilde{u}^{\mathtt{s}})&\geq {\rm Tr}\left( \left(i\Lambda F_{0}-i\Lambda F_{t}\right) \widetilde{u}^{\mathtt{s}}\right) + |\tilde{u}^{-\mathtt{s}/2}\hat{\nabla}\tilde{u}^{\mathtt{s}}|^2_{h_0}
   \end{aligned}
    \end{equation}
    Integrating this inequality over $X$, and using the uniform bound for the curvature, together with the fact that $\tilde{u}^{\mathtt{s}}\leq I$ we obtain
    \begin{equation}\label{eq: w12Bound}
\int_{X}|\hat{\nabla}\tilde{u}^{\mathtt{s}}|^2_{h_0} \leq C
    \end{equation}
    and hence $\tilde{u}^{\mathtt{s}}$ is uniformly bounded in $W^{1,2}({\rm End}(E_1\otimes E_2^{\vee}) , h_{0})$, henceforth denoted simply by $W^{1,2}$.  

    Suppose, for the sake of contradiction, that there exists a sequence $t_{\ell} \rightarrow T$ such that $u_\ell:= h_0^{-1}h_{t_\ell}$ has  $\lim_{\ell\rightarrow \infty}\sup_{X} {\rm Tr}(u_\ell) =+\infty$. Defining $\tilde{u}_\ell^{\mathtt{s}}$ as above, we deduce from~\eqref{eq: w12Bound} that, for any $\mathtt{s} \in (0,1]$, the sequence $\tilde{u}_\ell^{\mathtt{s}}$ is uniformly bounded in $W^{1,2}$.

    Fixing $\mathtt{s} =1$, we may take a subsequence (not relabeled) converging weakly in $W^{1,2}$, strongly in $L^2$, and pointwise a.e. to a limit $\tilde{u}_{\infty}$.  By the $C^0$ bound obtained in Lemma~\ref{lem: conditionalPhiBound} we have that $\tilde{u}_{\infty} \Phi =0 $ a.e.  We claim that for every $\mathtt{s} \in (0,1]$ the sequence $\lim_{\ell\rightarrow \infty}\widetilde{u}_\ell^{\mathtt{s}}= \tilde{u}_{\infty}^{\mathtt{s}}$.  This is clear since $\widetilde{u}_\ell^{\mathtt{s}}$ converges pointwise a.e. to $\tilde{u}_{\infty}^{\mathtt{s}}$, and by~\eqref{eq: w12Bound}, this convergence can be taken weakly in $W^{1,2}$ and strongly in $L^2$.  By lower semi-continuity of the norm along weakly converging sequences we have
    \begin{equation}\label{eq: W12bndUinfinity}
    \int_{X}|\hat{\nabla}\widetilde{u}_{\infty}^{\mathtt{s}}|_{h_0}^2\leq C
    \end{equation}
    We claim that $\widetilde{u}_{\infty}^{\mathtt{s}}$ is not identically zero.  Indeed, by~\eqref{eq: LaplaceUsigmaLB} we have $ \Delta{\rm Tr}(\widetilde{u}_\ell^{\mathtt{s}}) \geq -C $  and by definition of $\widetilde{u}_\ell$ we have $\sup_{X}{\rm Tr}(\widetilde{u}_\ell^{\mathtt{s}}) \geq 1$.  It follows easily from standard elliptic theory that
    \begin{equation}\label{eq: uinfinityNotZero}
    \int_{X}{\rm Tr}\widetilde{u}^{\mathtt{s}}_\ell > \delta
    \end{equation}
    for some $\delta>0$ independent of $\ell,\mathtt{s}$.  Since $\widetilde{u}^{\mathtt{s}}_\ell$ converges to $\widetilde{u}^{\mathtt{s}}_{\infty}$ in $L^2$ and pointwise almost everywhere, the claim follows.  

    Finally, we take a limit as $\mathtt{s} \rightarrow 0$.  Again~\eqref{eq: W12bndUinfinity} implies that $\lim_{\mathtt{s} \rightarrow 0}\widetilde{u}_{\infty}^{\mathtt{s}} = p$ weakly in $W^{1,2}$, strongly in $L^2$ and pointwise almost everywhere.  By~\eqref{eq: uinfinityNotZero}, the limit $p$ is not zero.  Let $\pi= I-p$.  We claim that $\pi^2=\pi$ almost everywhere.  To see this, we compute
    \[
    \pi^2 = (I-p)^2 =\lim_{\mathtt{s} \rightarrow 0} (I-\widetilde{u}_{\infty}^{\mathtt{s}})^2 =\lim_{\mathtt{s} \rightarrow 0} I-2\widetilde{u}_{\infty}^{\mathtt{s}} + \widetilde{u}_{\infty}^{2\mathtt{s}} = I-p= \pi
    \]
    Thus, $\pi$ is a $W^{1,2}$ projection.  A fundamental result of Uhlenbeck-Yau \cite{UY} says that $\pi$ defines a proper holomorphic subbundle $\pi(E_1):=S \subsetneq E_1$, and by the $C^0$ bound Lemma~\ref{lem: conditionalPhiBound}, $\Phi \in H^{0}(X,S\otimes E_{2}^{\vee})$.  The remainder of the argument is to show that the triple $(S,E_2,\Phi)$ destabilizes $(E_1, E_2,\Phi)$.  Consider the exact sequence
    \[
    0 \rightarrow S \rightarrow E_1 \rightarrow Q \rightarrow 0.
    \]
    Let $\beta_0 \in \Lambda^{0,1}\otimes S\otimes Q^{\vee}$ denote the second fundamental form of $S \subset E_1$, defined by $h_0$, and denote $h_{1, t_{\ell}}= h_{1,\ell}$ and $h_{2, t_{\ell}}= h_{2,\ell}$  for simplicity, and write $F_{1,\ell}, F_{2,\ell}$ for the curvatures of $h_{1,\ell}$ and $h_{2,\ell}$ respectively.  We begin by computing
    \begin{equation}\label{eq: UYtracePiFirstStep}
    \begin{aligned}
    \int_{X}{\rm Tr}(i\Lambda F_{1,0}\pi^{\perp}) &= 2\pi \deg(Q) -\int_{X}|\beta_0|_{h_0}^2\\
    \end{aligned}
    \end{equation}
    Since $i\Lambda F_{1,0} +\Phi\Phi^{\dagger_0} = \tan(\htheta)I>0$  we have
    \[
    0 < \tan(\htheta){\rm rk}(Q)V_{X} = \int_{X}{\rm Tr}(i\Lambda F_{1,0}\pi^{\perp})=2\pi \deg(Q) -\int_{X}|\beta_0|_{h_0}^2
    \]
    since $\pi^{\perp}\Phi=0$.  This implies $\deg(Q)>0$ and hence the holomorphic triple $\cQ= (Q,0,0)$ has ${\rm Im}(Z_{\cX}(\cQ))>0$.

From ~\eqref{eq: UYtracePiFirstStep} we have
\[
\begin{aligned}
\int_{X}{\rm Tr}(i\Lambda F_{1,0} \pi^{\perp}) &= \lim_{\mathtt{s} \rightarrow 0}\lim_{\ell\rightarrow \infty}\int_{X}{\rm Tr}(i\Lambda F_{1,0} \tilde{u}_{\ell}^{\mathtt{s}})
\end{aligned}
\]
Then, by~\eqref{eq: curvatureRelation} we have
\[
\begin{aligned}
    \int_{X}{\rm Tr}(i\Lambda F_{1,0} \widetilde{u}_{\ell}^{\mathtt{s}}) - \int_{X}i\Lambda F_{2,0} {\rm Tr}(\widetilde{u}_{\ell}^s)&= \int_{X}{\rm Tr}(i\Lambda F_{1,\ell} \widetilde{u}_{\ell}^{\mathtt{s}})- \int_{X}i\Lambda F_{2,\ell} {\rm Tr}(\widetilde{u}_{\ell}^s) \\
    &\quad + \int {\rm Tr} (\dbar (\widetilde{u}_\ell^{-1} \hat{\nabla} \widetilde{u}_\ell) \widetilde{u}_{\ell}^{\mathtt{s}})\\
    &= \int_{X}{\rm Tr}(i\Lambda F_{1,\ell} \widetilde{u}_{\ell}^{\mathtt{s}}) -\int_{X}i\Lambda F_{2,\ell} {\rm Tr}(\widetilde{u}_{\ell}^s)\\
    &\quad- \int {\rm Tr} ( (\widetilde{u}^{-1}_{\ell} \hat{\nabla} \widetilde{u}_\ell) \dbar(\widetilde{u}_{\ell}^{\mathtt{s}})).\\
\end{aligned}
\]
Now we use~\eqref{eq: UYFormula1}
\[
\int {\rm Tr} ( (\widetilde{u}^{-1}_\ell \hat{\nabla} \widetilde{u}_\ell) \dbar(\widetilde{u}_{\ell}^{\mathtt{s}})) \geq \int_{X}|\widetilde{u}_\ell^{-\mathtt{s}/2}\hat{\nabla} \widetilde{u}_\ell^{\mathtt{s}}|_{h_0}^2 \geq \int_{X}|\hat{\nabla}\widetilde{u}_\ell^{\mathtt{s}}|_{h_0}^2.
\] 
By lower semi-continuity of the norm along weakly converging sequences we have
\[
\lim_{\mathtt{s} \rightarrow 0}\lim_{\ell\rightarrow \infty} \int_{X}|\hat{\nabla}\widetilde{u}_\ell^{\mathtt{s}}|_{h_0}^2 \geq \int|\beta_0|^2_{h_0}.
\]
In particular, we have
\begin{equation}\label{eq: dominateSecondFundForm}
   \lim_{\mathtt{s} \rightarrow 0} \lim_{\ell \rightarrow \infty} \int_{X} {\rm Tr} ( (\widetilde{u}^{-1}_\ell \hat{\nabla} \widetilde{u}_\ell) \dbar(\widetilde{u}_{\ell}^{\mathtt{s}})) \geq \int_{X}|\beta_0|^2_{h_0}
\end{equation}
Summarizing, we have proven that
\[
\begin{aligned}
\cos(\hat{\theta})(2\pi \deg(Q)) &= \cos(\hat{\theta})\lim_{\mathtt{s} \rightarrow 0}\lim_{\ell\rightarrow \infty}\int_{X}{\rm Tr}(i\Lambda F_{1,\ell}\widetilde{u}_\ell^{\mathtt{s}})\\
&\quad +\cos(\htheta)\int_{X}\left(i\Lambda F_{2,\ell} - i\Lambda F_{2,0}\right){\rm Tr}(\widetilde{u}_{\ell}^s)\\
&\quad- \cos(\hat{\theta}) \left(\lim_{\mathtt{s} \rightarrow 0}\lim_{\ell \rightarrow \infty}\int_{X}{\rm Tr}\left((\widetilde{u}_\ell^{-1}\hat{\nabla} \widetilde{u}_\ell) \dbar \widetilde{u}_\ell^{\mathtt{s}})\right) - \int_{X}|\beta_0|^2\right)
\end{aligned}
\]
and by~\eqref{eq: dominateSecondFundForm}, the bracketed term on the third line is non-negative; in particular, the last line is non-positive, since $\cos(\htheta)>0$.  For the term on the second line we have
\[
\int_{X}\left(i\Lambda F_{2,\ell} - i\Lambda F_{2,0}\right){\rm Tr}(\widetilde{u}_{\ell}^s)=\int_{X}\left(i\Lambda F_{2,\ell} - i\Lambda F_{2,0}\right)({\rm Tr}(\widetilde{u}_{\ell}^s)- \rank(Q))
\]
since the difference $\left(i\Lambda F_{2,\ell} - i\Lambda F_{2,0}\right)$ is $\ddbar$-exact.  Then using the curvature bounds together with the strong $L^2$ convergence $\lim_{\mathtt{s}\rightarrow 0}\lim_{\ell \rightarrow \infty}\widetilde{u}_{\ell}^s = (I-\pi) $ we get
\[
\lim_{\mathtt{s}\rightarrow 0}\lim_{\ell \rightarrow \infty}\int_{X}\left(i\Lambda F_{2,\ell} - i\Lambda F_{2,0}\right)({\rm Tr}(\widetilde{u}_{\ell}^s)-\rank(Q))=0
\]

The remainder of the proof is dedicated to analyzing the sign of the remaining terms.  As a first step, we use the equation to obtain some useful formulas.  By~\eqref{eq: dimReduceDHYM-MOC} we have
\[
\begin{aligned}
\cos(\hat{\theta})\left(\int_{X}{\rm Tr}\left(i\Lambda F_{1,\ell}\widetilde{u}_{\ell}^{\mathtt{s}}\right) + \int_{X}\langle \widetilde{u}_{\ell}^{\mathtt{s}} \Phi, \Phi \rangle_{h_{\ell}}\right)
&= \sin(\hat{\theta})\left( \int_{X}{\rm Tr}(\widetilde{u}_{\ell}^{\mathtt{s}}) - \frac{t_{\ell}}{2}\int_{X}{\rm Tr}\left(\{i\Lambda F_{1,\ell}, \Phi \Phi^{\dagger_{\ell}}\} \widetilde{u}_{\ell}^{\mathtt{s}}  \right)\right)\\
&\quad + \sin(\hat{\theta}) \frac{t_{\ell}}{2}\int_{X}{\rm Tr}\left(\left(\nabla^{\ell} \Phi\right)\left(\nabla^{\ell}\Phi\right)^{\dagger_{\ell}} \widetilde{u}_{\ell}^{\mathtt{s}}\right)
\end{aligned}
\]
where, as usual, $\dagger_{\ell}$ denotes the adjoint with respect to $h_{1,\ell}h_{2,\ell}^{-1}$.  Clearly we have
\[
\lim_{\mathtt{s}\rightarrow 0}\lim_{\ell \rightarrow \infty} \int_{X}{\rm Tr}(\widetilde{u}_{\ell}^{\mathtt{s}})= V_{X}\rank(Q)
\]
and so
\begin{equation}\label{eq: stabilityContradiction}
\begin{aligned}
&\cos(\hat{\theta})(2\pi \deg(Q))- \sin(\hat{\theta})V_{X}\rank(Q)\\
&=-\cos(\htheta)\lim_{\mathtt{s}\rightarrow 0}\lim_{\ell \rightarrow \infty}\left(\int_{X}|\widetilde{u}_{\ell}^{-\mathtt{s}/2}\hat{\nabla}\widetilde{u}_{\ell}^{\mathtt{s}}|^2_{h_0}- \int_{X}|\beta_0|_{h_{0}}^2\right)\\
&\quad - \cos(\hat{\theta}) \left(\lim_{\mathtt{s} \rightarrow 0}\lim_{\ell \rightarrow \infty}\int_{X}{\rm Tr}\left((\widetilde{u}_\ell^{-1}\hat{\nabla} \widetilde{u}_\ell) \dbar \widetilde{u}_\ell^{\mathtt{s}})\right) - |\widetilde{u}_{\ell}^{-\mathtt{s}/2}\hat{\nabla}\widetilde{u}_{\ell}^{\mathtt{s}}|^2_{h_0}\right)\\
&\quad -\cos(\hat{\theta}) \lim_{\mathtt{s} \rightarrow 0}\lim_{\ell \rightarrow \infty}\int_{X}\langle \widetilde{u}_{\ell}^{\mathtt{s}} \Phi, \Phi \rangle_{h_{\ell}}\\
&\quad - \lim_{\mathtt{s} \rightarrow 0}\lim_{\ell \rightarrow \infty}\frac{t_{\ell}}{2}\sin(\hat{\theta})\left(\int_{X}{\rm Tr}\left(\{i\Lambda F_{1,\ell}, \Phi \Phi^{\dagger_{\ell}}\} \widetilde{u}_{\ell}^{\mathtt{s}}\right) - {\rm Tr}\left(\left(\nabla^{\ell} \Phi\right)\left(\nabla^{\ell}\Phi\right)^{\dagger_{\ell}} \widetilde{u}_{\ell}^{\mathtt{s}}\right)\right).
\end{aligned}
\end{equation}
By stability we have
\[
\cos(\hat{\theta})(2\pi \deg(Q))- \sin(\hat{\theta})V_{X}\rank(Q)>0
\]
and so, in order to obtain a contradiction, it suffices to show that the expression on the right is non-positive.  The first line on the right hand side of~\eqref{eq: stabilityContradiction} is non-positive by~\eqref{eq: dominateSecondFundForm}. We are therefore led to study the remaining terms, which are
\begin{equation}\label{eq: difficultTermUY}
(I):= \cos(\hat{\theta}){\rm Tr}(\Phi\Phi^{\dagger_{\ell}}\widetilde{u}_{\ell}^{\mathtt{s}}) + \frac{t_{\ell}}{2}\sin(\hat{\theta}){\rm Tr}\left(\{\Phi\Phi^{\dagger_{\ell}},i\Lambda F_{1,\ell}\} \widetilde{u}_{\ell}^{\mathtt{s}}\right) - \frac{t_{\ell}}{2}\sin(\hat{\theta}){\rm Tr}((\nabla^{\ell} \Phi)(\nabla^{\ell} \Phi)^{\dagger_{\ell}}\widetilde{u}_{\ell}^{\mathtt{s}})
\end{equation}
and
\begin{equation}\label{eq: UYexcess}
\mathcal{E}_{UY}={\rm Tr}\left((\widetilde{u}_\ell^{-1}\hat{\nabla} \widetilde{u}_\ell) \dbar \widetilde{u}_\ell^{\mathtt{s}})\right) - |\widetilde{u}_{\ell}^{-\mathtt{s}/2}\hat{\nabla}\widetilde{u}_{\ell}^{\mathtt{s}}|^2_{h_0}.
\end{equation}
We refer to $\mathcal{E}_{UY}$ as the {\em Uhlenbeck-Yau excess}.  To obtain a contradiction it would suffice to show that $(I) + \mathcal{E}_{UY} \geq 0$.  Unfortunately, as one can check using Lemma~\ref{lem: H1CurveLemSimplfiedPositiveCos}, $(I)$ is not pointwise positive.  More precisely, $(I)$ contains the problematic $-(\nabla^{\ell}\Phi)(\nabla^{\ell}\Phi)^{\dagger_{\ell}}$ term which cannot be pointwise controlled by other terms in $(I)$ nor any terms in $\mathcal{E}_{UY}$ (which does not depend on $\Phi$).  Instead, we need to use integration by parts together with some subtle cancellation to obtain a sign on the integral of $(I)$. We use the following simple integration by parts formula which holds for any Hermitian endomorphism $\tau$ of $E_1$.

\begin{equation}\label{eq: integrationByPartsUY}
\begin{aligned}
&\int_{X} \frac{1}{2}{\rm Tr}\left(\{ i\Lambda F_{1,\ell},\Phi\Phi^{\dagger}_{\ell}\} \tau\right) - \int_{X}{\rm Tr}\left((\nabla^{\ell}\Phi)( \nabla^{\ell}\Phi)^{\dagger_{\ell}}\tau\right)\\
&= \int_{X}{\rm Re}\left(\langle \nabla^{\ell}\Phi, (\nabla^{\ell}\tau) \Phi \rangle_{h_{\ell}}\right) + \int_{X}i\Lambda F_{2,\ell}\langle \Phi, \tau \Phi \rangle_{h_{\ell}}
\end{aligned}
\end{equation}
To see this we compute, using that $\Phi$ is holomorphic
\[
\begin{aligned}
\Delta \langle \Phi, \tau \Phi \rangle_{h_{\ell}} &= \del \langle \Phi, \nabla^{\ell} (\tau \Phi) \rangle_{h_{\ell}}\\
&= \langle \nabla^{\ell} \Phi, \nabla^{\ell} (\tau \Phi) \rangle_{h_{\ell}} + \langle \Phi, \dbar \nabla^{\ell} (\tau \Phi)\rangle_{h_{\ell}}\\
&= \langle \nabla^{\ell} \Phi, \tau \nabla^{\ell} \Phi \rangle_{h_{\ell}} + \langle \nabla^{\ell} \Phi, (\nabla^{\ell} \tau) \Phi \rangle_{h_{\ell}} -\langle \Phi, i\Lambda F_{1,\ell} (\tau \Phi)\rangle_{h_{\ell}}\\
&\quad + i\Lambda F_{2,\ell} \langle \Phi, \tau \Phi \rangle_{h_{\ell}} + \langle \Phi, \nabla^{\ell} \dbar (\tau \Phi) \rangle_{h_{\ell}}\\
&= \langle \nabla^{\ell} \Phi, \tau \nabla^{\ell} \Phi \rangle_{h_{\ell}} + \langle \nabla^{\ell} \Phi, (\nabla^{\ell} \tau) \Phi \rangle_{h_{\ell}} -\langle \Phi, i\Lambda F_{1,\ell} (\tau \Phi)\rangle_{h_{\ell}}\\
&\quad + i\Lambda F_{2,\ell} \langle \Phi, \tau \Phi \rangle_{h_\ell} + \dbar \langle \Phi,  (\dbar\tau) \Phi \rangle_{h_{\ell}}.
\end{aligned}
\]
Since $\tau$ is $h_{\ell}$ Hermitian we have
\[
\frac{1}{2}{\rm Tr}\left(\{ i\Lambda F_{1,\ell},\Phi\Phi^{\dagger}_{\ell}\} \tau\right) = {\rm Re}\left(\langle i\Lambda F_{1,\ell}\Phi, \tau \Phi \rangle_{h_{\ell}}\right)
\]
and so, integration over $X$ yields~\eqref{eq: integrationByPartsUY}.  Substituting this formula into $(I)$ we have
\[
\int_{X}(I) = \int_{X}(II)
\]
where
\begin{equation}\label{eq: IIsplit}
\begin{aligned}
(II) &= \underbrace{\cos(\htheta)\langle \widetilde{u}_{\ell}^{\mathtt{s}}\Phi, \Phi \rangle_{h_{\ell}} +t_{\ell}\sin(\htheta)i\Lambda F_{2,\ell}\langle \Phi, \widetilde{u}_{\ell}^{\mathtt{s}} \Phi \rangle_{h_{\ell}}}_{(IIa)} \\
&\quad +\underbrace{t_{\ell}\sin(\htheta){\rm Re}\left(\langle\nabla^{\ell}\Phi, (\nabla^{\ell} \widetilde{u}_{\ell}^{\mathtt{s}})\Phi \rangle_{h_{\ell}}\right)}_{(IIb)} +\underbrace{\frac{t_{\ell}}{2}\sin(\htheta)\langle \nabla^{\ell}\Phi, \widetilde{u}_{\ell}^{\mathtt{s}} \nabla^{\ell}\Phi \rangle_{h_{\ell}}}_{(IIc)}.
\end{aligned}
\end{equation}
Thus, it would suffice to show that $(II)$ is non-negative, or that its negative part decays to zero almost everywhere as $\ell \rightarrow \infty$ followed by $\mathtt{s} \rightarrow 0$.  Clearly $(IIc)$ is non-negative.  Furthermore, by~\eqref{eq: dimReduceDHYM-MOC} (or Lemma~\ref{lem: H2CurvLem}), we have
\[
\begin{aligned}
\frac{(IIa)}{\langle \Phi, \widetilde{u}_{\ell}^{\mathtt{s}}\rangle_{h_{\ell}}}&= \cos(\htheta) + t_{\ell}\sin(\htheta)\frac{\cos(\hat{\theta})(|\Phi|^2_{h_{\ell}} - \frac{4\pi}{\sigma}) + \sin(\hat{\theta})(1+\frac{t_{\ell}}{2}|\nabla^{\ell}\Phi|^2_{h_{\ell}})}{\cos(\hat{\theta}) + (\frac{4\pi}{\sigma} -t_{\ell}|\Phi|^2_{h_{\ell}})\sin(\hat{\theta})}\\
&= \frac{\cos^2(\htheta)+t_{\ell}\sin^2(\htheta) + \frac{4\pi}{\sigma}\sin(\htheta)\cos(\htheta)(1-t_{\ell}) + \frac{t_{\ell}^2}{2}\sin^2(\htheta)|\nabla^{\ell}\Phi|^2_{h_{\ell}}}{\cos(\htheta)+(\frac{4\pi}{\sigma}-t_{\ell}|\Phi|^2_{h_{\ell}})\sin(\htheta)}
\end{aligned}
\]
and so $(IIa)$ is also non-negative.  It will be critical to preserve the positive gradient term in the formula for $(IIa)$.  We record the following estimate
\begin{equation}\label{eq: (IIa)GoodTerm}
(IIa) \geq \frac{1}{2}\frac{t_{\ell}^2\sin^2(\htheta)}{\cos(\htheta + (\frac{4\pi}{\sigma}-t_{\ell}|\Phi|^2_{h_{\ell}})\sin(\htheta)}|\nabla^{\ell}\Phi|^2_{h_{\ell}}\langle \Phi, \widetilde{u}_{\ell}^{\mathtt{s}}\rangle_{h_{\ell}}
\end{equation}
The only problematic term is $(IIb)$.  What we will show is that
\begin{equation}\label{eq: IIbTargetEstimate}
\lim_{\mathtt{s}\rightarrow 0}\lim_{\ell \rightarrow \infty} \int_{X}\cos(\htheta)\mathcal{E}_{UY} + (IIa) + (IIb) \geq 0
\end{equation}
This suffices to complete the proof since
\[
\int_{X}\cos(\htheta)\mathcal{E}_{UY} +(I) \geq \int_{X}\cos(\htheta)\mathcal{E}_{UY} + (IIa)+(IIb)
\]
and so ~\eqref{eq: IIbTargetEstimate} leads to a contradiction to stability in~\eqref{eq: stabilityContradiction}.

What we will show is that, outside a set of measure $\epsilon$, the negative part of $(IIb)$ consists of terms decaying to zero and a term controlled by $(IIa)$ and the Uhlenbeck-Yau excess. 

\subsubsection{Proof of~\eqref{eq: IIbTargetEstimate} }
\smallskip

\noindent Let $0<\rank(S):= r' \leq r = \rank (E_1)$. For every $\mathtt{s} \in (0,1]$ we have
\[
\begin{aligned}
\lim_{\ell \rightarrow \infty} \widetilde{u}_{\ell}^{\mathtt{s}} &= \widetilde{u}_{\infty}^{\mathtt{s}}\\
\lim_{\mathtt{s} \rightarrow 0} \widetilde{u}_{\infty}^{\mathtt{s}} &= I-\pi
\end{aligned}
\]
where the convergence is in $L^2({\rm End}(E_1\otimes E_{2}^{\vee}), h_{0})$, and hence pointwise almost everywhere.  Let $\mu^{\ell}_i:= e^{\lambda_i^{\ell}}$ be the eigenvalues of $\widetilde{u}_{\ell}$, and assume that $0<\mu^{\ell}_1 \leq \mu^{\ell}_2 \leq \cdots \leq \mu^{\ell}_r\leq 1$.  We need the following easy lemma:

\begin{lem}\label{lem: pointwiseConvergenceLemma}
    For every $\epsilon >0$ there exists a measurable set $U_{\epsilon}\subset X$ with $|X\setminus U_{\epsilon}|<\epsilon$, and numbers $D=D(\epsilon)>0$, $\ell_0(\epsilon)\in \mathbb{N}$, depending on $\epsilon$, such that on $U_\epsilon$ we have 
\begin{equation}\label{eq: eigenvalueOrderUeps}
0 \geq \lambda^{\ell}_{r} \geq \cdots \geq \lambda^{\ell}_{r'+1} > -D \gg \lambda_{r'}^{\ell} \geq \cdots \geq \lambda_1^{\ell} \quad \text{ for all } \ell \geq \ell_0
\end{equation}
\begin{equation}\label{eq: eigenvalueGoToZero}
\lim_{\ell \rightarrow \infty}\sup_{U_{\epsilon}} \lambda^{\ell}_{j} = -\infty \quad \text{ for all }\quad  j \leq r'
\end{equation}
\end{lem}
\begin{proof}[Proof of Lemma~\ref{lem: pointwiseConvergenceLemma}]
By the theorems of Lusin and Egorov, we may choose a closed set $U_1\subset X$ such that  $|X\setminus U_1| < \frac{1}{2}\epsilon$ and $\widetilde{u}_{\infty}$ is continuous on $U_1$ and $\widetilde{u}_{\ell}$ converge uniformly to $\widetilde{u}_{\infty}$. Since $\pi$ has rank $r'$, and $I-\widetilde{u}_{\infty}^{\mathtt{s}}$ converges pointwise a.e. to $\pi$ as $\mathtt{s} \rightarrow 0$, we must have that $u_{\infty}^{\mathtt{s}}$ has rank $r-r'$ a.e. on $U_1$.  Thus, we can find a closed set $U_2 \subset U_1$ such that $|X\setminus U_2|<\epsilon$, and such that $\widetilde{u}_{\infty}$ has rank $r-r'$ everywhere on $U_2$.  In particular, the eigenvalues of $\widetilde{u}_{\infty}$ are given by
\[
\mu^{\infty}_{r} \geq \cdots \geq \mu^{\infty}_{r'+1} > 0  = \mu^{\infty}_{r'} = \cdots = \mu^{\infty}_{1}
\]
Let $e^{-D/2} = \inf_{U_2}\mu^{\infty}_{r'+1}$. Then $0<D<+\infty$ since $\widetilde{u}_{\infty}$ is continuous on $U_2\subset U_1$. Since $\widetilde{u}_{\ell}$ converges uniformly to $\widetilde{u}_{\infty}$ on $U_2$ we can find $\ell \geq \ell_0$ such that
\[
0 \geq \lambda^{\ell}_{r} \geq \cdots \geq \lambda^{\ell}_{r'+1} > -D > \lambda_{r'}^{\ell} \geq \cdots \geq \lambda_1^{\ell}  \quad \text{ on } U_2.
\]
and
\[
\lim_{\ell \rightarrow \infty} \sup_{U_{\epsilon}}\lambda^{\ell}_{j} = -\infty \quad \text{ for all }\quad  j \leq r' \quad \text{ on } U_{2}.
\]
\end{proof}

Fix $U_{\epsilon}$ as in Lemma~\ref{lem: pointwiseConvergenceLemma} and assume that $\ell \gg \ell_0$. The bounds for $\Phi, \nabla^{\ell}\Phi$ and $i\Lambda F_{1,\ell}, i\Lambda F_{2,\ell}$ yield a constant $C$, independent of $\ell, \mathtt{s}, \epsilon$ so that 
\begin{equation}\label{eq: compUeps}
\bigg|\int_{X\setminus U_{\epsilon}} \langle \nabla^{\ell} \Phi, (\nabla^{\ell}\widetilde{u}_{\ell}^{\mathtt{s}}) \Phi \rangle_{h_{\ell}}\bigg| \leq C\epsilon^{1/2}.
\end{equation}
In detail, we have
\[
\begin{aligned}
\langle \nabla^{\ell} \Phi, (\nabla^{\ell}\widetilde{u}_{\ell}^{\mathtt{s}}) \Phi \rangle_{h_{\ell}}&=\langle \nabla^{\ell} \Phi, (\nabla^{\ell}\widetilde{u}_{\ell}^{\mathtt{s}})\widetilde{u}_{\ell}^{-\mathtt{s}/2}(\widetilde{u}_{\ell}^{\mathtt{s}/2} \Phi )\rangle_{h_{\ell}}\\
& \leq |\nabla^{\ell}\Phi|_{h_{\ell}}|(\nabla^{\ell}\widetilde{u}_{\ell}^{\mathtt{s}})\widetilde{u}_{\ell}^{-\mathtt{s}/2}|_{h_{\ell}}|\widetilde{u}_{\ell}^{\mathtt{s}/2} \Phi|_{h_{\ell}}
\end{aligned}
\]
Since $0<\tilde{u}_{\ell}^{\mathtt{s}}<I$ we have
\[
\begin{aligned}
\bigg|\int_{X\setminus U_{\epsilon}} \langle \nabla^{\ell} \Phi, (\nabla^{\ell}\widetilde{u}_{\ell}^{\mathtt{s}}) \Phi \rangle_{h_{\ell}}\bigg| &\leq C|X\setminus U_{\epsilon}|^{1/2}\left(\int_{X}|(\nabla^{\ell}\widetilde{u}_{\ell}^{\mathtt{s}})\widetilde{u}_{\ell}^{-\mathtt{s}/2}|_{h_{\ell}}^2\right)^{1/2}\\
&\leq C\epsilon^{1/2}\left(\int_{X}\langle (\nabla^{\ell}\widetilde{u}_{\ell})\widetilde{u}_{\ell}^{-1}, \nabla^{\ell}\widetilde{u}_{\ell}^{\mathtt{s}}\rangle_{h_{\ell}}\right)^{1/2}\\
&\leq C\epsilon^{1/2{}}
\end{aligned}
\]
where we used ~\eqref{eq: UYFormula2} in passing from the first line to the second, and integration by parts and the curvature bounds in the final inequality.  That is, in the final line we used
\begin{equation}\label{eq: UYINBcurvBound}
0<\int_{X}\langle (\nabla^{\ell}\widetilde{u}_{\ell})\widetilde{u}_{\ell}^{-1}, \nabla^{\ell}\widetilde{u}_{\ell}^{\mathtt{s}}\rangle_{h_{\ell}}=\int_{X}{\rm Tr}(i\Lambda F_{\ell}\widetilde{u}_{\ell}^{\mathtt{s}}) -{\rm Tr}(i\Lambda F_0 \widetilde{u}_{\ell}^{\mathtt{s}}) \leq C
\end{equation}
It suffices to estimate the integral over $U_{\epsilon}$.  To this end, fix a point $p \in U_{\epsilon}$, let $e_1, \ldots e_r$ be a local $h_{\ell}$-orthonormal frame of $\widetilde{u}_{\ell}$ eigensections.  We assume that $\ell \geq \ell_0$ so that we may take the eigenvalues to be ordered according to~\eqref{eq: eigenvalueOrderUeps}.  We write
\[
\begin{aligned}
(\nabla^{\ell} \Phi)^{B} &= \sum_{\gamma > r'} (\nabla^{\ell} \Phi)^{\gamma}e_{\gamma} &(\nabla^{\ell} \Phi)^S &= \sum_{\gamma \leq r'} (\nabla^{\ell} \Phi)^{\gamma}e_{\gamma}\\
 \Phi^{B} &= \sum_{\gamma > r'} \Phi^{\gamma}e_{\gamma} &\Phi^S &= \sum_{\gamma \leq r'}  \Phi^{\gamma}e_{\gamma}.
\end{aligned}
\]
Here, the superscript $B$ stands for the component lying in the span of the eigensections with ``big" eigenvalues of $\widetilde{u}_{\ell}$, while the superscript $S$ stands for the component lying in the span of the eigensections with ``small" eigenvalues. Now we decompose the inner product appearing in the term $(IIb)$ by
\[
\begin{aligned}
\langle \nabla^{\ell} \Phi, (\nabla^{\ell}\widetilde{u}_{\ell}^{\mathtt{s}} )\Phi\rangle_{h_{\ell}} &=\underbrace{\langle (\nabla^{\ell} \Phi)^B, (\nabla^{\ell}\widetilde{u}_{\ell}^{\mathtt{s}} )(\Phi)^S\rangle_{h_{\ell}}}_{(IIIa)}+\underbrace{\langle \nabla^{\ell} (\Phi)^B, (\nabla^{\ell}\widetilde{u}_{\ell}^{\mathtt{s}} )(\Phi)^B\rangle_{h_{\ell}}}_{(IIIb)}\\
&\quad + \underbrace{\langle (\nabla^{\ell} \Phi)^S, (\nabla^{\ell}\widetilde{u}_{\ell}^{\mathtt{s}} )\Phi\rangle_{h_{\ell}}}_{(IIIc)}.
\end{aligned}
\]
Each of these terms will be estimated separately.
\smallskip

\noindent {\bf Estimate of  $(IIIa)$.}  Write
\[
(IIIa) = \langle (\nabla^{\ell} \Phi)^B, (\nabla^{\ell}\widetilde{u}_{\ell}^{\mathtt{s}} )\widetilde{u}_{\ell}^{-\mathtt{s}/2}\widetilde{u}_{\ell}^{\mathtt{s}/2}  (\Phi)^S\rangle_{h_{\ell}}.
\]
Since $\Phi$ is uniformly bounded with respect to $h_{\ell}$ by Lemma~\ref{lem: conditionalPhiBound}, and $|\nabla^{\ell}\Phi|_{h_{\ell}}$ is uniformly bounded  by Proposition~\ref{prop: gradientBound} we get 
\begin{equation}\label{eq: IIIa estimate}
\begin{aligned}
\bigg|(IIIa)\bigg|&\leq Ce^{\frac{\mathtt{s}}{2} \lambda^{\ell}_{r'}}|(\nabla^{\ell}\widetilde{u}_{\ell}^{\mathtt{s}})\widetilde{u}_{\ell}^{-\mathtt{s}/2} |_{h_{\ell}}\\
&\leq Ce^{\frac{\mathtt{s}}{2} \lambda^{\ell}_{r'}}\left(\langle (\nabla^{\ell}\widetilde{u}_{\ell}) \widetilde{u}_{\ell}^{-1}, \nabla^{\ell}\widetilde{u}_{\ell}^{\mathtt{s}}\rangle_{h_{\ell}} \right)^{\frac{1}{2}}
\end{aligned}
\end{equation}
where we used the Uhlenbeck-Yau bound~\eqref{lem: UYFormulaImprovement} in passing to the second line.
\smallskip

 \noindent {\bf Estimate of $(IIIb)$.}  We write $(IIIb)$ in terms of a local orthonormal frame of eigensections for $\widetilde{u}_{\ell}$.  Define the connection coefficients of the Chern connection $\nabla^{\ell}$ in this frame by
\[
\nabla^{\ell}_ze_{\alpha} = A_{z\, \alpha}^{\gamma} e_{\gamma}, \qquad  \overline{\nabla}^{\ell}_{\bar{z}}e_{\alpha} = A_{\bar{z}\, \alpha}^{\gamma} e_{\gamma}.
\]
Using the Uhlenbeck-Yau formula~\eqref{eq: coverDerivUSigma} below, together with the uniform bounds of $|\Phi|_{h_{\ell}}$ and $|\nabla^{\ell}\Phi|_{h_{\ell}}$ we have
\[
\begin{aligned}
\bigg|\langle (\nabla \Phi)^B, (\nabla^{\ell}\widetilde{u}_{\ell}^{\mathtt{s}} )  (\Phi)^B\rangle_{h_{\ell}}\bigg|^2 &=\bigg|\sum_{\alpha > r'} \mathtt{s}\, \overline{(\nabla \Phi)^{\alpha}}\,\Phi^{\alpha} e^{\mathtt{s}\lambda_{\alpha}^{\ell}}\del \lambda_{\alpha}^{\ell} + \sum_{\alpha >r'}\sum_{\gamma>r'}\overline{(\nabla \Phi)^{\gamma}} \Phi^{\alpha} (e^{\mathtt{s} \lambda_{\alpha}^{\ell}}-e^{\mathtt{s}\lambda_{\gamma}^{\ell}})A_{z\, \alpha}^{\gamma}\big|^2\\
& \leq C\left(\sum_{\alpha >r'} \mathtt{s}^2e^{2\mathtt{s} \lambda_{\alpha}^{\ell}}|\del \lambda_{\alpha}|^2 + \sum_{\alpha>r'}\sum_{\gamma>r'}|(\nabla^{\ell} \Phi)^{\gamma}|^2|\Phi^{\alpha}|^2(e^{\mathtt{s} \lambda_{\alpha}^{\ell}}-e^{\mathtt{s}\lambda_{\gamma}^{\ell}})^2|A_{z\, \alpha}^{\gamma}|^2\right)
\end{aligned}
\]
Now since $|\Phi|_{h_{\ell}}$ and $|\nabla^{\ell}\Phi|_{h_{\ell}}$ are bounded and $\lambda_{\alpha}\leq 0$, we obtain 
\[
\begin{aligned}
\bigg|\langle (\nabla \Phi)^B, (\nabla^{\ell}\widetilde{u}_{\ell}^{\mathtt{s}} )  (\Phi)^B\rangle_{h_{\ell}}\bigg|^2 & \leq C\left(\sum_{\alpha >r'} \mathtt{s}^2e^{2\mathtt{s} \lambda_{\alpha}^{\ell}}|\del \lambda_{\alpha}^{\ell}|^2 + \sum_{\alpha'>r}\sum_{\gamma>r}e^{-\mathtt{s} \lambda_{\alpha}^{\ell}}(e^{\mathtt{s} \lambda_{\alpha}^{\ell}}-e^{\mathtt{s}\lambda_{\gamma}^{\ell}})^2|A_{z\, \alpha}^{\gamma}|^2\right)\\
&=C\left(\sum_{\alpha >r'} \mathtt{s}^2e^{2\mathtt{s} \lambda_{\alpha}^{\ell}}|\del \lambda_{\alpha}^{\ell}|^2 + \sum_{\alpha'>r}\sum_{\gamma>r}(e^{\mathtt{s} \lambda_{\alpha}^{\ell}}-e^{\mathtt{s}\lambda_{\gamma}^{\ell}})(1-e^{\mathtt{s}(\lambda_{\gamma}^{\ell} -\lambda_{\alpha}^{\ell})})|A_{z\, \alpha}^{\gamma}|^2\right)
\end{aligned}
\]

Now we need the following elementary lemma
\begin{lem}\label{lem: calculus}
    Fix $D>0$ and $\mathtt{s} \in (0,1)$.  Then for $\tau=\tau(D,\mathtt{s})=\mathtt{s} e^{(1-\mathtt{s})D}$, the following holds: for all $x \in [0,D]$ we have
    \[
    1-e^{-\mathtt{s} x} \leq \tau\cdot  (1-e^{-x}) \quad \text{ and,} \quad  (e^{\mathtt{s} x} -1) \leq \tau\cdot (e^{x}-1)
    \]
\end{lem}
\begin{proof}[Proof of Lemma~\ref{lem: calculus}]
    Given $D, \mathtt{s}$, the desired inequalities evidently hold at $x=0$.  Plugging in $x=D$ we must choose
    \[
    \tau \geq \frac{1-e^{-\mathtt{s} D}}{1-e^{-D}},\quad  \text{ and } \quad \tau\geq \frac{e^{\mathtt{s} D}-1}{e^{D}-1}
    \]
    where the first bound is dominant.  With these choices the functions $f_{-}(x):= \tau\cdot (1-e^{-x})-(1-e^{-\mathtt{s} x})$ and $f_{+}(x):=\tau\cdot (e^{x}-1)-(e^{\mathtt{s} x} -1)$ are non-negative  at $x=0, D$.  Now we compute
    \[
    \frac{d f_{-}}{dx} = \tau e^{-x}-\mathtt{s} e^{-\mathtt{s} x}
    \]
    and this is non-negative provided $\tau \geq \mathtt{s} e^{(1-\mathtt{s})D}$. Computing similarly
    \[
    \frac{df_+}{dx} = \tau e^{x}-\mathtt{s} e^{\mathtt{s} x}
    \]
    which is non-negative as long as $\tau >\mathtt{s}$.  Since $\mathtt{s} e^{(1-\mathtt{s})D} \geq \frac{1-e^{-\mathtt{s} D}}{1-e^{-D}}$, the desired result holds.

\end{proof}

Now, for $\alpha > r'$ we know that $\lambda_{\alpha} \in [-D,0]$ on $U_{\epsilon}$ by Lemma~\ref{lem: pointwiseConvergenceLemma}. Thus, Lemma~\ref{lem: calculus} gives that, for all $\alpha, \gamma >r'$ there holds
    \[
    (e^{\mathtt{s} \lambda_{\alpha}^{\ell}}-e^{\mathtt{s}\lambda_{\gamma}^{\ell}})(1-e^{\mathtt{s}(\lambda_{\gamma}^{\ell} -\lambda_{\alpha}^{\ell})}) \leq \tau(D,\mathtt{s})
    (e^{\mathtt{s} \lambda_{\alpha}^{\ell}}-e^{\mathtt{s}\lambda_{\gamma}^{\ell}})(1-e^{(\lambda_{\gamma}^{\ell} -\lambda_{\alpha}^{\ell})})
    \]
As a consequence we obtain
\[
\bigg|(IIIb)\bigg|^2 \leq C\tau(D,\mathtt{s})\left(\sum_{\alpha >r'} \mathtt{s} e^{\mathtt{s} \lambda_{\alpha}}|\del \lambda_{\alpha}^{\ell}|^2 + \sum_{\alpha>r'}\sum_{\gamma>r'}(e^{\mathtt{s} \lambda_{\alpha}^{\ell}}-e^{\mathtt{s}\lambda_{\gamma}})(1-e^{(\lambda_{\gamma}^{\ell} -\lambda_{\alpha}^{\ell})})|A_{z\, \alpha}^{\gamma}|^2\right)
\]
where we used that $\mathtt{s}^2 e^{2\mathtt{s} \lambda_{\alpha}} \leq \mathtt{s} e^{\mathtt{s} \lambda_{\alpha}}$.  We now apply the Uhlenbeck-Yau formula~\eqref{eq: covarDerivLocalFormula} below, to get the following bound on $U_{\epsilon}$: 
\[
\bigg|\langle (\nabla \Phi)^B, (\nabla^{\ell}\widetilde{u}_{\ell}^{\mathtt{s}} )  (\Phi)^B\rangle_{h_{\ell}}\bigg|^2 \leq C\tau(D,\mathtt{s})\langle (\nabla^{\ell}\widetilde{u}_{\ell})\widetilde{u}_{\ell}^{-1}, \nabla^{\ell}\widetilde{u}^{\mathtt{s}}_{\ell}\rangle_{h_{\ell}} \quad \text{ on } U_{\epsilon}.
\]
Thus we obtain
\begin{equation}\label{eq: IIIb estimate}
\bigg|(IIIb)\bigg| \leq C \tau(D,\mathtt{s})^{1/2}\left(\langle (\nabla^\ell\widetilde{u}_{\ell})\widetilde{u}_{\ell}^{-1}, \nabla^{\ell}\widetilde{u}^{\mathtt{s}}_{\ell}\rangle_{h_{\ell}}\right)^{\frac{1}{2}}. 
\end{equation}
\smallskip

\noindent {\bf Estimate of $(IIIc)$.} The estimate of $(IIIc)$ is the most complicated.  We are interested in understanding the expression
\[
\overline{(IIIc)} = \langle (\nabla^{\ell} \widetilde{u}_{\ell}^{\mathtt{s}} ) \Phi, (\nabla^{\ell} \Phi)^S \rangle_{h_{\ell}}.
\]
We write this in coordinates using our local frame, as above,
\[
\overline{(IIIc)}= \underbrace{\mathtt{s} \sum_{\alpha \leq r'} \overline{(\nabla^{\ell}\Phi)^{\alpha}}\Phi^{\alpha}\del\lambda_{\alpha}^{\ell}e^{\mathtt{s}\lambda_{\alpha}^{\ell}}}_{\overline{(IIIc)_{1}}} + \underbrace{\sum_{\gamma \leq r'}\sum_{\alpha}\overline{(\nabla^{\ell}\Phi)^{\gamma}}(e^{\mathtt{s} \lambda_{\alpha}^{\ell}} -e^{\mathtt{s} \lambda_{\gamma}^{\ell}})\Phi^{\alpha}A_{z\, \alpha}^\gamma}_{\overline{(IIIc)}_2}
\]
The first term, $\overline{(IIIc)}_1$ is easily bounded.  By Cauchy-Schwarz we have
\[
|\overline{(IIIc)}_1| \leq \left(\mathtt{s} \sum_{\alpha}e^{\mathtt{s} \lambda_{\alpha}}|\del \lambda_{\alpha}|^2\right)^{\frac{1}{2}}\left(\mathtt{s} \sum_{\alpha \leq r'}e^{\mathtt{s} \lambda_{\alpha}}|(\nabla^{\ell} \Phi)^{\alpha}|^2|\Phi^{\alpha}|^2\right)^{\frac{1}{2}}
\]
By the Uhlenbeck-Yau formula \eqref{eq: covarDerivLocalFormula}, we get
\[
|\overline{(IIIc)}_1| \leq \mathtt{s}^{1/2}\left(\langle (\nabla^{\ell}\widetilde{u}_{\ell}) \widetilde{u}_{\ell}^{-1}, \nabla^{\ell}\widetilde{u}_{\ell}^{\mathtt{s}}\rangle_{h_{\ell}} \right)^{\frac{1}{2}} e^{\frac{\mathtt{s}}{2}\lambda^{\ell}_{r'}}(\sup_{X}|\Phi|_{h_{\ell}})( \sup_{X}|\nabla^{\ell}\Phi|_{h_{\ell}}).
\]
Next we estimate $\overline{(IIIc)}_{2}$.
\[
\overline{(IIIc)}_{2}:=\sum_{\gamma \leq r'}\sum_{\alpha} \overline{(\nabla^{\ell}\Phi)^{\gamma}}(e^{\mathtt{s}\lambda_{\alpha}^{\ell}}-e^{\mathtt{s}\lambda^{\ell}_{\gamma}})\Phi^{\alpha}A_{z}^{\gamma}\,_{\alpha}
\]
We decompose the sum according to whether $\alpha>r'$ or $\alpha \leq r'$.  Write
\[
\begin{aligned}
\overline{(IIIc)}_{2B}&:=\sum_{\gamma \leq r'}\sum_{\alpha>r'} \overline{(\nabla^{\ell}\Phi)^{\gamma}}(e^{\mathtt{s}\lambda_{\alpha}^{\ell}}-e^{\mathtt{s}\lambda^{\ell}_{\gamma}})\Phi^{\alpha}A_{z}^{\gamma}\,_{\alpha}\\
\overline{(IIIc)}_{2S} &:=\sum_{\gamma \leq r'}\sum_{\alpha\leq r'} \overline{(\nabla^{\ell}\Phi)^{\gamma}}(e^{\mathtt{s}\lambda_{\alpha}^{\ell}}-e^{\mathtt{s}\lambda^{\ell}_{\gamma}})\Phi^{\alpha}A_{z}^{\gamma}\,_{\alpha}
\end{aligned}
\]
First consider $\overline{(IIIc)}_{2B}$.  We apply Cauchy-Schwarz to get
\[
\begin{aligned}
\bigg|\overline{(\nabla^{\ell}\Phi)^{\gamma}}(e^{\mathtt{s}\lambda_{\alpha}^{\ell}}-e^{\mathtt{s}\lambda^{\ell}_{\gamma}})\Phi^{\alpha}A_{z}^{\gamma}\,_{\alpha} \bigg|&\leq \frac{t_{\ell}\sin(\htheta)}{4\cos(\htheta)}e^{s\lambda_{\alpha}^{\ell}}|(\nabla^{\ell} \Phi)^{\gamma}|^2|\Phi^{\alpha}|^2 \\
&\quad + \frac{\cos(\htheta)}{t_{\ell}\sin(\htheta)}e^{-s\lambda_{\alpha}^{\ell}}|A_{z}^{\gamma}\,_{\alpha}|^2(e^{\mathtt{s}\lambda_{\alpha}^{\ell}}-e^{\mathtt{s}\lambda^{\ell}_{\gamma}})^2\\
\end{aligned}
\]
Thus we have
\begin{equation}\label{eq: IIIc2Besimtate}
\begin{aligned}
|(IIIc)_{2B}| &\leq \frac{t_{\ell}\sin(\htheta)}{4\cos(\htheta)}\sum_{\gamma \leq r'}\sum_{\alpha >r'}e^{s\lambda_{\alpha}^{\ell}}|(\nabla^{\ell} \Phi)^{\gamma}|^2|\Phi^{\alpha}|^2 \\
&\quad + \frac{\cos(\htheta)}{t_{\ell}\sin(\htheta)}\sum_{\gamma \leq r'}\sum_{\alpha > r'}e^{-\mathtt{s}\lambda_{\alpha}^{\ell}}|A_{z}^{\gamma}\,_{\alpha}|^2(e^{\mathtt{s}\lambda_{\alpha}^{\ell}}-e^{\mathtt{s}\lambda^{\ell}_{\gamma}})^2\\
& \leq \frac{t_{\ell}\sin(\htheta)}{4\cos(\htheta)}|\nabla^{\ell}\Phi|^2\langle \Phi, \tilde{u}_{\ell}^{\mathtt{s}}\Phi \rangle_{h_{\ell}} + \frac{\cos(\htheta)}{t_{\ell}\sin(\htheta)}\underbrace{\sum_{\gamma \leq r'}\sum_{\alpha > r'}e^{-\mathtt{s}\lambda_{\alpha}^{\ell}}|A_{z}^{\gamma}\,_{\alpha}|^2(e^{\mathtt{s}\lambda_{\alpha}^{\ell}}-e^{\mathtt{s}\lambda^{\ell}_{\gamma}})^2}_{(E1)}
\end{aligned}
\end{equation}
It is convenient, at this point, to introduce the following notation:
\[
r_{\gamma\alpha}:= e^{\lambda_{\gamma}^{\ell} -\lambda_{\alpha}^{\ell}}
\]
so that we can write
\[
(E1) = \sum_{\gamma \leq r'}\sum_{\alpha > r'}e^{\mathtt{s}\lambda_{\alpha}^{\ell}}|A_{z}^{\gamma}\,_{\alpha}|^2(1-r_{\gamma \alpha}^\mathtt{s})^2.
\]
We need the following algebraic claim: for any $\epsilon>0$, $\mathtt{s}\in(0,\frac{1}{2})$ and $\ell$ sufficiently large, depending on $ \epsilon$ we have
\[
(1-r_{\gamma \alpha}^{\mathtt{s}})(1-r_{\gamma\alpha}^{1-\mathtt{s}}) -(1-r_{\gamma\alpha}^{\mathtt{s}})^2 = r_{\gamma\alpha}+r_{\gamma \alpha}^{\mathtt{s}}-r_{\gamma \alpha}^{2s}-r_{\gamma \alpha}^{1-\mathtt{s}} \geq 0 \quad \text{ on } U_{\epsilon} \text{ for all } \alpha >r', \gamma \leq r'
\]
To see this recall that for $\ell$ sufficiently large, if $\gamma \leq r' <\alpha$, then $\lambda_{\gamma}^{\ell} <-D < \lambda_{\alpha}^{\ell}<0$  on $U_{\epsilon}$ by Lemma~\ref{lem: pointwiseConvergenceLemma}.  Therefore, $r_{\gamma \alpha} \in(0,1)$.  The claim follows from the fact that, for any $\mathtt{s}\in(0,\frac{1}{2})$, the function $x\mapsto  x+x^{\mathtt{s}}-x^{2\mathtt{s}}-x^{1-\mathtt{s}}$ is non-negative for $x\in [0,1]$.

In particular, this leads to the following estimate
\begin{equation}\label{eq: E1estimate}
(E1) \leq \sum_{\gamma \leq r'}\sum_{\alpha > r'}e^{\mathtt{s}\lambda_{\alpha}^{\ell}}|A_{z}^{\gamma}\,_{\alpha}|^2(1-r_{\gamma \alpha}^{\mathtt{s}})(1-r_{\gamma\alpha}^{1-\mathtt{s}})
\end{equation}
for any fixed $\mathtt{s}\in (0,\frac{1}{2})$ and all $\ell$ sufficiently large depending on $\mathtt{s}, \epsilon$.

In what follows we shall use the additional easy fact that for $\mathtt{s}\in(0,1)$ we have
\[
(1-r_{\gamma\alpha}^{\mathtt{s}})(1-r_{\gamma \alpha}^{(1-\mathtt{s})})>0 \quad \text{ provided } r_{\gamma\alpha}\ne 1.
\]
We now estimate $(IIIc)_{2S}$ using another application of Cauchy-Schwarz.  Since the term in the sum defining $(IIIc)_{2S}$ vanishes when $\lambda_{\gamma}^{\ell} =\lambda_{\alpha}^{\ell}$, we may assume $\lambda_{\gamma}^{\ell} \ne \lambda_{\alpha}^{\ell}$.  We estimate
\[
\begin{aligned}
\bigg|\overline{(\nabla^{\ell}\Phi)^{\gamma}}(e^{\mathtt{s}\lambda_{\alpha}^{\ell}}-e^{\mathtt{s}\lambda^{\ell}_{\gamma}})\Phi^{\alpha}A_{z}^{\gamma}\,_{\alpha} \bigg|& \leq \frac{t_{\ell}\sin(\htheta)}{4\cos(\htheta)}\epsilon |(\nabla^{\ell} \Phi)^{\gamma}|^2|\Phi^{\alpha}|^2 + \frac{\cos(\htheta)}{t_{\ell}\sin(\htheta)\epsilon}|A_{z}^{\gamma}\,_{\alpha}|^2(e^{\mathtt{s}\lambda_{\alpha}^{\ell}}-e^{\mathtt{s}\lambda_{\gamma}^{\ell}})^2
\end{aligned}
\]
For reasons that will become clear later in the proof, we choose 
\[
\epsilon=\frac{e^{-\mathtt{s}\lambda_{\alpha}^{\ell}}(e^{\mathtt{s}\lambda_{\alpha}^{\ell}}-e^{\mathtt{s}\lambda_{\gamma}^{\ell}})^2}{(1-r_{\gamma\alpha}^{\mathtt{s}})(1-r_{\gamma \alpha}^{(1-\mathtt{s})})}
\]
noting that since we are assuming $\lambda_{\gamma}^{\ell} \ne \lambda_{\alpha}^{\ell}$, this is well-defined.  It is easy to verify that for $s\in(0,\frac{1}{2})$ we have
\[
\epsilon \leq \max\{e^{\mathtt{s}\lambda_{\alpha}^{\ell}}, e^{\mathtt{s}\lambda_{\gamma}^{\ell}}\} \leq e^{\mathtt{s}\lambda_{r'}^{\ell}}
\]
With this estimate we obtain
\begin{equation}\label{eq: IIIc2Sestimate}
|(IIIc)_{2S}| \leq e^{\mathtt{s}\lambda_{r'}^{\ell}}\frac{t_{\ell}\sin(\htheta)}{4\cos(\htheta)}|\nabla^{\ell}\Phi|^2_{h_{\ell}}|\Phi|^2_{h_{\ell}} + \frac{\cos(\htheta)}{t_{\ell}\sin(\htheta)}\sum_{\gamma \leq r'}\sum_{\alpha \leq r'}e^{\mathtt{s}\lambda_{\alpha}^{\ell}}|A_{z}^{\gamma}\,_{\alpha}
|^2(1-r_{\gamma\alpha}^{\mathtt{s}})(1-r_{\gamma \alpha}^{(1-\mathtt{s})})
\end{equation}
Combining the estimates~\eqref{eq: IIIc2Besimtate},\eqref{eq: E1estimate} and~\eqref{eq: IIIc2Sestimate} we obtain
\begin{equation}\label{eq: IIIcEstimate}
\begin{aligned}
|(IIIc)| &\leq e^{\mathtt{s}\lambda_{r'}^{\ell}}\frac{t_{\ell}\sin(\htheta)}{\cos(\htheta)}|\nabla^{\ell}\Phi|^2_{h_{\ell}}|\Phi|^2_{h_{\ell}} +  \frac{t_{\ell}\sin(\htheta)}{4\cos(\htheta)}|\nabla^{\ell}\Phi|^2\langle \Phi, \tilde{u}_{\ell}^{\mathtt{s}}\Phi \rangle_{h_{\ell}} \\
&\quad \frac{\cos(\htheta)}{t_{\ell}\sin(\htheta)}\sum_{\gamma \leq r'}\sum_{\alpha }e^{\mathtt{s}\lambda_{\alpha}^{\ell}}|A_{z}^{\gamma}\,_{\alpha}
|^2(1-r_{\gamma\alpha}^{\mathtt{s}})(1-r_{\gamma \alpha}^{(1-\mathtt{s})})\\
\end{aligned}
\end{equation}
Recall the Uhlenbeck-Yau excess $\mathcal{E}_{UY}$ defined in~\eqref{eq: UYexcess}. We apply the formula for $\mathcal{E}_{UY}$ obtained in Lemma~\ref{lem: UYExcessFormula} (noting Remark~\ref{rk: UYexcessHomog}) which gives
\[
\mathcal{E}_{UY} \geq \sum_{\gamma}\sum_{\alpha}e^{\mathtt{s}\lambda_{\alpha}^{\ell}}|A_{z}^{\gamma}\,_{\alpha}
|^2(1-r_{\gamma\alpha}^{\mathtt{s}})(1-r_{\gamma \alpha}^{(1-\mathtt{s})}).
\]
Furthermore, since $\tan(\htheta) \leq \frac{\sigma}{4\pi}$ we have
\[
\frac{t_{\ell}^2\sin^2(\htheta)}{4\cos(\htheta)}|\nabla^{\ell}\Phi|^2\langle \Phi, \tilde{u}_{\ell}^{\mathtt{s}}\Phi\rangle_{h_{\ell}}\leq \frac{1}{2}\frac{t_{\ell}^2\sin^2(\htheta)}{\cos(\htheta + (\frac{4\pi}{\sigma}-t_{\ell}|\Phi|^2_{h_{\ell}})\sin(\htheta)}|\nabla^{\ell}\Phi|^2_{h_{\ell}}\langle \Phi, \widetilde{u}_{\ell}^{\mathtt{s}}\rangle_{h_{\ell}} \leq (IIa)
\]
using the inequality~\eqref{eq: (IIa)GoodTerm}.
Therefore, combining our estimates we have
\begin{equation}\label{eq: IIIcEstimateWithExcess}
\begin{aligned}
|(IIIc)| &\leq e^{\mathtt{s}\lambda_{r'}^{\ell}}\frac{t_{\ell}\sin(\htheta)}{\cos(\htheta)}|\nabla^{\ell}\Phi|^2_{h_{\ell}}|\Phi|^2_{h_{\ell}} +  \frac{1}{t_{\ell}\sin(\htheta)}(IIa) +\frac{\cos(\htheta)}{t_{\ell}\sin(\htheta)}\mathcal{E}_{UY}
\end{aligned}
\end{equation}
We now combine our estimates.  From~\eqref{eq: IIIa estimate},~\eqref{eq: IIIb estimate} and ~\eqref{eq: IIIcEstimateWithExcess}, the uniform bounds for $|\Phi|_{h_{\ell}}$ and $|\nabla^{\ell}\Phi|_{h_{\ell}}$ and recalling the definition of $(IIb)$ in~\eqref{eq: IIsplit} we have that, on $U_{\epsilon}$ the following estimate holds: for all $\mathtt{s}\in (0, \frac{1}{2})$ and for all $\ell$ sufficiently large, depending on $\mathtt{s}, \epsilon$, we have
\[
\begin{aligned}
(IIa)+(IIb)+\cos(\htheta)\mathcal{E}_{UY} &\geq -C\left(e^{\frac{\mathtt{s}}{2}\lambda^{\ell}_{r'}} + \tau(D,\mathtt{s})^{1/2}\right)\left(1+\langle (\nabla^{\ell}\widetilde{u}_{\ell}) \widetilde{u}_{\ell}^{-1}, \nabla^{\ell}\widetilde{u}_{\ell}^{\mathtt{s}}\rangle_{h_{\ell}} \right)^{\frac{1}{2}}
\end{aligned}
\] 
Thus, all together we have
\[
(IIa)+(IIb) + \cos(\htheta)\mathcal{E}_{UY} \geq  -C\left(e^{\frac{\mathtt{s}}{2}\lambda^{\ell}_{r'}} + \tau(D,\mathtt{s})^{1/2}\right)\left(1+\langle (\nabla^{\ell}\widetilde{u}_{\ell}) \widetilde{u}_{\ell}^{-1}, \nabla^{\ell}\widetilde{u}_{\ell}^{\mathtt{s}}\rangle_{h_{\ell}} \right)^{\frac{1}{2}}
\]
By Cauchy-Schwarz we have
\[
\begin{aligned}
&\int_{U_{\epsilon}}\left(e^{\frac{\mathtt{s}}{2}\lambda^{\ell}_{r'}} + \tau(D,\mathtt{s})^{1/2}\right)\left(1+\langle (\nabla^{\ell}\widetilde{u}_{\ell}) \widetilde{u}_{\ell}^{-1}, \nabla^{\ell}\widetilde{u}_{\ell}^{\mathtt{s}}\rangle_{h_{\ell}} \right)^{\frac{1}{2}}\\
&\quad \leq C \left(\int_{U_{\epsilon}}e^{\mathtt{s}\lambda^{\ell}_{r'}} +\tau(D,\mathtt{s})\right)^{1/2} \left(\int_{X}1+\langle (\nabla^{\ell}\widetilde{u}_{\ell}) \widetilde{u}_{\ell}^{-1}, \nabla^{\ell}\widetilde{u}_{\ell}^{\mathtt{s}}\rangle_{h_{\ell}}\right)^{1/2}\\
&\quad \leq C\left(\int_{U_{\epsilon}}e^{\mathtt{s}\lambda^{\ell}_{r'}}+\tau(D,\mathtt{s})\right)^{1/2}
\end{aligned}
\]
using~\eqref{eq: UYINBcurvBound}.  Taking the limit as $\ell \rightarrow \infty$ and using that $\lambda_{r'}^{\ell} \rightarrow -\infty$ uniformly on $U_{\epsilon}$ yields
\[
\lim_{\ell\rightarrow \infty}\int_{U_{\epsilon}}\left((IIa)+(IIb) + \cos(\htheta)\mathcal{E}_{UY}\right) \geq -C \tau(D,\mathtt{s})^{1/2}
\]
for all $\mathtt{s}\in (0,\frac{1}{2})$.  We may now take the limit as $\mathtt{s} \rightarrow 0$, using that, for fixed $D$, $\tau(D,\mathtt{s})= O(s)$ as $s\rightarrow 0$.  We finally conclude
\[
\lim_{\mathtt{s}\rightarrow 0}\lim_{\ell\rightarrow \infty}\int_{U_{\epsilon}}\left((IIa)+(IIb) + \cos(\htheta)\mathcal{E}_{UY}\right) \geq 0.
\]
Finally, using~\eqref{eq: compUeps} we have that, for all $\epsilon>0$ 
\[
\lim_{\mathtt{s}\rightarrow 0}\lim_{\ell\rightarrow \infty}\int_{X}\left((IIa)+(IIb) + \cos(\htheta)\mathcal{E}_{UY}\right) \geq -C\epsilon^{1/2}.
\]
Taking the limit as $\epsilon \rightarrow 0$ yields~\eqref{eq: IIbTargetEstimate} and the desired contradiction.
\end{proof}

Proposition~\ref{prop: UYargumentH} shows that stability rules out the possibility of $h_1h_{2}^{-1}$ being unbounded from above along the method of continuity.  However, it is still possible that $h_1, h_2$ are not themselves bounded.  In order to rule this out we need a second argument by contradiction.

\begin{prop}\label{prop: UYargumentTypeB}
    Suppose that the triple $\cE= (E_1, E_2, \Phi)$ has $\cos(\htheta)>0$, $\tan(\htheta)\leq \frac{\sigma}{4\pi}$ and is stable in the sense of Definition~\ref{def: stabOnXxP} or, equivalently, Definition~\ref{defn: stabOfTriple}.  Suppose that $(h_{1,t}, h_{2,t})$ are smooth solutions of~\eqref{eq: dimReduceDHYM-MOC} for $t\in [0, 1]$ satisfying
    \begin{equation}\label{eq: normalizationCondition}
    \int_{X}\log h_{2,0}^{-1}h_{2,t} + \log (\det(h_{1,0}^{-1}h_{1,t})) =0
    \end{equation}
    and $\inf_{X}|\Phi|^2 \leq \frac{4\pi}{\sigma t}$. Then there is a uniform constant $C>0$ depending on $\cE, X$ and $(h_{1,0}, h_{2,0})$ so that
\[
 C^{-1} h_{1,0} \leq h_{1,t} \leq  Ch_{1,0} \qquad C^{-1}h_{2,0} \leq  h_{2,t}\leq Ch_{2,0}.
\]
\end{prop}

\begin{proof}
    We employ a second blow-up argument.  Let $u_{1,t} = h_{1,0}^{-1}h_{1,t}$ be the relative endomorphism of $E_1$ relating $h_{1,0}$ and $h_{1,t}$.  As in the proof of Proposition~\ref{prop: UYargumentH}, we note that $u_{1,t}$ is Hermitian with respect to both $h_{1,t}$ and $h_{1,0}$, and $u_{1,t}$ is positive definite.  Similarly we define $u_{2,t} = h_{2,0}^{-1}h_{2,t}$. We first prove the upper bounds. Assume there exists a sequence of times $t_{\ell} \in [0,1]$ with $t_{\ell}\rightarrow T$ such that either:
    \begin{itemize}
    \item[$(i)$] The largest eigenvalue of $u_{1,t_{\ell}} \rightarrow \infty$, and/or
        \item[$(ii)$] $u_{2,t_{\ell}}\rightarrow +\infty$.
    \end{itemize}
    To ease notation, let us denote $u_{1,t_{\ell}}= u_{1,\ell}$, and $u_{2,t_{\ell}}= u_{2,\ell}$. By Proposition~\ref{prop: UYargumentH} we have that
    \[
    u_{1,\ell} \leq C_0 I_{E_1}u_{2,\ell}
    \]
    for a uniform constant $C_0$, and so if $(i)$ holds, then necessarily so does $(ii)$.  Let $m_{\ell}=C_0\sup_{X}u_{2,\ell}$ and define
    \[
    \widetilde{u}_{1,\ell} = m_{\ell}^{-1}u_{1,\ell} \qquad \widetilde{u}_{2,\ell} = m_{\ell}^{-1}u_{2,\ell} 
    \]
   Note that this rescaling preserves the metric $h_{\ell} = h_{1,\ell}h_{2,\ell}^{-1}$ on $E_1\otimes E_2^{\vee}$. Then we have $ \widetilde{u}_{1,\ell} \leq I_{E_1}$.  As in the proof of Proposition~\ref{prop: UYargumentH} we follow the Uhlenbeck-Yau technique, and consider $\widetilde{u}_{1,\ell}^{\mathtt{s}}$, $\widetilde{u}_{2,\ell}^{\mathtt{s}}$, for $\mathtt{s} \in (0,1]$.  Using the curvature bounds from Corollary~\ref{cor: totalCurvBound}, and arguing as in the proof of Proposition~\ref{prop: UYargumentH}, the results of Uhlenbeck-Yau \cite{UY} yield that, along a subsequence
    \[
    \begin{aligned}
    \lim_{\mathtt{s}\rightarrow 0}\lim_{\ell \rightarrow \infty} \widetilde{u}_{1,\ell}^{\mathtt{s}} &= 1-\pi_1 \qquad \text{ weakly in } \quad W^{1,2}(X,{\rm End}(E_1), h_{1,0})\\
    \lim_{\mathtt{s}\rightarrow 0}\lim_{\ell \rightarrow \infty} \widetilde{u}_{2,\ell}^{\mathtt{s}} &= 1-\pi_2\qquad \text{ weakly in } \quad W^{1,2}(X)
    \end{aligned}
    \]
    where the convergence takes place weakly in $W^{1,2}$, strongly in $L^2$ and pointwise a.e.  By the curvature bounds, $\pi_2$ defines a proper subbundle of $E_2$, which is necessarily the zero subbundle since $E_2$ has rank $1$.  

    We claim that $\pi_1$ defines a non-zero sub-bundle $S\subset E_1$, which may possibly be all of $E_1$.  To see this first recall that if $\Delta f \in L^{\infty}$ on a compact Riemann surface then using the Green's function we have
    \[
    \osc_{X} f \leq C\|\Delta f\|_{L^{\infty}(X)}.
    \]
    Now we observe that 
    \[
    \Delta \log u_{2,\ell} = i\Lambda F_{2,0}-i\Lambda F_{2,\ell}
    \]
     is uniformly bounded by Corollary~\ref{cor: totalCurvBound}.  Thus if $\sup_{X} u_{2,\ell} =C_0^{-1}m_{\ell} \gg 1$, then
    \[
    \int_{X}\log u_{2,\ell} \geq V_X(\log(C_{0}^{-1}m_{\ell})-C)
    \]
    for a uniform constant $C$.  It follows from the normalization condition~\eqref{eq: normalizationCondition} that
    \[
    \int_{X}\log\det u_{1,\ell}\leq  -V_X(\log(C_{0}^{-1}m_{\ell})-C)
    \]
and hence there exists a point $p\in X$ where $u_{1,\ell}$ has an eigenvalue less than or equal to $1$. Since
\[
\Delta \log \det u_{1,\ell}= {\rm Tr}(i\Lambda F_{1,0}) - {\rm Tr}(i\Lambda F_{1,\ell})
\]
is bounded by Corollary~\ref{cor: totalCurvBound}, there exists a $\delta >0$ and an open set $U$  with $|U|>\delta$ such that $\log \det u_{1,\ell} \leq 0$ on $U$.  It follows that $1-\pi_1 \neq I_{E_1}$, and hence $\pi_1$ defines a non-zero subbundle $S\subset E_1$.

Furthermore, since $\sup_{X}\widetilde{u}_{2,\ell} = C_{0}^{-1}$ the preceding Green's function bound yields a constant $\delta '>0$, independent of $\ell$ so that we have
\[
\widetilde{u}_{2,\ell} \geq \delta' >0.
\]

To extract a contradiction we appeal to stability of $\cE$ with respect to the saturated subtriple $(S,0,0)$.  Let $\beta_{1,0}$ denote the second fundamental form of $S\subset E_1$ defined using the metric $h_{1,0}$.  We have
\[
\begin{aligned}
2\pi\deg(S) &= \int_{X}{\rm Tr}(i\Lambda F_{1,0}\pi) - \int_{X}|\beta_{1,0}|^2_{h_{1,0}}\\
&=  \lim_{\mathtt{s} \rightarrow 0}\lim_{\ell \rightarrow \infty} \int_{X}{\rm Tr}(i\Lambda F_{1,0}(I_{E_{1}}-\widetilde{u}_{1,\ell}^{\mathtt{s}} )) - \int_{X}|\beta_{1,0}|^2_{h_{1,0}}\\
& = \lim_{\mathtt{s} \rightarrow 0}\lim_{\ell \rightarrow \infty} \left[\int_{X}{\rm Tr}(i\Lambda F_{1,\ell}(I_{E_1}-\widetilde{u}_{1,\ell}^{\mathtt{s}} )) +\int_{X}{\rm Tr}\left(\dbar\left(u_{1,\ell}^{-1}\left( \nabla^{1,0}u_{1,\ell}\right)\right)(I_{E_1}-\widetilde{u}_{1,\ell}^{\mathtt{s}})\right)\right]\\
&\quad - \int_{X}|\beta_{1,0}|^2_{h_{1,0}}\\
&=  \lim_{\mathtt{s} \rightarrow 0}\lim_{\ell \rightarrow \infty} \left[\int_{X}{\rm Tr}(i\Lambda F_{1,\ell}(I_{E_1}-\widetilde{u}_{1,\ell}^{\mathtt{s}} )) +\int_{X}{\rm Tr}\left(u_{1,\ell}^{-1}\left( \nabla^{1,0}u_{1,\ell}\right)\dbar \widetilde{u}_{1,\ell}^{\mathtt{s}}\right)\right]\\
&\quad - \int_{X}|\beta_{1,0}|^2_{h_{1,0}}\\
&\geq  \lim_{\mathtt{s} \rightarrow 0}\lim_{\ell \rightarrow \infty} \int_{X}{\rm Tr}(i\Lambda F_{1,\ell}(I_{E_1}-\widetilde{u}_{1,\ell}^{\mathtt{s}} ))
\end{aligned}
\]
where in the last line we used ~\eqref{eq: UYFormula1}, together with the bound $\widetilde{u}_{1,\ell}^{\mathtt{s}} \leq I_{E_1}$ and the lower semi-continuity of the $W^{1,2}({\rm End}(E_1), h_{1,0})$ norm along weakly converging sequences.  Applying stability, specifically Lemma~\ref{lem: phaseToRatioStab}, for the saturated subtriple $(S,0,0)$ we have that
\[
2\pi \deg(S) < \frac{\sin(\htheta)}{\cos(\htheta)} \rank(S) V_{X}
\]
and so we will obtain a contradiction if we can show that
\[
\lim_{\mathtt{s} \rightarrow 0}\lim_{\ell \rightarrow \infty} \int_{X}{\rm Tr}(i\Lambda F_{1,\ell}(I_{E_1}-\widetilde{u}_{1,\ell}^{\mathtt{s}} )) \geq \frac{\sin(\htheta)}{\cos(\htheta)}{\rm rk}(S)V_{X}.
\]
We apply the equation ~\eqref{eq: dimReduceDHYM-MOC} for $F_{1,\ell}$  which yields
\begin{equation}\label{eq: UYcontradictionPart2CurvTerm1}
\begin{aligned}
\cos(\htheta)\int_{X}{\rm Tr}(i\Lambda F_{1,\ell}(I_{E_1}-\widetilde{u}_{1,\ell}^{\mathtt{s}} ))&=-\cos(\htheta)\int_{X}{\rm Tr}(\Phi\Phi^{\dagger_{\ell}}(I_{E_1}-\widetilde{u}_{1,\ell}^{\mathtt{s}} ))\\
&\quad +\sin(\htheta) \int_{X}{\rm Tr}(I_{E_1}-\widetilde{u}_{1,\ell}^{\mathtt{s}} )\\
&\quad -\frac{1}{2}\sin(\htheta)\int_{X}{\rm Tr}\left(\{i\Lambda F_{1,\ell}, \Phi\Phi^{\dagger_{\ell}}\}(I_{E_1}-\widetilde{u}_{1,\ell}^{\mathtt{s}} )\right)\\
&\quad +\frac{1}{2}\sin(\htheta)\int_{X}{\rm Tr}\left((\nabla^{\ell} \Phi)(\nabla^{\ell}\Phi)^{\dagger_{\ell}}(I_{E_1}-\widetilde{u}_{1,\ell}^{\mathtt{s}} )\right)
\end{aligned}
\end{equation}
Note that we can freely take real parts on the left and right, since both sides are real.

\begin{lem}\label{lem: PhiDaggerGoesToZero}
    In the above setting, there is a uniform constant $C>0$ independent of $\ell$ so that 
\[
|\Phi^{\dagger_{\ell}}(I_{E_1}-\widetilde{u}_{1,\ell}^{\mathtt{s}})|^2_{h_{\ell}} \leq Cs^2
\]
\end{lem}
\begin{proof}
Let $e_{\alpha}$ be an $h_{1,0}$ orthonormal frame of eigenvectors of $\widetilde{u}_{1,\ell}$, and let $\xi$ be an $h_{2,0}$ section of $E_2$ with length $1$.  Then we have
\[
\tilde{u}_{1,\ell} = \mu_{1,\ell}^{\alpha} e_\alpha \otimes e_{\alpha}^* \qquad \widetilde{u}_{2,\ell} = \mu_{2,\ell}  \xi \otimes \xi^{*}
\]
We compute
\[
\Phi^{\dagger_{\ell}}= \sum_{\alpha} \overline{\Phi^{\alpha}}\frac{\mu_{1,\ell}^{\alpha}}{\mu_{2,\ell}} e_{\alpha}^*
\]
and so
\[
\Phi^{\dagger_{\ell}}(I_{E_1}-\widetilde{u}_{1,\ell}^{s}) = \sum_{\alpha} \overline{\Phi^{\alpha}}\frac{\mu_{1,\ell}^{\alpha}}{\mu_{2,\ell}}(1-(\mu_{1,\ell}^{\alpha})^{s}) e_{\alpha}^*
\]
With respect to $h_{\ell}$ we have 
\[
\langle e_{\alpha}^*\otimes \xi, e_{\alpha}^*\otimes \xi\rangle_{h_{\ell}} =  \frac{\mu_{2,\ell}}{\mu_{1,\ell}^{\alpha} }.
\]
Thus,
\[
\begin{aligned}
|\Phi^{\dagger_{\ell}}(I_{E_1}-\widetilde{u}_{1,\ell}^{s})|_{h_{\ell}}^2 &= \sum_{\alpha} |\Phi^{\alpha}|^2\left(\frac{\mu_{1,\ell}^{\alpha}}{\mu_{2,\ell}}\right)^2(1-(\mu_{1,\ell}^{\alpha})^{s})^2\frac{\mu_{2,\ell}}{\mu_{1,\ell}^{\alpha} }\\
&=\sum_{\alpha} |\Phi^{\alpha}|^2\left(\frac{\mu_{1,\ell}^{\alpha}}{\mu_{2,\ell}}\right)(1-(\mu_{1,\ell}^{\alpha})^{s})^2
\end{aligned}
\]
Now since $0<\delta' < \mu_{2,\ell} \leq C_0^{-1}$, and $\mu_{1,\ell}^{\alpha}\in(0,1)$ we have
\[
\sum_{\alpha} |\Phi^{\alpha}|^2\left(\frac{\mu_{1,\ell}^{\alpha}}{\mu_{2,\ell}}\right)(1-(\mu_{1,\ell}^{\alpha})^{s})^2 \leq C |\Phi|_{h_0}^2 s^2
\]
where we used that, for any $s\in (0,1)$
\[
\sup_{x\in[0,1]}x(1-x^s)^2 \leq Cs^2
\]
for a uniform constant $C$.
\end{proof}

Now we can finish the proof. From Lemma~\ref{lem: PhiDaggerGoesToZero} and the curvature bounds (Corollary~\ref{cor: totalCurvBound}) and the upper bound for $|\Phi|_{h_{\ell}}$ (Lemma~\ref{lem: conditionalPhiBound}) that we have
\[
\begin{aligned}
&\lim_{\mathtt{s} \rightarrow 0}\lim_{\ell\rightarrow\infty}\int_{X}{\rm Tr}(\Phi\Phi^{\dagger_{\ell}}(I_{E_1}-\widetilde{u}_{1,\ell}^{\mathtt{s}} ))=0\\
&\lim_{\mathtt{s} \rightarrow 0}\lim_{\ell \rightarrow\infty}\int_{X}\Tr\left(i\Lambda F_{1,\ell}\Phi\Phi^{\dagger_\ell}(I_{E_1}-\widetilde{u}_{\ell}^{\mathtt{s}})\right)=0.
\end{aligned}
\]
This leads to the following simplification of~\eqref{eq: UYcontradictionPart2CurvTerm1}
\[
\begin{aligned}
\lim_{\mathtt{s}\rightarrow 0}\lim_{\ell \rightarrow \infty}\cos(\htheta)
&\int_{X}{\rm Re}\left({\rm Tr}(i\Lambda F_{1,\ell}(I_{E_1}-\widetilde{u}_{1,\ell}^{\mathtt{s}} ))\right)= \sin(\htheta) \rank(S)V_{X}\\
&\quad -\frac{1}{2}\sin(\htheta) \lim_{\mathtt{s}\rightarrow 0}\lim_{\ell \rightarrow \infty}\underbrace{\int_{X}{\rm Re}\left({\rm Tr}\left(\Phi\Phi^{\dagger_{\ell}}i\Lambda F_{1,\ell}(I_{E_1}-\widetilde{u}_{1,\ell}^{\mathtt{s}} )\right)\right)}_{(I)}\\
&+ \frac{1}{2}\sin(\htheta) \lim_{\mathtt{s}\rightarrow 0}\lim_{\ell \rightarrow \infty}\underbrace{\int_{X}{\rm Re}\left({\rm Tr}\left(\nabla^{\ell} \Phi(\nabla^{\ell}\Phi)^{\dagger_{\ell}}(I_{E_1}-\widetilde{u}_{1,\ell}^{\mathtt{s}} )\right)\right)}_{(II)}
\end{aligned}
\]
For the remaining terms, observe that
\[
\begin{aligned}
\overline{{\rm Tr}\left(\Phi\Phi^{\dagger_{\ell}}i\Lambda F_{1,\ell}(I_{E_1}-\widetilde{u}_{1,\ell}^{\mathtt{s}} )\right)} &= {\rm Tr}\left(\left(\Phi\Phi^{\dagger_{\ell}}i\Lambda F_{1,\ell}(I_{E_1}-\widetilde{u}_{1,\ell}^{\mathtt{s}} )\right)^{\dagger_{\ell}}\right)\\
&={\rm Tr}\left( (I_{E_1}-\widetilde{u}_{1,\ell}^{\mathtt{s}} )i\Lambda F_{1,\ell}\Phi\Phi^{\dagger_{\ell}}\right)\\
&={\rm Tr}\left(i\Lambda F_{1,\ell}\Phi\Phi^{\dagger_{\ell}}(I_{E_1}-\widetilde{u}_{1,\ell}^{\mathtt{s}} )\right)
\end{aligned}
\]
and so, by Lemma~\ref{lem: PhiDaggerGoesToZero} there is a uniform constant $C>0$, independent of $s,\ell$ so that
\[
|(I)| \leq C\sup_{X}|\Phi|_{h_{\ell}}\cdot \sup_{X}|i\Lambda F_{1,\ell}|_{h_{\ell}} \cdot s \leq Cs
\]
using again the uniform bounds for $|\Phi|_{h_{\ell}},$ and $|i\Lambda F_{1,\ell}|_{h_{\ell}}$.  The term $(II)$ is plainly positive, since $\tilde{u}_{1,\ell}^{s} \leq I_{E_1}$.

In total, we conclude that
\[
\lim_{\mathtt{s}\rightarrow 0}\lim_{\ell \rightarrow \infty}\cos(\htheta)
\int_{X}{\rm Re}\left({\rm Tr}(i\Lambda F_{1,\ell}(I_{E_1}-\widetilde{u}_{1,\ell}^{\mathtt{s}} ))\right) \geq  \sin(\htheta) \rank(S)V_{X}
\]
and so $
\left( \cos(\htheta) \deg(S) - \frac{V}{2\pi}\sin(\htheta) {\rm rk}(S)\right) \geq 0$  contradicting stability.  This establishes the upper bounds
\[
h_{1,t} \leq Ch_{1,0}, \qquad h_{2,t} \leq Ch_{2,0}
\]
for a uniform constant $C$.  The lower bounds follow immediately from the upper bounds from Corollary~\ref{cor: totalCurvBound} and the normalization condition~\eqref{eq: normalizationCondition}.

\end{proof}

Combining our results yields estimates to all orders.

\begin{prop}\label{prop: stableImpliesEstimatesToAllOrders}
   Suppose that $\cE= (E_1, E_2, \Phi)$ has $\cos(\htheta)>0$, $\tan(\htheta) \leq \frac{\sigma}{4\pi}$ and $\cE$ is stable in the sense of Definition~\ref{def: stabOnXxP} or equivalently Definition~\ref{defn: stabOfTriple}.  Suppose that $(h_{1,t}, h_{2,t})$ are solutions of~\eqref{eq: dimReduceDHYM-MOC} for some $t\in [0, 1]$ satisfying $\inf_{X}|\Phi|^2 < \frac{4\pi}{\sigma t}$.  Let $(h_{1,0}, h_{2,0})$ be fixed reference metrics and normalize $(h_{1,t}, h_{2,t})$ by
   \[
    \int_{X}\log h_{2,0}^{-1}h_{2,t} + \log (\det(h_{1,0}^{-1}h_{1,t})) =0.
  \]
   For each $k\in \mathbb{Z}_{>0}$ and each $\alpha \in (0,1)$ there is a uniform constant $C_k>0$ depending on $k, \alpha, X, g, E_1, E_2, \Phi$, $(h_{1,0},h_{2,0})$ so that
\[
 \|h_{1,t}\|_{C^{k,\alpha}(X,g, h_{1,0})} + \|h_{2,t}\|_{C^{k,\alpha}(X,g, h_{2,0})} \leq C_{k}. 
\]
\end{prop}
\begin{proof}
We have established in Corollary~\ref{cor: totalCurvBound} that the curvatures of $h_{1,t}$ and $h_{2,t}$  are uniformly bounded.  By Proposition~\ref{prop: UYargumentTypeB}, the metrics $h_{i,t}$ are uniformly equivalent to $h_{i,0}$.  Under these assumptions the curvature becomes a uniformly elliptic operator of divergence form and  the higher order regularity now follows from a boot-strap argument as in Uhlenbeck-Yau \cite{UY}.  
\end{proof}

\section{Fixed point theorem and completion of the proof}\label{sec: fixedPoint}

In this section we complete the proof of Theorem~\ref{thm: mainTheorem} by an application of the Leray-Schauder degree theory. The first step is to recast the system~\eqref{eq: dimReduceDHYM-MOC} for $\cos(\htheta)>0$ as a fixed point problem for a certain non-linear operator. 

To begin, fix metrics $(h_{1,0},h_{2,0})$ on $(E_1,E_2)$ solving the system~\eqref{eq: dimReduceDHYM-MOC} at time $t=0$.  Since we are in the regime where $\cos(\htheta)>0$ we can recast the system~\eqref{eq: dimReduceDHYM-MOC} using Lemma~\ref{lem: H1CurveLemSimplfiedPositiveCos} and Lemma~\ref{lem: H2CurvLem} as
\begin{equation}\label{eq: systemForSchauder}
    \begin{aligned}
    i\Lambda F_1 &= \alpha I - \frac{1+t\alpha^2}{1+t\alpha |\Phi|^2}\Phi\Phi^{\dagger}+ \frac{t\alpha}{2}\nabla \Phi(\nabla \Phi)^{\dagger} - \frac{t^2\alpha^2}{2(2+t\alpha|\Phi|^2)}\{\Phi\Phi^{\dagger}, \nabla \Phi(\nabla \Phi)^{\dagger}\}\\
&\quad + \frac{t^3\alpha^3 |\langle \nabla \Phi, \Phi\rangle|^2}{2(2+t\alpha |\Phi|^2)(1+t\alpha|\Phi|^2)}\Phi\Phi^{\dagger} \\
\,\\
i\Lambda F_{2} &=\left(\frac{ -\cos(\hat{\theta})\left(\frac{4\pi}{\sigma} - |\Phi|^2\right) +\sin(\hat{\theta})\left(1 + \frac{t}{2}|\nabla \Phi|^2\right)}{\left(\cos(\hat{\theta}) + \left(\frac{4\pi}{\sigma}-t|\Phi|^2\right)\sin(\hat{\theta})\right)} \right) 
\end{aligned}
\end{equation}
Given metrics $h_1$ on $E_1$ and $h_2$ on $E_2$ we write
\[
h_1= h_{1,0}e^{u_1} \qquad h_{2}= h_{2,0}e^{u_2}
\]
where $u_1 \in {\rm Hom}(E_1,E_1)$ is $h_{1,0}$-Hermitian, and $u_{2}$ is a real valued function.  We can rewrite the curvature of $h_1$ and $h_2$ in terms of the curvature of $h_{1,0}$  and $h_{2,0}$ as follows
\[
i\Lambda F_{1} = i\Lambda F_{h_{1,0}} - i\Lambda \dbar (e^{-u_1}\hat{\nabla}e^{u_1}) \qquad i\Lambda F_{2} = i\Lambda F_{h_{2,0}} - i\Lambda \dbar \del u_2
\]
where $\hat{\nabla}$ denotes the Chern connection defined by the background metrics $h_{i,0}$ for $i=1,2$; see~\eqref{eq: curvatureRelation}. We need the following elementary lemma.

\begin{lem}\label{lem: PoperatorCurvature}
\[
  i\Lambda \dbar (e^{-u_1}\hat{\nabla}e^{u_1})= \mathtt{P}({\rm ad}_{u_1})\Delta u + Q(u_1, \hat{\nabla}u_1)
\]
where $Q$ is quadratic, and $\mathtt{P}$ is defined by the convergent power series
\[
\mathtt{P}(z) = \frac{1-e^{-z}}{z} = \sum_{k=1}^{\infty}\frac{(-z)^{k-1}}{k!}
\]  
\end{lem}
\begin{proof}
Consider
\[
f(s) = e^{-su_1}\hat{\nabla}e^{su_1}
\]
and recall that ${\rm ad}_u(A) = [u,A]$. Differentiating $f(s)$ with respect to $s$ yields
\[
\begin{aligned}
f'(s) &= -u_1e^{-su_1}\hat{\nabla}e^{su_1} + e^{-su_1}\hat{\nabla}(u_1e^{su_1})\\
&= -u_1e^{-su_1}\hat{\nabla}e^{su_1} + e^{-s{\rm ad}_u}\hat{\nabla} u_1 + u_1e^{-su_1}\hat{\nabla} e^{su_1}
\end{aligned}
\]
using that $u_1$ and $e^{\pm su_1}$ commute.  Canceling the first and last terms yields
\[
f'(s) = e^{-s{\rm ad}_u}\hat{\nabla} u_1
\]
Thus, since $f(0)=0$, we obtain
\[
f(1) = \int_0^1f'(s) ds = \left(\int_{0}^1e^{-s{\rm ad}_{u_1}}ds\right)\hat{\nabla} u_1.
\]
The operator ${\rm ad}_{u_1}$ acts linearly so we can integrate this directly to obtain
\[
f(1) = \mathtt{P}({\rm ad}_{u_1})\hat{\nabla} u_1.
\]
After applying $\dbar$ we obtain the desired result.
\end{proof}

Consider the following Banach spaces:

\[
\mathcal{B}_{k,\alpha}
=\left\{(u_1,u_2)\in C^{k,\alpha}\left(\Herm_{h_{1,0}}(E_1)\right)\oplus C^{k,\alpha}(X,\R):
\int_X (\Tr(u_1)+u_2)=0\right\}.
\]
where the $C^{k,\alpha}$ norms are defined by the metrics $h_{1,0}$ and $h_{2,0}$. Define the linear elliptic operator
\[
\mathcal{L}(u_1,u_2)
=
\begin{pmatrix}
\cos(\htheta)\left(-i\Lambda\dbar\hat{\nabla}u_1+\Phi\Phi^{\dagger_0}u\right)\\
-\left(\cos(\htheta)+\frac{4\pi}{\sigma}\sin(\htheta)\right)i\Lambda\dbar\del u_2-\cos(\htheta)\Phi^{\dagger_0}u\Phi
\end{pmatrix}
\]
where $\dagger_{0}$ denotes the adjoint defined by the initial metrics $h_{i,0}$, and $u=u_1-u_2I$. The following lemma shows that $\mathcal{L}$ maps  $h_{i,0}$-Hermitian endomorphisms to $h_{i,0}$-Hermitian endomorphisms.  This is not a priori clear since $\rank(E_1)>1$ and the operator $u\mapsto \Phi\Phi^{\dagger_0}u$ does not preserve Hermitian endomorphisms.

\begin{lem}
\label{lem:sec10-linear-Hermitianity}
Let $u_1$ be $h_{1,0}$-Hermitian and let $u_2$ be real-valued.  If $(h_{1,0},h_{2,0})$ solve the equation~\eqref{eq: dimReduceDHYM-MOC} at $t=0$, then $-i\Lambda\dbar\hat\nabla u_1+\Phi\Phi^{\dagger_0}u$ is $h_{1,0}$-Hermitian.
\end{lem}

\begin{proof}
By the definition of $u$ we have
\[
\Phi\Phi^{\dagger_0}u=\Phi\Phi^{\dagger_0}u_1-u_2\Phi\Phi^{\dagger_0}.
\]
Since $u_2$ is real and $\Phi\Phi^{\dagger_0}$ is $h_{1,0}$-Hermitian, we have
\[
\left(\Phi\Phi^{\dagger_0}u_1-u_2\Phi\Phi^{\dagger_0}\right)^{\dagger_0}
=u_1\Phi\Phi^{\dagger_0}-u_2\Phi\Phi^{\dagger_0}
=\Phi\Phi^{\dagger_0}u_1-u_2\Phi\Phi^{\dagger_0}
+[u_1,\Phi\Phi^{\dagger_0}].
\]
We also have the following curvature equality:
\[
\left(-i\Lambda\dbar\hat\nabla u_1\right)^{\dagger_0}
=-i\Lambda\dbar\hat\nabla u_1+[u_1,i\Lambda F_{h_{1,0}}].
\]
Using that the metrics $h_{i,0}$ solve the $t=0$ equation we have
\[
i\Lambda F_{h_{1,0}}=\frac{\sin(\htheta)}{\cos(\htheta)}I-\Phi\Phi^{\dagger_0},
\]
and so we obtain $[u_1,i\Lambda F_{h_{1,0}}]=-[u_1,\Phi\Phi^{\dagger_0}]$. Therefore
\[
\begin{aligned}
\left(-i\Lambda\dbar\hat\nabla u_1+\Phi\Phi^{\dagger_0}u_1-u_2\Phi\Phi^{\dagger_0}\right)^{\dagger_0}
&=-i\Lambda\dbar\hat\nabla u_1-[u_1,\Phi\Phi^{\dagger_0}]\\
&\quad+\Phi\Phi^{\dagger_0}u_1-u_2\Phi\Phi^{\dagger_0}
+[u_1,\Phi\Phi^{\dagger_0}]\\
&=-i\Lambda\dbar\hat\nabla u_1+\Phi\Phi^{\dagger_0}u_1-u_2\Phi\Phi^{\dagger_0}.
\end{aligned}
\]
which is the desired result.
\end{proof}

\begin{lem}
\label{lem:sec10-linear-isomorphism}
Suppose that the triple $\cE = (E_1, E_2,\Phi)$ is stable in the sense of Definition~\ref{def: stabOnXxP}. Then the operator $\cL$ defines an isomorphism
\[
\cL:\mathcal{B}_{k+2,\alpha}\longrightarrow \mathcal{B}_{k,\alpha}.
\]
\end{lem}

\begin{proof}
Recall that by Proposition~\ref{prop: stableImpliesSimple} if $\cE$ is stable, then it is simple in the sense of Definition~\ref{defn: PhiSimple}.  This is all we will need in the proof. First we check that the image of $\cL$ lies in $\mathcal{B}_{k,\alpha}$.
By integration by parts we have:
\[
\int_X\Tr\left((\cL u)_1\right)+(\cL u)_2
=-i\cos(\hat\theta)\int_X \Tr \dbar\nabla u_1 -i\left(\cos(\hat\theta) + \frac{4\pi}{\sigma}\sin(\hat\theta)\right)\int_X \Tr \dbar\nabla u_2 =0.
\]
Thus $\cL u\in \mathcal{B}_{k,\alpha}$. We now investigate the kernel. By integration-by-parts
\[
\int_X\Tr\left(\left(-i\Lambda\dbar\hat\nabla u_1\right)u_1\right)
=\int_X|\hat\nabla u_1|^2,
\]
and similarly
\[
\int_X\left(-i\Lambda\dbar\del u_2\right)u_2
=\int_X|\del u_2|^2.
\]
Further we have
\[
\int_X\Tr\left((\Phi\Phi^{\dagger_0}u)u_1\right)
-\int_X(\Phi^{\dagger_0}u\Phi)u_2
=\int_X|u\Phi|^2.
\]
Combining these identities gives
\[
\begin{aligned}
\left\langle \cL u,u\right\rangle=\cos(\htheta)\int_X |\hat{\nabla}u_1|^2+\left(\cos(\htheta)+\frac{4\pi}{\sigma}\sin(\htheta)\right)\int_X |\del u_2|^2+\cos(\htheta)\int_X |u\cdot\Phi|^2.
\end{aligned}
\]
Hence if $\cL u=0$, we have:
\[
\hat{\nabla}u_1=0,\qquad \del u_2=0,\qquad u\Phi=0.
\]
Thus $u_2$ is constant, and $u_1$ is a $h_{1,0}$-parallel, Hermitian endomorphism of $E_1$ which implies that $u_1$ is holomorphic.  Furthermore, since $u\Phi=0$, $u$ defines an endomorphism of the triple $(E_1,E_2,\Phi)$. Since the triple is simple by assumption we have
\[
u_1=\lambda I_{E_1},\qquad u_2=\lambda
\]
for some $\lambda\in \R$.  Using the normalization restriction of $\mathcal{B}_{k+2,\alpha}$ we have:
\[
0=\int_X\Tr(u_1)+u_2
=\lambda(\rank(E_1)+1)\vol(X),
\]
Thus $\lambda=0$ and $\ker(\cL)=0$ on $\mathcal{B}_{k+2,\alpha}$. Since $\cL$ is elliptic, self-adjoint and the kernel vanishes, $\cL$ is an isomorphism by the standard elliptic regularity for uniformly elliptic divergence-type systems~\cite{ChenWu}.
\end{proof}

\subsection{The fixed point problem}
\label{sec:sec10-compact-map}

We now return to the nonlinear system. We write the system~\eqref{eq: systemForSchauder} as $\cF_t=0$, where $\cF_t=(\cF_{1,t},\cF_{2,t})$ is given by
\[
\cF_{1,t}=i\Lambda F_{1}+R_{1,t}(h_{1,t},h_{2,t},\Phi,\nabla^t\Phi),\qquad
\cF_{2,t}=i\Lambda F_{2}+R_{2,t}(h_{1,t},h_{2,t},\Phi,\nabla^t\Phi).
\]
The nonlinear operators $R_{1,t},R_{2,t}$ contain no second derivatives of the metric.  From now on we suppress the dependence on $t$, with it understood that, unless otherwise noted, all quantities are computed using the metrics $(h_{1,t},h_{2,t})$.  We will be more precise about the operators $R_{1,t},R_{2,t}$ since it will be needed later.  The system~\eqref{eq: dimReduceDHYM-MOC} can be written in terms of the operators 
\begin{equation}\label{eq: frakF_1def}
\mathfrak{F}_{1,t}(u_1,u_2):= T_{\Phi,t}(i\Lambda F_1) -\left(\sin(\htheta) I - \cos(\htheta)\Phi\Phi^{\dagger} + \frac{t}{2}\sin(\htheta)\nabla \Phi (\nabla \Phi)^{\dagger}\right)
\end{equation}
and
\begin{equation}\label{eq: frakF_2def}
\begin{aligned}
\mathfrak{F}_{2,t}(u_1,u_2) :=& \left(\cos(\hat{\theta}) + \left(\frac{4\pi}{\sigma}-t|\Phi|^2\right)\sin(\hat{\theta})\right)i\Lambda F_{2}\\
&\quad -\left( -\cos(\hat{\theta})\left(\frac{4\pi}{\sigma} - |\Phi|^2\right) +\sin(\hat{\theta})\left(1 + \frac{t}{2}|\nabla \Phi|^2\right) \right) 
\end{aligned}
\end{equation}
where $T_{\Phi,t}$ is the linear operator introduced in~\eqref{eq: TPhiOperator}.  By integration by parts we have
\[
\int_{X}\Tr(\mathfrak{F}_{1,t}(u_1,u_2))+\mathfrak{F}_{2,t}(u_1,u_2)=0.
\]
Then
\[
\cF_{1,t}=T_{\Phi,t}^{-1}(\mathfrak{F}_{1,t}), \qquad \cF_{2,t} = \frac{\mathfrak{F}_{2,t}}{\left(\cos(\hat{\theta}) + \left(\frac{4\pi}{\sigma}-t|\Phi|^2\right)\sin(\hat{\theta})\right)}
\]
From these expressions we can find the operators $R_{1,t}$ and $R_{2,t}$ explicitly.  In terms of the operator $T_{\Phi,t}$ introduced in~\eqref{eq: TPhiOperator} and its inverse, as computed in Lemma~\ref{lem: TinverseFormula} we have
\[
\begin{aligned}
R_{1,t} &= -T_{\Phi,t}^{-1}\left(\sin(\htheta) I - \cos(\htheta)\Phi\Phi^{\dagger} + \frac{t}{2}\sin(\htheta)\nabla \Phi (\nabla \Phi)^{\dagger}\right)\\
R_{2,t} &=-\frac{\left( -\cos(\hat{\theta})\left(\frac{4\pi}{\sigma} - |\Phi|^2\right) +\sin(\hat{\theta})\left(1 + \frac{t}{2}|\nabla \Phi|^2\right) \right)}{\left(\cos(\hat{\theta}) + \left(\frac{4\pi}{\sigma}-t|\Phi|^2\right)\sin(\hat{\theta})\right)}
\end{aligned}
\]
Evidently $\mathcal{F}_{1,t}$ is $h_{1,t}$-Hermitian, while $\cF_{2,t}$ is real valued.  We define the corrected nonlinear error term by
\[
\mathtt{E}_t(u)=
\begin{pmatrix}
\cos(\htheta)\left(-i\Lambda\dbar\hat{\nabla}u_1+\Phi\Phi^{\dagger_0}u
-\mathtt{S}({\rm ad}_{u_1})^{-1}e^{u_1/2}\cF_{1,t}(u)e^{-u_1/2}\right)\\
-\left(\cos(\htheta)+\frac{4\pi}{\sigma}\sin(\htheta)\right)i\Lambda\dbar\del u_2
-\cos(\htheta)\Phi^{\dagger_0}u\Phi
-\left(\cos(\htheta)+\frac{4\pi}{\sigma}\sin(\htheta)\right)\cF_{2,t}(u)
\end{pmatrix}.
\]
where $\mathtt{S}({\rm ad}_{u_1})$ is defined by the convergent power series associated to
\[
\mathtt{S}(z)=\frac{e^{z/2}-e^{-z/2}}{z}.
\]
The reason for introducing the operator $\mathtt{S}({\rm ad}_{u_1})$ is to preserve the $h_{i,0}$-Hermitian property of the operator.
\begin{lem}
\label{lem:sec10-s-operator}
If \(u_1\) is \(h_{1,0}\)-Hermitian, then
\(\mathtt S(\operatorname{ad}_{u_1})\) is invertible on
\(\operatorname{End}(E_1)\), maps the space of \(h_{1,0}\)-Hermitian endomorphisms to
itself, and depends smoothly on \(u_1\).
\end{lem}

\begin{proof}
Observe that the power series defining \(\mathtt S\) contains only even powers of $z$. Precisely,
\[
\mathtt S(z)
=\frac{2\sinh(z/2)}{z}
=\sum_{m=0}^{\infty}\frac{z^{2m}}{2^{2m}(2m+1)!},
\]
with \(\mathtt S(0)=1\).  In a local \(h_{1,0}\)-unitary frame $\{e_1,\ldots, e_{r}\}$ diagonalizing \(u_1\),
say \(u_1=\operatorname{diag}(\lambda_1,\ldots,\lambda_r)\), the operator
\(\operatorname{ad}_{u_1}\) acts on the matrix unit \(E_{ij}=e_i\otimes e_{j}^{\vee}\) by
\[
\operatorname{ad}_{u_1}(E_{ij})=(\lambda_i-\lambda_j)E_{ij}.
\]
The eigenvalue differences \(\lambda_i-\lambda_j\) are real, and
\[
\mathtt S(z)=\frac{2\sinh(z/2)}{z}>0
\]
for every real \(z\), with the value at \(z=0\) understood by continuity.  Hence
\(\mathtt S(\operatorname{ad}_{u_1})\) has no zero eigenvalues and is invertible.

Since \(\operatorname{ad}_{u_1}\) sends Hermitian endomorphisms to anti-Hermitian
endomorphisms and vice-versa, even
powers of \(\operatorname{ad}_{u_1}\) preserve the space of Hermitian endomorphisms.  The power series defining $\mathtt{S}$ contains only even powers, so \(\mathtt S(\operatorname{ad}_{u_1})\), and therefore
its inverse, preserve \(h_{1,0}\)-Hermitian endomorphisms.  Smooth dependence follows
easily from the same functional calculus, or equivalently from the convergent power
series.
\end{proof}

A slightly inconvenient issue is that we need to restrict our consideration to subspaces for which the denominator $\cos(\htheta)+(\frac{4\pi}{\sigma}-t|\Phi|^2)\sin(\htheta)$ appearing in the formula for $i\Lambda F_2$ is bounded away from zero.  This is necessary in order to ensure that our operators define continuous maps.  Related issues also arise in the a priori estimates of Section~\ref{sec: Estimates}.  For this reason we introduce
\[
\Omega_{k,\alpha,t} = \mathfrak{B}_{k,\alpha} \cap \bigg\{ (h_{1,0}e^{u_1},h_{2,0}e^{u_2}) : \sup_{X}|\Phi|_{h_0e^{u}}^2 \leq \frac{3\pi}{\sigma t}\bigg\}.
\]
where we have denoted $h_0e^{u} = h_0e^{u_1-u_{2}I_{E_1}}$.
\begin{lem}
\label{lem:sec10-second-order-cancellation}
The operator $\mathtt{E}_t$ contains only zeroth and first order derivatives of $(u_1,u_2)$.  More precisely,
for each $k \geq 1$, $\mathtt{E}_{t}$ defines a continuous, nonlinear map
\[
\mathtt{E}_{t}: \Omega_{k,\alpha,t}\rightarrow C^{k-1,\alpha}(\Herm_{h_{1,0}}(E_1))\oplus C^{k-1,\alpha}(\mathbb{R}).
\]
\end{lem}

\begin{proof}
Recall that by Lemma~\ref{lem: PoperatorCurvature} we have
\[
i\Lambda\dbar\left(e^{-u_1}\hat\nabla e^{u_1}\right)
=P({\rm ad}_{u_1})\,i\Lambda\dbar\hat\nabla u_1+Q(u_1,\hat\nabla u_1),
\]
where $Q$ is some quadratic function and $\mathtt{P}$ is the operator whose power series is defined by
\[
\mathtt{P}(z)=\frac{1-e^{-z}}{z}.
\]
For $h_1=h_{1,0}e^{u_1}$, the curvature transformation formula gives
\[
i\Lambda F_{1}
=i\Lambda F_{h_{1,0}}-i\Lambda\dbar\left(e^{-u_1}\hat\nabla e^{u_1}\right).
\]
Moreover,
\[
e^{-u_1}\hat\nabla e^{u_1}=\mathtt{P}({\rm ad}_{u_1})\hat\nabla u_1.
\]
After applying $\dbar$ and contracting, this becomes
\[
i\Lambda\dbar\left(e^{-u_1}\hat\nabla e^{u_1}\right)
=\mathtt{P}({\rm ad}_{u_1})\,i\Lambda\dbar\hat\nabla u_1+Q(u_1,\hat\nabla u_1),
\]
where $Q$ is universal and depends only on $u_1$ and first derivatives of $u_1$.  In
particular, $Q$ maps bounded subsets of $C^{k,\alpha}$ continuously to $C^{k-1,\alpha}$.

Since
\[
\mathtt{S}(z)=e^{z/2}\mathtt{P}(z),
\]
we have
\[
e^{{\rm ad}_{u_1}/2}\mathtt{P}({\rm ad}_{u_1})=\mathtt{S}({\rm ad}_{u_1}).
\]
Therefore
\[
\begin{aligned}
&\mathtt{S}({\rm ad}_{u_1})^{-1}
e^{u_1/2}\cF_{1,t}(u)e^{-u_1/2}\\
&\quad =
-i\Lambda\dbar\hat\nabla u_1
+\mathtt{S}({\rm ad}_{u_1})^{-1}
e^{{\rm ad}_{u_1/2}}
\left(i\Lambda F_{h_{0,1}}+R_{1,t}(\Phi,\nabla\Phi)-Q(u_1,\hat\nabla u_1)\right).
\end{aligned}
\]
The only second order term on the right is the displayed
$-i\Lambda\dbar\hat\nabla u_1$.  Hence this term cancels in the expression
\[
-i\Lambda\dbar\hat\nabla u_1+\Phi\Phi^{\dagger_0}u
-\mathtt{S}({\rm ad}_{u_1})^{-1}e^{u_1/2}\cF_{1,t}(u)e^{-u_1/2}.
\]
All remaining terms involve only $u_1,u_2$, and first derivatives of $u_1,u_2$ through
$\nabla\Phi$, and the fixed background data.  Thus the first component of $\mathtt{E}_t$ extends to a
well-defined map from $B_{k,\alpha}$ to $C^{k-1,\alpha}$.

Finally, $e^{u_1/2}\cF_{1,t}(u)e^{-u_1/2}$ is $h_{1,0}$-Hermitian because
$\cF_{1,t}(u)$ is $h_1$-Hermitian, and $\mathtt{S}({\rm ad}_{u_1})^{-1}$ preserves
$h_{1,0}$-Hermitian endomorphisms.  The other terms are $h_{1,0}$-Hermitian by Lemma~\ref{lem:sec10-linear-Hermitianity}.  Hence the image lies in $C^{k-1,\alpha}(\Herm_{h_{1,0}}(E_1))$.  The conclusion for the second component of $\mathtt{E}_t$ is clear.
\end{proof}

We can now cast the system~\eqref{eq: systemForSchauder} as a fixed point problem.  Consider the operator
\begin{equation}\label{eq: protoFixedPointOperator}
\cL u-\mathtt{E}_t(u)=
\begin{pmatrix}
\cos(\htheta)\mathtt{S}({\rm ad}_{u_1})^{-1}e^{u_1/2}\cF_{1,t}(u)e^{-u_1/2}\\
\left(\cos(\htheta)+\frac{4\pi}{\sigma}\sin(\htheta)\right)\cF_{2,t}(u)
\end{pmatrix}.
\end{equation}
Clearly if $\cL u-\mathtt{E}_t(u)=0$, then $\mathcal{F}_{1,t}(u)=0$ and $\mathcal{F}_{2,t}(u)=0$, and hence $h_i=h_{i,0}e^{u_i}$ solve the system~\eqref{eq: dimReduceDHYM-MOC} at time $t$.  Since $(E_1,E_2,\Phi)$ is $Z$-stable $\mathcal{L}$ is invertible by Lemma~\ref{lem:sec10-linear-isomorphism}. Therefore, solutions of the equation $\cL u-\mathtt{E}_t(u)=0$ are fixed points of the operator $u \mapsto \mathcal{L}^{-1}\mathtt{E}_t(u)$.  The only caveat is that $\mathtt{E}_t$ may not map $\Omega_{k+1,\alpha,t}$ to $\mathfrak{B}_{k,\alpha}$, as $\int_{X}\Tr(\mathtt{E}_t(u))$ may not vanish.  We therefore project onto this normalized space.  Let
\[
\ell(f)=\int_X\Tr(f_1)+f_2,
\]
and set
\[
\rho=\begin{pmatrix}\cos(\htheta)I_{E_1}\\
\cos(\htheta)+\frac{4\pi}{\sigma}\sin(\htheta)\end{pmatrix}.
\]
Since
\[
\ell(\rho)=\left((\rank(E_1)+1)\cos(\htheta)+\frac{4\pi}{\sigma}\sin(\htheta)\right)\vol(X),
\]
which is nonzero, we define
\[
\Pi_{\mathfrak{B}}(f)=f-\frac{\ell(f)}{\ell(\rho)}\rho.
\]
Then $\Pi_{\mathfrak{B}}(f)\in\mathcal{B}_{k,\alpha}$ and $\Pi_{\mathfrak{B}}$ restricts to the identity on
$\mathcal{B}_{k,\alpha}$.  We define the fixed point map by
\[
T_t(u):=\cL^{-1}\circ \Pi_{\mathfrak{B}}\circ \mathtt{E}_t(u).
\]
For $R>0$ let
\[
\Omega_{1,\alpha, t}(R)=\{u\in \Omega_{1,\alpha,t}:\|u\|_{C^{1,\alpha}}<R \}.
\]
We claim that the corrected error term has the mapping property
\[
\Pi_{\mathfrak{B}}\circ\mathtt{E}_t:\Omega_{1,\alpha,t}(R)\subset \mathcal{B}_{1,\alpha}\longrightarrow \mathcal{B}_{0,\alpha}.
\]
Indeed, this follows from Lemma~\ref{lem:sec10-second-order-cancellation}. So, we consider
\[
T_t=\cL^{-1}\circ\Pi_{\mathfrak{B}}\circ \mathtt{E}_t:\Omega_{1,\alpha,t}(R)\longrightarrow \mathcal{B}_{2,\alpha}.
\]
Since the inclusion $\mathcal{B}_{2,\alpha}\hookrightarrow \mathcal{B}_{1,\alpha}$ is compact, $T_t$ is a compact mapping.  Moreover, the dependence on $t\in[0,1]$ is
continuous in the $C^{1,\alpha}$ topology.  We claim that fixed points of $T_t$ solve system~\eqref{eq: dimReduceDHYM-MOC}.  By definition if $T_t(u)=u$, then
\[
\mathcal{L}(u)-(E_{t}(u) - \lambda\rho)=0
\]
On the other hand, since
\[
I_{E_1} = \mathtt{S}({\rm ad_{u_1}})^{-1}e^{u_1/2}I_{E_1}e^{-u_1/2}
\]
we obtain from~\eqref{eq: protoFixedPointOperator} that fixed points of $T_t$ satisfy the equation
\[
\mathcal{F}_{1,t}(u)= \lambda I_{E_1} \qquad \mathcal{F}_{2,t}= \lambda.
\]
The following lemma shows that $\lambda=0$
\begin{lem}
\label{lem:sec10-scalar-residual}
Suppose that, for some $t\in[0,1]$, a pair of metrics $(h_1,h_2)$ satisfies the equations
\[
\cF_{1,t}=\lambda I_{E_1},\qquad \cF_{2,t}=\lambda
\]
for a real constant $\lambda$.  Then $\lambda=0$ and $(h_1, h_2)$ solve~\eqref{eq: dimReduceDHYM-MOC}.
\end{lem}

\begin{proof}
Applying $T_{\Phi,t}$ to $\cF_{1,t}$ we have
\[
\mathfrak{F}_{1,t} = \lambda (\cos(\htheta) I_{E} + t\sin(\htheta) \Phi\Phi^{\dagger})
\]
where $\mathfrak{F}_{1,t}$ is defined in~\eqref{eq: frakF_1def}.  Similarly we have
\[
\mathfrak{F}_{2,t} = \lambda \left(\cos(\htheta) + (\frac{4\pi}{\sigma}-t|\Phi|^2)\sin(\htheta)\right)
\]
 Since we have  $\int_X\Tr(\mathfrak{F}_{1,t})+\mathfrak{F}_{2,t}=0$,
 we must have $\lambda(\cos(\htheta)(\rank(E_1)+1) + \frac{4\pi}{\sigma}\sin(\htheta))=0$. Since $\sin(\htheta), \cos(\htheta)>0$ we deduce $\lambda=0$.

\end{proof}

We can now prove the main theorem.  We use an index theory argument, based on the Leray-Schauder topological degree, which was kindly suggested to us by S. Brendle who also pointed out to us the reference \cite{ChangGursky} where a similar argument is used.  The reference for the material used in this section is contained in Nirenberg's lecture notes \cite{Nirenberg}

\subsection{Proof of Theorem~\ref{thm: mainTheorem}}

\begin{proof}[Proof of Theorem~\ref{thm: mainTheorem}]
We consider two cases, according to whether ${\rm Re}(Z_{\cX}(\cE))$ is positive or negative.
\\
\bigskip

{\bf Case 1}:  Suppose that ${\rm Re}(Z_{\cX}(\cE))\leq 0$.  Then the theorem is a consequence of Corollary~\ref{cor: classifySolRealNeg}.
\bigskip

{\bf Case 2}: Suppose $\cE=(E_1,E_2,\Phi)$ has ${\rm Re}(Z_{\cX}(\cE)) >0$ and admits a solution of~\eqref{eq: dimReduceDHYM}+~\eqref{eq: dimReduceDHYM-Pos} with $\tan(\htheta)\leq \frac{\sigma}{4\pi}$.  Then $\cE$ is $Z$-stable by Proposition~\ref{prop: stabNecessaryCosPos}.
\bigskip

Conversely, suppose $\cE=(E_1,E_2,\Phi)$ is a triple of vortex type that is stable in the sense of Definition~\ref{def: stabOnXxP}.  By Theorem~\ref{thm: solveInitialEquation} there exists a unique solution of the system~\eqref{eq: dimReduceDHYM-MOC} at  time $t=0$.  If $\tan (\htheta) > \frac{2 \pi}{\sigma}$ then let $\mu(t):[0,1]\rightarrow \mathbb{R}_{>0}$ be a monotone decreasing, piecewise constant function such that
\[
(1.25)\frac{4\pi}{\sigma f(\htheta, t)}\leq \mu(t)\leq (1.5)\frac{4\pi}{\sigma f(\htheta, t)}
\]
where $f$ is defined in Lemma~\ref{lem: conditionalPhiBound}.  If $\tan(\htheta) \leq \frac{2\pi}{\sigma}$ we take $\mu(t)= \frac{3\pi}{\sigma}$ to be constant.  We choose $\mu$ so that it takes on only finitely many values
\[
\mu(t) = \mu_i \quad \text{ for } t \in [t_i,t_{i+1}) \quad  0 \leq i \leq N
\]
where $\mu_0\geq \mu_1 \geq \cdots \geq \mu_{N}= \frac{3\pi}{\sigma}$ and we set $\mu(1)=\mu_{N}$.  For each $i\in\{0,\ldots, N\}$ and $R>0$ define
\[
\widehat{\Omega}_{1,\alpha}(\mu_i,R) = \bigg\{(u_1,u_2)\in \mathcal{B}_{1,\alpha} : \sup_{X}|\Phi|^2_{h_0e^{u}} < \mu_i \quad \text{ and } \quad \|(u_1,u_2)\|_{C^{1,\alpha}(h_{1,0},h_{2,0})} < R\bigg\}
\]
From our choice of $\mu$ we have that
\[
\widehat{\Omega}_{1,\alpha}(\mu_i, R) \subset \Omega_{1,\alpha,t}(R) \quad \text{ for } t \in [t_i,t_{i+1}]
\]
Since the upper bound $\sup_{X}|\Phi|^2_{h_0e^{u}} < \mu_i$ certainly implies $\inf_{X}|\Phi|^2 < \frac{4\pi}{\sigma t}$, we may apply Lemma~\ref{lem: conditionalPhiBound} and Proposition~\ref{prop: stableImpliesEstimatesToAllOrders} to find a uniform $R>0$ such that any solution of the system~\eqref{eq: dimReduceDHYM-MOC} for $t\in [t_i,t_{i+1}]$ (or $t\in [t_{N},1]$) lies in $\widehat{\Omega}_{1,\alpha}(\mu_i, R)$.  For $t\in [t_i,t_{i+1}]$, we consider the compact operator
\[
T_{t}: \widehat{\Omega}_{1,\alpha}(\mu_i, 2R) \rightarrow \mathcal{B}_{1,\alpha}.
\]
By Theorem~\ref{thm: solveInitialEquation}, $T_{0}$ has a unique fixed point and a short calculation shows that the Leray-Schauder degree is $1$. By Lemma~\ref{lem:sec10-scalar-residual}, any fixed point of $T_{t}$ solves ~\eqref{eq: dimReduceDHYM-MOC} and, for any such solution
\[
\sup_{X}|\Phi|^2_{h}\leq \frac{4\pi}{\sigma f(\htheta,t)} < \mu_0
\]
thanks to Lemma~\ref{lem: conditionalPhiBound}.  The homotopy invariance of the Leray-Schauder degree yields that $T_{t}:\widehat{\Omega}_{1,\alpha}(\mu_i, 2R) \rightarrow \mathcal{B}_{1,\alpha}$ has degree $1$ for all $t\in[0,t_1]$.  On the other hand, for any fixed point of $T_{t_1}$ we have
\[
\sup_{X}|\Phi|^2_{h}\leq \frac{4\pi}{\sigma f(\htheta,t_1)} < \mu_1 < \mu_0
\]
and hence, by excision, $T_{t_1}:\widehat{\Omega}_{1,\alpha}(\mu_1, 2R) \rightarrow \mathcal{B}_{1,\alpha} $ has Leray-Schauder degree $1$.  Continuing in this way we obtain that $T_{1}: \widehat{\Omega}_{1,\alpha}(\mu_N, 2R) \rightarrow \mathcal{B}_{1,\alpha} $ has Leray-Schauder degree $1$, and hence a solution of the system~\eqref{eq: dimReduceDHYM} exists.  Finally, by Lemma~\ref{lem: positivityPreservedAlongMOC} this solution automatically satisfies~\eqref{eq: dimReduceDHYM-Pos} and the theorem follows.
\end{proof}

\section{Positivity notions, GIT, and Bridgeland stability} \label{sec: comments}
In this section we comment on some features of the solutions we have produced and formulate some questions for the future.

\subsection{Positivity notions and ellipticity}
We first discuss some analytic features of our solutions.  Recall the following notion introduced by Dervan-McCarthy-Sektnan \cite{DMS}, which we formulate in dimension $2$.

\begin{defn}\label{defn: Zpositive}
Let $\cE \rightarrow \cX$ be a holomorphic vector bundle over a complex surface, and suppose that $\cE$ admits a solution of the dHYM equation with $\htheta \in (0,\pi)$.  We say that a solution of the deformed Hermitian-Yang-Mills equation
\[
{\rm Im}(e^{-i\htheta}(\omega\otimes I_{\cE}-F)^2)=0
\]
is $Z$-positive if, for every point $x\in \cX$ and every $\xi \in \Lambda^{0,1}_{x}\otimes {\rm End}(\cE)_x$ we have
\[
{\rm Tr} \left[{\rm Im}\left(ie^{-i\htheta}(\omega \otimes I_{\cE} -F)\right) \wedge i\xi^{\dagger}\wedge  \xi\right]+{\rm Tr} \left[i\xi^{\dagger}\wedge{\rm Im}\left(ie^{-i\htheta}(\omega \otimes I_{\cE} -F)\right) \wedge  \xi\right]>0
\]
\end{defn}
As explained in \cite{DMS}, this notion is related to the ellipticity of the dHYM system.  In our $SU(2)$ equivariant setting we can use the calculations in Section~\ref{sec: dimReduction} to find
\[
{\rm Im}(ie^{-i\htheta}(\widetilde{\omega}_{\sigma}\otimes I_{\cE} - F))= A\otimes \omega_{X} + B\otimes \omega_{FS} + \mathcal{C}_{mixed}
\]
where
\begin{equation}\label{eq: Amatrix}
A=\begin{pmatrix}\cos(\htheta) I_{E_1} +\sin(\htheta)i\Lambda F_1 & 0\\ 0 & \cos(\htheta) + \sin(\htheta)i\Lambda F_2\end{pmatrix}
\end{equation}
\begin{equation}\label{eq: Bmatrix}
B=\frac{\sigma}{2\pi} \begin{pmatrix} \cos(\htheta)I_{E_1}+\sin(\htheta)\Phi\Phi^{\dagger} & 0 \\ 0 & \cos(\htheta)+(\frac{4\pi}{\sigma}-|\Phi|^2)\sin(\htheta)\end{pmatrix}
\end{equation}
and 
\begin{equation}\label{eq: Cmatrix}
\mathcal{C}_{mixed} = \begin{pmatrix} 0 & i\sin(\htheta)\nabla \Phi \otimes \alpha\\
-i\sin(\htheta)(\nabla \Phi)^{\dagger} \otimes \alpha^{\dagger}\end{pmatrix}
\end{equation}
with $\alpha$ as in Section~\ref{sec: dimReduction}. Fix a point and choose a coordinate $w$ on $X$ and $z$ on $\mathbb{P}^1$ so that $\omega_{X}= idw\wedge d\bar{w}$, $\omega_{FS}= idz\wedge d\bar{z}$ and $\alpha = \sqrt{\frac{\sigma}{2\pi}}d\bar{z}$.  Let $\xi = \chi d\bar{w} + \eta d\bar{z}$ where $\chi, \eta \in {\rm End}(E_1\oplus E_2)_{x}$.  Let $C=\begin{pmatrix} 0 & -i\sqrt{\frac{\sigma}{2\pi}}\sin(\htheta)\nabla \Phi\\ 0 & 0\end{pmatrix}$.  Then $Z$-positivity amounts to the weak positivity of the block matrix,
\[
\mathcal{M} = \begin{pmatrix} A & C \\ C^{\dagger} & B \end{pmatrix}
\]
in the following sense:
\begin{equation}\label{eq: positivityTraceCondition}
{\rm Tr}(\mathcal{M}P)>0 \quad \text{ for all matrices } P = \begin{pmatrix}   \eta^{\dagger}\eta+\eta\eta^{\dagger} & \eta^{\dagger}\chi + \eta \chi^{\dagger}\\
\chi^{\dagger}\eta + \chi\eta^{\dagger} & \chi^{\dagger}\chi+\chi\chi^{\dagger}  \end{pmatrix}.
\end{equation}
This condition is implied by the positive definiteness of $\mathcal{M}$, but is in general significantly weaker.  For example, the matrix $\mathcal{M} = \begin{pmatrix} 1 & i\lambda\\ -i\lambda & 1\end{pmatrix}$ is weakly positive in the above sense for all $\lambda \in \mathbb{R}$.  The remainder if this section is spent examining these positivty notions for solutions.  Similar calculations were carried out by Ballal-Pingali \cite{AshPin} for the case of the vector bundle Monge-Amp\`ere equation.

\begin{lem}\label{lem: positivityProperties}
    Suppose that $\cE= (E_1,E_2, \Phi)$ is a bundle of vortex type with an $SU(2)$-equivariant solution of the dHYM system~\eqref{eq: dimReduceDHYM}+~\eqref{eq: dimReduceDHYM-Pos} with $\inf_{X}|\Phi|^2 < \frac{4\pi}{\sigma}$. Then:
    \begin{itemize}
        \item[(i)] The matrices $A$ and $B$, viewed as endomorphisms of $\cE_{x} \simeq (E_1\oplus E_2)_{x}$ are positive definite.
        \item[(ii)] If $\cos(\htheta)\leq 0$, then $C=0$, and hence $\mathcal{M}$ is positive definite.  In particular, any solution of the deformed Vortex equations is $Z$-positive for $\htheta \in [\frac{\pi}{2}, \pi]$.
    \end{itemize}
\end{lem}
\begin{proof}
    The assumption $\inf_{X}|\Phi|^2 < \frac{4\pi}{\sigma}$ implies $|\Phi| \leq \frac{2\pi}{\sigma}$ by Lemma~\ref{lem: conditionalPhiBound}. Let us first deal with the case that $\cos(\htheta)\leq 0$.  In this case, by Proposition~\ref{prop: charSolCosNeg} we have $E_1 \simeq E_2$ and $\Phi$ is nonvanishing with $|\Phi|^2 \equiv \frac{2\pi}{\sigma}$, and 
    \[
    i\Lambda F_1= i\Lambda F_2= \frac{\sin(\htheta) -\frac{2\pi}{\sigma}\cos(\htheta)}{\cos(\htheta)+\frac{2\pi}{\sigma}\sin(\htheta)}
    \]
    and $\cos(\htheta) +\frac{2\pi}{\sigma}\sin(\htheta)>0$. This yields the result.

    Now assume that $\cos(\htheta)>0$.  That $B$ is positive follows from Lemma~\ref{lem: H2CurvLem} and the $C^0$ bound obtained in Lemma~\ref{lem: conditionalPhiBound}.  Now consider the matrix $A$.  The lower right block $\cos(\htheta)+i\sin(\htheta) i\Lambda F_2$ is positive by Lemma~\ref{lem: H2CurvLem} and Lemma~\ref{lem: conditionalPhiBound}.  To check the positivity of $\cos(\htheta) I_{E_11} + \sin(\htheta)i\Lambda F_1$ requires more work.  Let $\alpha = \tan(\htheta)$, $x=\sqrt{\alpha}\Phi$ and $y= \sqrt{\frac{\alpha}{2}}\nabla \Phi$.  By Lemma~\ref{lem: H1CurveLemSimplfiedPositiveCos} we have
    \begin{equation}\label{eq: F_1curveformulaxybasis}
    (I+ \alpha i\Lambda F_1) = \underbrace{(1+\alpha^2)I - \frac{(1+\alpha^2)}{1+|x|^2}xx^{\dagger}}_{Q_1}+\underbrace{\alpha yy^{\dagger}-\frac{\alpha}{2+|x|^2}\{xx^{\dagger},yy^{\dagger}\} + \frac{\alpha |\langle x,y \rangle|^2}{(2+|x|^2)(1+|x|^2)}xx^{\dagger}}_{Q_{2}}
    \end{equation}
    Now if $v$ is a unit vector in the fiber of $E_1$ we have
    \[
v^{\dagger}Q_{1}v = (1+\alpha^2) -\frac{(1+\alpha^2)}{1+|x|^2}|\langle x,v\rangle|^2 >0
    \]
    Similarly,
    \[
    v^{\dagger}Q_{2}v = \alpha |\langle v,y \rangle |^2 - \frac{2\alpha}{2+|x|^2}{\rm Re}(\langle v,x \rangle\langle x,y\rangle \langle y,v \rangle) + \frac{\alpha|\langle x,y \rangle|^2}{(2+|x|^2)(1+|x|^2)}|\langle x,v \rangle|^2
    \]
    We can complete a square in $v^{\dagger}Q_{2}v$ to get
    \[
    v^{\dagger}Q_{2}v= \alpha \left| \langle y, v\rangle - \frac{\langle y, x\rangle}{2+|x|^2}\langle x, v\rangle \right|^2 + \alpha |\langle x, y\rangle|^2 |\langle x, v\rangle|^2 \left[ \frac{1}{(2+|x|^2)(1+|x|^2)} - \frac{1}{(2+|x|^2)^2} \right]
    \]
    and since $\frac{1}{(2+|x|^2)(1+|x|^2)} - \frac{1}{(2+|x|^2)^2}>0$, the lemma follows.
\end{proof}

Lemma~\ref{lem: positivityProperties} shows that solutions of the deformed Vortex system satisfy the $Z$-positivity property in the ``hypercritical" phase regime when $\hat{\theta} \in [\frac{\pi}{2}, \pi]$.  On the other hand, it seems unlikely that $Z$-positivity holds when $\hat{\theta} \in (0, \frac{\pi}{2})$.  To understand this we need the following 

\begin{lem}
For $\{i,j\} \in \{1,2\}$, let $A_{ij}, B_{ij}, C_{ij}$ be the block matrix decompositions of the matrices $A,B,C$ defined in~\eqref{eq: Amatrix},~\eqref{eq: Bmatrix}, and~\eqref{eq: Cmatrix}. $\mathcal{M}$ satisfies~\eqref{eq: positivityTraceCondition} if and only if the following conditions hold:
\begin{itemize}
\item[$(i)$]
\begin{equation}\label{eq: Zpositive-scalar-equation}
A_{11}+A_{22}I_{E_1} > \frac{C_{12}C_{12}^{\dagger}}{B_{22}}
\end{equation}
and, 
\item[$(ii)$] if we set $\mathcal{T}_{B} = B_{22}I_{E_1} + B_{11}$, then the Hermitian linear map $K: {\rm End}(E_1)_{x} \rightarrow {\rm End}(E_1)_{x}$ given by
\begin{equation}\label{eq: Zpositive-matrix-equation}
K(\eta) = \eta A_{11} + A_{11}\eta - C_{12}C_{12}^{\dagger}\eta \mathcal{T}_{B}^{-1} - \mathcal{T}_{B}^{-1}\eta C_{12}C_{12}^{\dagger}.
\end{equation}
is positive definite in the sense that ${\rm Tr}(K(\eta) \eta^{\dagger})>0$ for all $\eta \ne 0$.
\end{itemize}
\end{lem}
\begin{proof}
    Let 
    \[
    \eta = \begin{pmatrix} \eta_{11} & \eta_{12}\\ \eta_{21} & \eta_{22} \end{pmatrix} \qquad \chi = \begin{pmatrix} \chi_{11} & \chi_{12} \\ \chi_{21} & \chi_{22} \end{pmatrix} 
    \]
    be the block decomposition of $\eta, \chi$.  Let also $A_{ij}, B_{ij}, C_{ij}$ denote the blocks of the matrices $A,B,C$. Then weak positivity is the condition that
    \[
    \begin{aligned}
{\rm Tr}(\mathcal{M} P)&= {\rm Tr}(A_{11}(\eta_{11}^{\dagger}\eta_{11}+\eta_{11}\eta_{11}^{\dagger})) + {\rm Tr}(A_{11}(\eta_{21}^{\dagger}\eta_{21} + \eta_{12}\eta_{12}^{\dagger}))\\
&\quad+ {\rm Tr}(A_{22}(\eta_{22}^{\dagger}\eta_{22}+\eta_{22}\eta_{22}^{\dagger})) + {\rm Tr}(A_{22}(\eta_{12}^{\dagger}\eta_{12} + \eta_{21}\eta_{21}^{\dagger}))\\
&\quad+{\rm Tr}(B_{11}(\chi_{11}^{\dagger}\chi_{11}+\chi_{11}\chi_{11}^{\dagger})) + {\rm Tr}(B_{11}(\chi_{21}^{\dagger}\chi_{21} + \chi_{12}\chi_{12}^{\dagger}))\\
&\quad+ {\rm Tr}(B_{22}(\chi_{22}^{\dagger}\chi_{22}+\chi_{22}\chi_{22}^{\dagger})) + {\rm Tr}(B_{22}(\chi_{12}^{\dagger}\chi_{12} + \chi_{21}\chi_{21}^{\dagger}))\\
&\quad+{\rm Tr}(C_{12}(\chi_{12}^{\dagger}\eta_{11}+\chi_{22}^{\dagger}\eta_{21}+\chi_{21}\eta_{11}^{\dagger} +\chi_{22}\eta_{12}^{\dagger}))\\
&\quad+ {\rm Tr}(C_{12}^{\dagger}(\eta_{11}^{\dagger}\chi_{12} + \eta_{21}^{\dagger}\chi_{22}  + \eta_{11}\chi_{21}^{\dagger}+\eta_{12}\chi_{22}^{\dagger}))\\
&= Q_{(I)}(\eta_{22},\chi_{11})+Q_{(II)}(\chi_{22}, \eta_{21},\eta_{12})+Q_{(III)}(\eta_{11}, \chi_{21},\chi_{12})
\end{aligned}
    \]
where,
\[
\begin{aligned}
Q_{(I)}(\eta_{22},\chi_{11}) &=  {\rm Tr}(A_{22}(\eta_{22}^{\dagger}\eta_{22}+\eta_{22}\eta_{22}^{\dagger}))+{\rm Tr}(B_{11}(\chi_{11}^{\dagger}\chi_{11}+\chi_{11}\chi_{11}^{\dagger})) \\
Q_{(II)}(\chi_{22}, \eta_{21},\eta_{12}) &= {\rm Tr}(B_{22}(\chi_{22}^{\dagger}\chi_{22}+\chi_{22}\chi_{22}^{\dagger}))\\
&\quad + {\rm Tr}(A_{11}(\eta_{21}^{\dagger}\eta_{21} + \eta_{12}\eta_{12}^{\dagger}))
+{\rm Tr}(A_{22}(\eta_{12}^{\dagger}\eta_{12} + \eta_{21}\eta_{21}^{\dagger}))\\
&\quad +{\rm Tr}(C_{12}(\chi_{22}^{\dagger}\eta_{21}+\chi_{22}\eta_{12}^{\dagger})) + {\rm Tr}(C_{12}^{\dagger}(\eta_{21}^{\dagger}\chi_{22}+\eta_{12}\chi_{22}^{\dagger}))\\
Q_{(III)}(\eta_{11}, \chi_{21},\chi_{12})&= {\rm Tr}(A_{11}(\eta_{11}^{\dagger}\eta_{11}+\eta_{11}\eta_{11}^{\dagger}))\\
&\quad+{\rm Tr}(B_{11}(\chi_{21}^{\dagger}\chi_{21} + \chi_{12}\chi_{12}^{\dagger})+{\rm Tr}(B_{22}(\chi_{12}^{\dagger}\chi_{12} + \chi_{21}\chi_{21}^{\dagger}))\\
&\quad +{\rm Tr}(C_{12}(\chi_{12}^{\dagger}\eta_{11}+\chi_{21}\eta_{11}^{\dagger})) + {\rm Tr}(C_{12}^{\dagger}(\eta_{11}^{\dagger}\chi_{12} + \eta_{11}\chi_{21}^{\dagger}))
\end{aligned}
\]
To minimize this expression we can set $\eta_{22}=\chi_{11}=0$, since $Q_{(I)}$ is positive by Lemma~\ref{lem: positivityProperties}.  Since $Q_{(II)}$ and $Q_{(III)}$ are functions of independent variables, we need each term to be positive independently.  Let us focus on $Q_{(II)}$.  We write
\[
Q_{(II)} = {\rm Tr}(B_{22}(\chi_{22}^{\dagger}\chi_{22}+\chi_{22}\chi_{22}^{\dagger})) + Q_{(IIa)}(\chi_{22},\eta_{21}) + Q_{(IIb)}(\chi_{22}, \eta_{12})
\]
where
\[
\begin{aligned}
&Q_{(IIa)}(\chi_{22},\eta_{21})= {\rm Tr}(A_{11}\eta_{21}^{\dagger}\eta_{21})  + {\rm Tr}(A_{22}\eta_{21}\eta_{21}^{\dagger})+ {\rm Tr}(C_{12}\chi_{22}^{\dagger}\eta_{21}) + {\rm Tr}(C_{12}^{\dagger}\eta_{21}^{\dagger}\chi_{22} )\\
&Q_{(IIb)}(\chi_{22},\eta_{12}) = {\rm Tr}(A_{11}\eta_{12}\eta_{12}^{\dagger}) + {\rm Tr}(A_{22}\eta_{12}^{\dagger}\eta_{12})+ {\rm Tr}(C_{12}\chi_{22}\eta_{12}^{\dagger}) + {\rm Tr}(C_{12}^{\dagger}\eta_{12}\chi_{22}^{\dagger})
\end{aligned}
\]
Since $Q_{IIb}(\chi_{22},0)=0$, and $Q_{(IIa)}(\chi_{22},0)=0$, we can again check positivity of $Q_{(II)}$ independently for $Q_{(IIa)}$ and $Q_{(IIb)}$.  Let
\[
\mathcal{T}_A= A_{22}I_{E_1}+A_{11} \qquad \mathcal{T}_{B}= B_{22}I_{E_1}+B_{11}
\]
Using that $A_{22} \in \mathbb{R}_>0$ is a scalar, and completing the square with respect to $\eta_{21}$ yields
\[
\min_{\eta_{21}}Q_{(IIa)} = -{\rm Tr}(C_{12}^{\dagger}\mathcal{T}_{A}^{-1}C_{12})|\chi_{22}|^2
\]
Minimization of $Q_{(IIb)}$ over $\eta_{12}$ yields the same result, and hence $Q_{(II)}$ is positive if and only if
\[
B_{22}-C_{12}^{\dagger}\mathcal{T}_{A}^{-1}C_{12}>0
\]
which is equivalent to
\[
A_{11}+A_{22}I_{E_1} > \frac{C_{12}C_{12}^{\dagger}}{B_{22}}.
\]
We apply similar analysis to $Q_{(III)}$.  Write
\[
\begin{aligned}
    Q_{(IIIa)}(\eta_{11},\chi_{21}) &= {\rm Tr}(B_{11}\chi_{21}^{\dagger}\chi_{21})+ {\rm Tr}(B_{22}\chi_{21}\chi_{21}^{\dagger})+ {\rm Tr}(C_{12}\chi_{21}\eta_{11}^{\dagger}) + {\rm Tr}(C_{12}^{\dagger}\eta_{11}\chi_{21}^{\dagger})\\
    Q_{(IIIb)}(\eta_{11},\chi_{12}) &= {\rm Tr}(B_{11}\chi_{12}\chi_{12}^{\dagger})+{\rm Tr}(B_{22}\chi_{12}^{\dagger}\chi_{12}) + {\rm Tr}(C_{12}\chi_{12}^{\dagger}\eta_{11}) + {\rm Tr}(C_{12}^{\dagger}\eta_{11}^{\dagger}\chi_{12})
\end{aligned}
\]
Minimizing $Q_{(IIIa)}$ over $\chi_{21}$ and $Q_{(IIIb)}$ over $\chi_{12}$ yields that
\[
Q_{(III)}(\eta_{11},\chi_{12},\chi_{21}) \geq {\rm}Tr(K(\eta_{11}) \eta_{11}^{\dagger})
\]
where $K: {\rm End}(E_1)_{x}\rightarrow {\rm End}(E_1)_{x}$ is the linear map
\[
K(\eta) = \eta A_{11} + A_{11}\eta - C_{12}C_{12}^{\dagger}\eta \mathcal{T}_{B}^{-1} - \mathcal{T}_{B}^{-1}\eta C_{12}C_{12}^{\dagger}.
\]
\end{proof}

We can now check the positivity properties using our specific matrices, together with the a priori estimate $|\Phi|^2 \leq \frac{2\pi}{\sigma}$ obtained in Lemma~\ref{lem: conditionalPhiBound}.  This is somewhat computationally intensive, but the complexity is reduced by noting that  if we set $W= {\rm span}\{\Phi, \nabla \Phi\}$, then positivity is automatic on ${\rm Hom}(W^{\perp}, W^{\perp})$.  

\subsubsection{Checking~\eqref{eq: Zpositive-scalar-equation}} Write $W= {\rm span}\{\Phi, \nabla \Phi\}$. The operator $A_{11}+A_{22}I_{E_1}-\frac{C_{12}C_{12}^{\dagger}}{B_{22}}$ preserves $W$ and $W^{\perp}$, and is plainly positive on the latter. So it suffices to check the positivity of the $2\times 2$-Hermitian matrix in ${\rm Hom}(W,W)$.  This can be done symbolically and one finds that~\eqref{eq: Zpositive-scalar-equation} holds uniformly for $\htheta \in (0,\frac{\pi}{2})$, $\sigma >0$ under the condition $|\Phi|^2 \leq \frac{2\pi}{\sigma}$, which holds for any solution thanks to Lemma~\ref{lem: conditionalPhiBound}.

\subsubsection{Checking~\eqref{eq: Zpositive-matrix-equation}} We will   check the positivity of the $4\times 4$-Hermitian matrix acting on ${\rm Hom}(W,W)$.  We consider minimizing the quadratic form in~\eqref{eq: Zpositive-matrix-equation}. We work in the basis $x=\sqrt{\alpha}\Phi$ and $y=\sqrt{\frac{\alpha}{2}} \nabla \Phi$ with $\alpha =\tan(\htheta)$ and consider only ${\rm Span}\{x,y\}$ which we assume is $2$-dimensional. Fix an orthonormal basis $\{e_1, e_2\}$ with $e_1$ parallel to $\frac{x}{|x|}$. We may assume that $y = ae_1+be_2$ for $a,b \in \mathbb{R}_{\geq 0}$.  Let $p=|x|^2$ and $L=2 + \frac{4\pi\alpha}{\sigma}$.  Then
\[
\mathcal{T}_{B} = B_{11} + B_{22}I = \frac{\sigma}{2\pi}\cos(\htheta) \mathcal{D} \qquad \text{ where } \quad\mathcal{D}:= \begin{pmatrix}L & 0 \\ 0 & (L-p) \end{pmatrix}.
\]
Using the formula~\eqref{eq: F_1curveformulaxybasis} and changing basis we get
\[
I+\alpha i\Lambda F_1= \mathcal{A}:=\begin{pmatrix}\frac{1+\alpha^2 +\alpha a^2}{1+p} & \frac{2\alpha ab}{2+p}\\
\frac{2\alpha ab}{2+p} & 1+\alpha^2 +\alpha b^2 \end{pmatrix}.
\]
Canceling factors of $\frac{\sigma}{2\pi}$ we can write
\[
{\rm Tr}(K(\eta)\eta^{\dagger}) = \cos(\htheta){\rm Tr}\left(\mathcal{A}\eta + \eta \mathcal{A} -2 \alpha \mathcal{D}^{-1}\eta yy^{\dagger} - 2\alpha yy^{\dagger}\eta \mathcal{D}^{-1}\right)
\]
Write $\eta= \begin{pmatrix} u & v \\ w & z\end{pmatrix}$ this leads to the quadratic form
\[
\begin{aligned}
\widetilde{Q}(u,v,w,z) &= 2\frac{1+\alpha^2 +\alpha a^2}{1+p}|u|^2 + \left(\frac{1+\alpha^2 +\alpha a^2}{1+p}+1+\alpha^2+\alpha b^2\right)|v|^2+|w|^2\\
&\quad + 2(1+\alpha^2+\alpha b^2)|z|^2 + \frac{4\alpha ab}{2+p}{\rm Re}(\bar{u}v+\bar{u}w + \bar{v}z+\bar{w}z)\\
&-2\alpha \left(\frac{|au+bv|^2+|au+bw|^2}{L} + \frac{|aw+bz|^2 + |av+bz|^2}{L-p}\right)
\end{aligned}
\]
$Z$-positivity of solutions with the bounds $\sup_{X}|\Phi|^2 \leq \frac{2\pi}{\sigma}$ and $\tan(\htheta)\leq\frac{\sigma}{4\pi}$ then amounts to this quadratic form being positive on the domain where $\alpha>0$, $2<L\leq 3$, $0<p<\frac{L-2}{2}$, $a\geq 0$, $b \geq 0$.  One can check this is not the case.  For example, when
\[
\sigma =4\pi,\quad \alpha=1, \quad L=3 \quad p=2/5, \quad a=8, \quad b=\frac{1}{10}, \quad u=w=z=0, \quad v=1
\]
the quadratic form is strictly negative.  By continuity this also holds for nearby values of the parameters.  The conclusion is that $Z$-positivity is not an algebraic consequence of the solvability of the equation, even after adding in the global estimate $\sup_{X}|\Phi|^2 \leq \frac{2\pi}{\sigma}$.  On the other hand, as we have emphasized, the behaviour of a solution to a nonlinear PDE is not purely local, so $Z$-positivity may still hold.  This would seem to depend on a non-trivial quantitative estimate for $|\nabla \Phi|^2$.  The second author has carried out some numerical experiments for solutions of the deformed Vortex equations on the flat torus $\mathbb{C}/\Lambda$ for which $Z$-positivity does seem to hold.

We can formulate a different notion of positivity, which is strictly weaker than the one in Definition~\ref{defn: Zpositive}.  This notion of positivity was also studied by Chen-Ghosh \cite{ChenGhosh}.
\begin{defn}\label{defn: GriffithsZpositive}
Let $\cE \rightarrow \cX$ be a holomorphic vector bundle over a complex surface, and suppose that $\cE$ admits a solution of the dHYM equation with $\htheta \in (0,\pi)$.  We say that a solution of the deformed Hermitian-Yang-Mills equation
\[
{\rm Im}(e^{-i\htheta}(\omega\otimes I_{\cE}-F)^2)=0
\]
is Griffiths-$Z$-positive if, for every point $x\in \cX$ and every pure tensor $\xi \in \Lambda^{0,1}_{x}\otimes {\rm End}(\cE)_x$ we have
\[
{\rm Tr} \left[{\rm Im}\left(ie^{-i\htheta}(\omega \otimes I_{\cE} -F)\right) \wedge i\xi^{\dagger}\wedge  \xi\right]+{\rm Tr} \left[i\xi^{\dagger}\wedge{\rm Im}\left(ie^{-i\htheta}(\omega \otimes I_{\cE} -F)\right) \wedge  \xi\right]>0
\]
\end{defn}

The difference between Definition~\ref{defn: GriffithsZpositive} and Definition~\ref{defn: Zpositive} is akin to the difference between Griffiths positivity and Nakano positivity of a holomorphic vector bundle.  In terms of checking Griffiths-$Z$-positivity, it amounts to checking~\eqref{eq: positivityTraceCondition} in the special case that $\eta, \chi$ are proportional. One can easily check, by specializing the preceding analysis, that this amounts to checking the positivity of the following $3\times 3$  matrix on ${\rm span}\{\Phi, \nabla \Phi\} \oplus E_2$.
\[
\mathcal{M}'_{t}=\begin{pmatrix} A_{11}+tB_{11} & \sqrt{t}C_{12}\\
\sqrt{t}C_{12}^{\dagger} & A_{22}+tB_{22}
\end{pmatrix}
\]
where $t\in [0,\infty]$ (where the value at $t=\infty$ is interpreted in the obvious way).  We can check the positivity of $\mathcal{M}'_t$ directly.  Again it suffices to check the span of $\{\Phi, \nabla \Phi\}$.  Using the same basis $\{e_1,e_2\}$ and coefficients $p, L, \alpha, a,b$ as before we compute
\[
\mathcal{M}_{t'}= \begin{pmatrix} \frac{1 +\alpha^2 + \alpha a^2}{1+p} + \frac{2\alpha}{L-2}(1+p)t & \frac{2\alpha ab}{2+p} & \frac{2\alpha \sqrt{t}}{\sqrt{L-2}}a\\
\frac{2\alpha ab}{2+p} & 1+\alpha^2+\alpha b^2 + \frac{2\alpha}{L-2}t & \frac{2\alpha \sqrt{t}}{\sqrt{L-2}}b\\
\frac{2\alpha \sqrt{t}}{\sqrt{L-2}}a & \frac{2\alpha \sqrt{t}}{\sqrt{L-2}}b& h(t)\end{pmatrix}
\]
where
\[
h(t)= \frac{1+\alpha^2+\alpha(a^2+b^2)}{L-1-p} + \frac{2\alpha}{L-2}(L-1-p)t
\]
Now $L-1-p = \frac{1}{\cos(\htheta)}( \cos(\htheta) + (\frac{4\pi}{\sigma} - |\Phi|^2)\sin(\htheta)) > 1+\frac{2\pi}{\sigma}\alpha$ by Lemma~\ref{lem: conditionalPhiBound}.  Thus $h(t) >0$ and $h(t)\approx ct$ for $c>0$ as $t\rightarrow +\infty$.  To check positivity it suffices to check the positivity of the Schur complement
\[
P_{t} = h(t) \begin{pmatrix} \frac{1 +\alpha^2 + \alpha a^2}{1+p} + \frac{2\alpha}{L-2}(1+p)t & \frac{2\alpha ab}{2+p} \\
\frac{2\alpha ab}{2+p} & 1+\alpha^2+\alpha b^2 + \frac{2\alpha}{L-2}t \end{pmatrix} - \frac{4\alpha^2 t}{L-2}\begin{pmatrix}a^2 & ab \\ ab & b^2\end{pmatrix}
\]
Write
\[
P_{t}= P_0 +t P_1 + t^2P_{2}
\]
$P_0$ is positive by Lemma~\ref{lem: positivityProperties}.  $P_{2}$ is manifestly positive.  We claim that $P_1$ is also positive.  Explicitly
\[
P_1= \frac{2\alpha}{L-2}(L-1-p)P_0 + h(0)\begin{pmatrix}  \frac{2\alpha}{L-2}(1+p) & 0 \\
0 & \frac{2\alpha}{L-2} \end{pmatrix} - \frac{4\alpha^2}{L-2}\begin{pmatrix}a^2 & ab \\ ab & b^2\end{pmatrix}
\]
We now factor out the common factor of $\frac{2\alpha}{L-2}$ and isolate the terms depending on $a,b$.
\[
\frac{L-2}{2\alpha}P_1= (1+\alpha^2)\begin{pmatrix} \frac{(L-1-p)}{1+p} + \frac{1+p}{L-1-p} & 0\\
&0 & L-1-p + \frac{1}{L-1-p} \end{pmatrix} + \alpha T
\]
where 
\[
T= \begin{pmatrix} \frac{a^2(L-1-p)}{1+p} + \frac{(a^2+b^2)(1+p)}{L-1-p}-2a^2 & (L-1-p)\frac{2ab}{2+p} -2ab \\ (L-1-p)\frac{2ab}{2+p} -2ab & \frac{a^2}{L-1-p} + \frac{b^2(L-2-p)^2}{(L-1-p)}. \end{pmatrix}
\]
The first matrix in this expression is clearly positive definite for $|\Phi|^2 \leq \frac{2\pi}{\sigma}$ (which is equivalent to $L-1-p \geq 1+p$).  Since the diagonal entries of $T$ are clearly positive, it suffices to compute the determinant which can be carried out using, for example, Mathematica.  One finds
\[
\det T= a^{4} r_1(L,p) + b^{4}r_{2}(L,p) + a^2b^2r_{3}(L,p)
\]
where $r_1, r_2, r_3$ are non-negative rational functions. Precisely, if we introduce the symbol $\delta = L-2-2p\geq 0$, where non-negativity follows from Lemma~\ref{lem: conditionalPhiBound}, then
\[
\begin{aligned}
r_1&=\frac{\delta^2}{(1+p)(L-1-p)^2}, \quad r_2=\frac{(1+p)(\delta+p)^2}{(L-1-p)^2}\\
r_3&= \frac{p^2\delta^2(p+\delta)^2+(1+p)^2(p^2+8p\delta +8\delta^2)}{(1+p)(2+p)^2(L-1-p)^2}
\end{aligned}
\]
We conclude that $P_1\geq 0$.  Positivity at $t=\infty$ follows from the strict positivity of $P_2$.  In summary we have proved
\begin{prop}
    The solutions to the deformed Hermitian-Yang-Mills equation produced by Theorem~\ref{thm: mainTheorem} are Griffiths-$Z$-positive for any $\htheta \in (0,\frac{\pi}{2})$.  They are $Z$-positive when $\htheta \in [\frac{\pi}{2}, \pi]$.
\end{prop}

\subsection{Infinite dimensional GIT}

It is natural to formulate the following conjecture:

\begin{conj}\label{conj: fullRange}
    Suppose $\cE=(E_1,E_2,\Phi)$ is a holomorphic triple of vortex type.  Suppose that ${\rm Re}(Z_{\cX}(\cE))>0$ and ${\rm Im}(Z_{\cX}(\cE))>0$ and $\tan(\htheta)\in[\frac{\sigma}{4\pi},+\infty)$.  Then $\cE$ admits a solution of the deformed Vortex equations~\eqref{eq: dimReduceDHYM}+~\eqref{eq: dimReduceDHYM-Pos} if and only if $\cE$ is $Z$-stable in the sense of Definition~\ref{def: stabOnXxP}.
\end{conj}

In our arguments, the assumption $\tan(\htheta) \in (0,\frac{\sigma}{4\pi}]$ appears in two places. The first appearance is in the proof of Lemma~\ref{lem: typeBStabNecessary} establishing the necessity of stability for sub-objects of type B (in the sense of Definition~\ref{defn: subObTypeA/B}). The second appearance is in the proof of Proposition~\ref{prop: UYargumentH}, where it is used to obtain a contradiction in a Uhlenbeck-Yau type blow-up argument which produces a sub-object of type A.  In both cases the assumption is used to exploit positive terms in $i\Lambda F_2$ to control unsigned second fundamental form-type terms. 

As discussed at the end of Section~\ref{sec: stability}, it seems unlikely that establishing necessity of stability, that is Conjecture~\ref{conj: stabTypeB}, can be proven by direct methods starting with the equations~\eqref{eq: dimReduceDHYM}+~\eqref{eq: dimReduceDHYM-Pos}.  Instead one may need to develop a variational approach to existence, as in the work of Collins-Yau \cite{CY, CYarX}; see also \cite{DMS} for discussion in the setting of general $Z$-critical connections.  Given a vortex bundle $\cE$ (or more generally any holomorphic vector bundle) consider the space $\mathcal{H}$ of ``almost calibrated metrics"; that is, Hermitian metrics $h$ on $\cE$ such that
\[
{\rm Re}(e^{-i\htheta}(\omega\otimes I_{\cE} - F_{h})^2)>0.
\]
On the space $\mathcal{H}$ we can define a Riemannian structure as follows: given $\dot{h}_1, \dot{h}_2 \in T_{h}\mathcal{H}$ we define
\[
\langle \dot{h}_1, \dot{h}_2\rangle_{h} = \frac{1}{2}\int_{X}{\rm Tr}\left(\left(h^{-1}\dot{h}_1h^{-1}\dot{h}_{2} +h^{-1}\dot{h}_2h^{-1}\dot{h}_{1}\right) {\rm Re}(e^{-i\htheta}(\omega I_{\cE} - F_{h})^2)\right).
\]
This Riemannian metric leads to a notion of geodesics in $\mathcal{H}$.

Consider also the $1$-forms  $\delta\mathcal{J}, \delta\mathcal{C}$ defined on $\mathcal{H}$ (or more generally defined on the space of all Hermitian metrics on $\cE$) given by
\[
\begin{aligned}
\delta \mathcal{J}(h) &= -\int_{X} {\rm Tr}(h^{-1}\dot{h}{\rm Im}(e^{-i\htheta}(\omega\otimes I_{\cE}-F_h)^2)),\\
\delta\mathcal{C}(h) &=  -\int_{X} {\rm Tr}(h^{-1}\dot{h}{\rm Re}(e^{-i\htheta}(\omega\otimes I_{\cE}-F_h)^2))
\end{aligned}
\]
Collins-Yau \cite{CY} introduced these $1$-forms in the case of line bundles, and their higher rank generalizations were introduced by McCarthy \cite{JBM}; both are related by mirror symmetry to functionals discovered by Solomon \cite{Solomon}, following on work of Thomas \cite{Th}.  An argument of McCarthy \cite[Section 4.5.1]{JBM} shows that these formulas integrate to well-defined functionals $J, \mathcal{C}: \mathcal{H}\rightarrow \mathbb{R}$.  The functional $\mathcal{J}$, which has critical points at solutions of the dHYM equation, should be compared with the Donaldson functionals used to study the Hermitian-Yang-Mills equation \cite{Do85, Do87}, the Higgs equations \cite{Simpson} and the Vortex equations \cite{Bradlow} (and also in Section~\ref{sec: initialEquation}). Importantly, we expect that $\mathcal{J}$ is convex along geodesics in $\mathcal{H}$, while $\mathcal{C}$ is affine.  It may even be the case that there is a suitable notion of ``sub-geodesic" along which $\mathcal{J}$ is convex.  We expect that producing geodesics may be highly non-trivial; see \cite{CY} for the case of line bundles on general K\"ahler manifolds in the hypercritical phase regime.  On the other hand, in the case of the deformed Vortex equations these equations should simplify drastically, and in any case producing sub-geodesics may suffice for our geometric purposes.  We note that McCarthy \cite{JBM} proved the convexity of $\mathcal{J}$ along standard $1$-parameter subgroups $h(t) = h_0e^{t\sigma}$ as long as $h(t)$ is $Z$-positive in the sense of Definition~\ref{defn: Zpositive}. 

Given a vector bundle $S\subset E_1$, consider $\cS \subset \cE$ to be the holomorphic sub-bundle generated by the triple $(S,0,0)$. It is not hard to write down families of metrics which degenerate along $E_1/S$ and one can compute explicitly the limit slope of these degenerations to find that if the limit slope is positive then
\[
{\rm Im}\left(\frac{Z_{\cX}(\cS)}{Z_{\cX}(\cE)}\right)<0
\]
which is precisely the desired inequality.  On the other hand, it may be non-trivial to determine if these degenerations correspond to points in the (formal) boundary $\del \mathcal{H}$.  In any event, assuming this can be understood, one could then prove the necessity of stability using (sub)-geodesics to connect a solution of the dHYM equation with the degenerating sequence and arguing as in \cite{CYarX}.

We expect that establishing necessity of stability in the full range of parameters will shed significant light on the existence problem as well.  Another question we have not addressed, which could be approached this way, is the issue of uniqueness.  We certainly expect that our solutions are unique modulo rescaling. Such a result would follow either from the strict convexity of $\mathcal{J}$ restricted to a level set of $\mathcal{C}$ (as in the case of line bundles \cite{CY, CYarX}), or by analyzing geodesics along which $\mathcal{J}$ is constant.

\subsection{Bridgeland stability conditions}

The existence of Bridgeland stability conditions on $D^{b}Coh(\cX)$ with the central charge $Z_{{\cX}}$ was established by Arcara-Bertram \cite{AB}. One can restrict their construction to the derived category of $SU(2)$-equivariant sheaves $D^{b}Coh^{SU(2)}(\cX)$; we will outline this below.  \'Alvarez-C\'onsul-Garc\'ia-Prada \cite{ACGPQuiver} showed very generally how to describe  equivariant coherent sheaves in terms of holomorphic chains on $X$ (or more generally quiver bundles), and hence the construction of Arcara-Bertram \cite{AB} leads to a notion of Bridgeland stability on the derived category of holomorphic chains.  Bridgeland stability conditions for holomorphic triples have been obtained in \cite{BRH, MRRHR, RHthesis}.

Let us briefly recall the construction of Arcara-Bertram \cite{AB} in our context.  We use the conventions and notation of \cite{AB}.  For background on the topic of Bridgeland stability conditions we refer the reader to \cite{MacriSchmidt}. Let us denote by $\mathcal{D}:= {\rm Coh}^{SU(2)}(\cX)$. Consider the torsion pair
\[
\mathcal{F}= \{\cE \in \mathcal{D} : (1)\, \cE \text{ is torsion free, and } (2) \text{ if  } 0\hookrightarrow \cS \hookrightarrow \cE  \text{ then } {\rm Im}(Z_{\cX}(\cS))\leq 0\}
\]
\[
\mathcal{T}=\{\cE \in \mathcal{D} : \text{ either } (1)\,  \cE\twoheadrightarrow \cQ \ne 0  \Rightarrow {\rm Im}(Z_{\cX}(\cQ))>0, \text{ or } (2)\,\,  \mathcal{E} \text{ is torsion }\}
\]
The pair $(\mathcal{T},\mathcal{F})$ form a torsion-pair in the sense that,
\[
{\rm Hom}_{\mathcal{D}}(\mathcal{T}, \mathcal{F})= \emptyset
\]
and, if $\cE \in \mathcal{D}$ then it follows from the uniqueness of Harder-Narasimhan filtrations that  there is an exact sequence
\[
0 \hookrightarrow \mathcal{T}_{\cE} \rightarrow \cE \rightarrow \mathcal{F}_{\cE}\rightarrow 0 
\]
with $\cF_{\cE}\in \cF$ and $\mathcal{T}_{\cE} \in \mathcal{T}$.  Note that the torsion class defined in Definition~\ref{defn: torsionClass} is precisely the same as torsion class $\mathcal{T}$ introduced above.   We define the {\em tilted heart}
\[
\mathcal{A}:= \langle \cF[1],\mathcal{T}\rangle.
\]
This is another abelian category with objects given by two term complexes.
\[
{\rm ob}(\mathcal{A}) = \{ E \in D^b\mathcal{D} : \mathcal{H}^{-1}(E) \in \mathcal{F}, \mathcal{H}^0(E)\in \mathcal{T}, \mathcal{H}^{j}(E)=0 \text{ else} \}
\]
and furthermore, $D^{b}\mathcal{A} \simeq D^{b}\mathcal{D}$. The key point \cite{AB} is that, thanks to the Bogomolov-Gieseker inequality we have $Z_{\cX}: \cA \rightarrow \mathbb{H}$. Using this, Arcara-Bertram show that $\cA$ is the heart of a Bridgeland stability condition.  Recall \cite{Br} that $\cE \in \cA$ is Bridgeland stable if, 
\[
\cE \twoheadrightarrow \cQ \text{ in } \cA \implies \Theta(\cQ) >\Theta(\cE).
\]
Any equivariant coherent sheaf $\cE \in \mathcal{T} \subset \mathcal{D}$ defines an object in $\cA$ and so we can compare  Bridgeland stability with $Z$-stability.  The main difficulty is the following: suppose we have an exact sequence in $\mathcal{D}$
\[
0\rightarrow \cS\rightarrow \cE \rightarrow \cQ \rightarrow 0
\]
with $\cE, \cQ \in \mathcal{T}$ and $Z_{\cX}(\cS) \in \mathbb{H}$. This exact sequence is not an exact sequence in the tilted heart $\cA$ unless $\cS \in \mathcal{T}$. Indeed, in general, we decompose $\cS$ by
\[
0\rightarrow \cS_{\mathcal{T}} \rightarrow \cS \rightarrow \cS_{\cF}\rightarrow 0
\]
and by taking distinguished triangles we get the following sequence in $\cA$:
\[
0\rightarrow \cS_{\mathcal{T}} \rightarrow \cE \rightarrow \cQ \rightarrow \cS_{\cF}[1]\rightarrow 0.
\]
This introduces two complications when comparing the notions of stability:
\begin{itemize}
    \item[(i)] Not all the exact sequences arising in the definition of $Z$-stability are exact in $\cA$.
    \\
    \item[(ii)] There are exact sequences in $\cA$ of the form $0\rightarrow \cS \rightarrow \cE \rightarrow \cQ \rightarrow 0$ where $\cE, \cS \in \mathcal{T}$, but $\cQ$ is a two-term complex, hence $\cQ \notin \mathcal{D}$ and the map $\cS \rightarrow \cE$ is not injective as a map of coherent sheaves.
\end{itemize} 

In our setting, we can prove the following refinement of $Z$-stability, which says that in Definition~\ref{def: stabOnXxP} it suffices to check the smaller class of exact sequences for which $\cS,\cQ$ lie in the torsion class.

\begin{lem}\label{lem: StabreplaceTorsion}
    If $\cE=(E_1,E_2,\Phi)$ is a holomorphic triple of vortex type then $\cE$ is $Z$-stable if and only if:
    \begin{itemize}
        \item[(i)] $\cE \in \mathcal{T}$, and
        \item[(ii)] for every exact sequence $0\rightarrow \cS \rightarrow \cE \rightarrow \cQ\rightarrow 0$ with $\cS, \cQ \in \mathcal{T}$ we have
        \[
        \Theta(\cS) < \Theta(\cE) \quad \text{ or, equivalently } \quad \Theta(\cE) < \Theta(\cQ).
        \]
    \end{itemize}
\end{lem}
\begin{proof}
    We have essentially already proved this result in Section~\ref{sec: stability}.  First assume that ${\rm Re}\left(Z_{\cX}(\cE)\right)>0$.  Since $Z$-stability implies $(i)$ and $(ii)$, it suffices to prove the converse.  We will show that if $\cE$ is not $Z$-stable, then it can be destabilized by some $0\rightarrow \cS\rightarrow \cE$ with $\cS \in \mathcal{T}$.

    Let $0 \rightarrow \cS \rightarrow \cE$ be a destabilizing sub-sheaf in $\mathcal{D}$ with $Z_{\cX}(\cS) \in \mathbb{H}$.  We may assume $\cS$ is a saturated sub-sheaf by Lemma~\ref{lem: Re>0ReduceStabSaturated}.  We decompose
    \[
    0\rightarrow \cS_{\mathcal{T}}\rightarrow \mathcal{S}\rightarrow \mathcal{S}_{\cF}\rightarrow 0
    \]
    Then we get $0\rightarrow \cS_{\mathcal{T}}\rightarrow \cE$.  If $\Theta(S_{\mathcal{T}}) \geq \Theta(\cE)$, then we are done, so we assume this is not the case.  Then
    \[
Z_{\cX}(\cS)= Z_{\cX}(\cS_{\mathcal{T}}) + Z_{\cX}(\cS_{\cF})
    \]
    By the definition of $\cF$ we have $\cS_{\cF}$ is torsion-free and by the Bogomolov-Gieseker inequality, $-Z_{\cX}(\cS_{\cF}) \in \mathbb{H}$, \cite{AB}.  Since $\cS$ is destabilizing we must have
    \[
    {\rm Im}\left(\frac{Z_{\cX}(\cS_{\cF})}{Z_{\cX}(\cE)}\right) >0
    \]
    Since ${\rm Re}(Z_{\cX}(\cE))>0$, we must have that ${\rm Re}(Z_{\cX}(\cS_{\cF})) < 0$.  From~\eqref{eq: subTripleCentralChargeTypeB} this implies $\cS_{\cF}= (S_F, E_2,\Phi)$ is a holomorphic triple of type A.  On the other hand, in the proof of Lemma~\ref{lem: phaseToRatioStab} we showed that if there is a triple $\cS_{\cF}= (S_F, E_2,\Phi)$ of type $A$, with ${\rm Re}(Z_{\cX}(\cS_{\cF})) < 0$ and ${\rm Im}\left( Z_{\cX}(\cS_{\cF})\right) \leq 0$, then $\cE$ will be $Z$-destabilized by the triple $\cS_{\Phi} = (L,E_2,\Phi)$ where $L={\rm sat}(\Phi(E_2))\subset E_1$. Indeed, we showed in Lemma~\ref{lem: phaseToRatioStab} that in this case $\Theta(\cS_{\Phi}) \in [\frac{\pi}{2}, \pi]$. It now suffices to observe that the triple $\cS_{\Phi}\in \mathcal{T}$.  By additivity of the central charge, this is equivalent to proving ${\rm Im}(Z_{\cX}(\cS_{\Phi}') < {\rm Im}(Z_{\cX}(\cS_{\Phi}) $ for every proper subtriple $\cS_{\Phi}'\subset \cS_{\Phi}$.  The only subtriples $\cS_{\Phi}' \subset \cS_{\Phi}$ are of the form
    \[
    (L(-D_1), E_2(-D_2), \Phi') \quad \text{ or } (L(-D_1), 0,0)
    \]
    where $D_1, D_2$ are effective divisors. In either case we have the desired inequality ${\rm Im}(Z_{\cX}(\cS_{\Phi}')) \leq {\rm Im}(Z_{\cX}(\cS_{\Phi}))$ with equality only if $\cS_{\Phi}' \simeq \cS_{\Phi}$. This concludes the case ${\rm Re}(Z_{\cX}(\cE))>0$.  
\smallskip

    We now assume ${\rm Re}(Z_{\cX}(\cE))\leq 0$.  Assume that $\cE$ satisfies $(i)$ and $(ii)$. The proof proceeds by analyzing three surjections. This should be compared with the proof of Lemma~\ref{lem: stabForTorsion}
    \begin{itemize}
        \item[(a)] To conclude $\rank(E_1)=1$ we consider $\cE/\cS_{\Phi}$.
        \item[(b)] To conclude that $\Phi$ is non-vanishing we consider $\cE \twoheadrightarrow p_1^*E_1\otimes \mathcal{O}_{F}$ for $F\simeq \mathbb{P}^1$ a fiber of $\cX\rightarrow X$.
        \item[(c)] To show that $\cos(\Theta(\cE)) + \frac{2\pi}{\sigma}\sin(\htheta)>0$ we consider the surjection $\cE\rightarrow \iota_*\iota^*\cE$ where $\iota: F \rightarrow \cX$ is the inclusion of a $\mathbb{P}^1$ fiber.
        \end{itemize}
        With these results we can appeal to Lemma~\ref{lem: stabForVeryRestr} to conclude $Z$-stability.  The main thing is to prove that, for each of these maps ${\rm Ker}(\cE\twoheadrightarrow \cQ) \in \mathcal{T}$. 
    
    We first consider $(a)$. Since ${\rm Re}(Z_{\cX}(\cE))\leq 0$ we have $\deg(E_2) \geq \frac{\sigma V_{X}}{8\pi^2}(\rank(E_1)+1)>0$.  As we argued above, this implies $\cS_{\Phi} \in \mathcal{T}$.  It is easy to show, as in Lemma~\ref{lem: PhiCannotVanishIDentically}, that $\Phi$ cannot vanish identically. Indeed, if $\Phi \equiv 0$, then $(0,E_2,0) \in \mathcal{T}$ is a subtriple of $\cE$ and the quotient triple $(E_1,0,0)$ has central charge lying in the first quadrant, contradicting $(ii)$. 
    
    We claim that $\rank(E_1)=1$.  Indeed, if $\rank(E_1)>1$, then consider the quotient $\cE/\cS_{\Phi} = (E_1/L,0,0)$.  Now
    \[
    {\rm Re}(Z_{\cX}(\cE/\cS_{\Phi}))\propto \rank(E_1/L) >0,
    \]
    but this yields $\Theta(\cE/\cS_{\Phi}) \in (0,\frac{\pi}{2})$, contradicting $(ii)$.  Therefore $\rank(E_1)=1$. Since $\Phi \ne 0$ we conclude $\deg(E_1)\geq \deg(E_2)>0$.

    Now we consider $(b)$ and conclude that $\{\Phi=0\} =\emptyset$. Suppose there is a point $x\in X$ such that $\Phi(x)=0$. $\Phi$ factors through the subtriple $\cS_{x}:= (E_1(-x), E_2, \Phi')$.  It is straightforward to show that $\cS_{x}\in \mathcal{T}$. Then
    \[
    \Theta(\cE/\cS_{x}) =\frac{\pi}{2}
    \]
    contradicting the stability of $\cE$. So $\cE= (E_2, E_2, 1)$ with $\deg(E_2)>0$.
    
    Finally, we consider $(c)$. Consider the subtriple $\cS_{(x,x)}:= (E_2(-x), E_2(-x), 1)$.  Since $\deg(E_2(-x))=\deg(E_2)-1 \geq 0$ it is easy to show that $\cS_{(x,x)}\in \mathcal{T}$. Now by considering the quotient $\cE/\cS_{(x,x)}$ we conclude that  $\cos(\htheta)+\frac{2\pi}{\sigma}\sin(\htheta)>0$. By appealing to Lemma~\ref{lem: stabForVeryRestr} we conclude that $\cE$ is $Z$-stable.
\end{proof}

From Lemma~\ref{lem: StabreplaceTorsion} and the preceding discussion we conclude that Bridgeland stability implies $Z$-stability.
\begin{cor}
    Let $\cE=(E_1,E_2,\Phi)$ be a holomorphic triple of vortex type with $\cE \in \mathcal{T}$. If $\cE$ is Bridgeland stable in the tilted heart $\cA$, then $\cE$ is $Z$-stable.
\end{cor}
\begin{proof}
    If $\cE \in \mathcal{T}$ is Bridgeland stable in the tilted heart then, for any sub-sheaf $0\rightarrow \cS \rightarrow \cE$ with $\cS \in \mathcal{T}$ we have $\Theta(\cS) < \Theta(\cE)$.  By Lemma~\ref{lem: StabreplaceTorsion} this implies $\cE$ is $Z$-stable.
\end{proof}

This result should be compared with the results of \cite{CShi, CLSY} which show that, for $Bl_{p}\mathbb{P}^2$ and Weierstrass elliptic $K3$ surfaces, Bridgeland stability does not imply solvability of the dHYM equation.

It is natural to conjecture that the notion of $Z$-stability for holomorphic triples $\cE= (E_1,E_2,\Phi)$ defined in Definition~\ref{def: stabOnXxP} is equivalent to Bridgeland stability in the category $D^{b}{\rm Coh}^{SU(2)}(\cX)$.  In order to prove that $Z$-stability implies Bridgeland stability one needs to rule out the possibility that $Z$-stable triples of vortex type can be destabilized by quotient objects consisting of two term complexes.

\appendix

\section{Uhlenbeck-Yau Formulas}
We give a proof of the Uhlenbeck-Yau formulas~\eqref{eq: UYFormula1} and ~\eqref{eq: UYFormula2}.

\begin{lem}\label{lem: UYFormulaImprovement}
Let $h$ be a Hermitian metric on $E_1$ and assume that $u$ is a positive definite Hermitian endomorphism of $E_1$.  Then, for any $\sigma \in (0, 1]$, the following formulas hold:
\begin{equation}\label{eq: UYFormula1Lem}
     \langle u^{-1}\nabla u, \nabla u^{\sigma}\rangle_{h} \geq |u^{-\sigma/2}\nabla u^{\sigma}|^2_{h}
    \end{equation}
    \begin{equation}\label{eq: UYFormula2Lem}
   \langle (\nabla u)u^{-1},\nabla u^{\sigma}\rangle_{h} \geq |(\nabla u^{\sigma})u^{-\sigma/2}|_{h}^2
    \end{equation}
    \end{lem}

    \begin{proof}
        The proof of both inequalities is identical, and follows verbatim from the calculation of Uhlenbeck-Yau, see \cite[Equation (4.6)]{UY}. We will give only the proof of ~\eqref{eq: UYFormula2Lem}.  Let $e_{\alpha}$ be an $h$-orthonormal frame of eigenvectors for $u$.  Write
        \[
        \begin{aligned}
        \nabla e_{\alpha} &= A_{z\,\alpha}^{\beta}e_{\beta}\\
        \dbar e_{\alpha} &= A_{\bar{z}\, \alpha}^{\beta}e_{\beta}
        \end{aligned}
        \]
        In this frame we have
        \[
        u = e^{\lambda_{\alpha}}e_{\alpha}\otimes e_{\alpha}^{*}
        \]
        Then we compute
        \begin{equation}\label{eq: covarDerivU}
        \nabla  u = e^{\lambda_{\alpha}}\del \lambda_{\alpha} e_{\alpha}\otimes e_{\alpha}^* + (e^{\lambda_{\alpha}}-e^{\lambda_{\gamma}})A_{z\, \alpha}^{\gamma}e_{\gamma}\otimes e_{\alpha}^*. 
        \end{equation}
        Multiplying on the right by $u^{-1}$ yields
        \[
        \begin{aligned}
         (\nabla u)u^{-1} &= \del \lambda_{\alpha}e_{\alpha}\otimes e_{\alpha}^*  + (1-e^{\lambda_{\gamma} - \lambda_{\alpha}})A_{z\,\alpha}^{\gamma}e_{\gamma}\otimes e_{\alpha}^*,\\
         u^{-1}(\nabla u) &=\del \lambda_{\alpha}e_{\alpha}\otimes e_{\alpha}^*  + (e^{\lambda_{\alpha}-\lambda_{\gamma}}-1)A_{z\,\alpha}^{\gamma}e_{\gamma}\otimes e_{\alpha}^* 
         \end{aligned}
        \]
        The formula for $\nabla u^{\sigma}$ can be obtained from~\eqref{eq: covarDerivU} by replacing $\lambda_{\alpha}$ with $\sigma\lambda_{\alpha}$, hence
        \begin{equation}\label{eq: coverDerivUSigma}
        \nabla u^{\sigma} =  \sigma e^{\sigma\lambda_{\alpha}}\del \lambda_{\alpha} e_{\alpha}\otimes e_{\alpha}^* + (e^{\sigma \lambda_{\alpha}}-e^{\sigma\lambda_{\gamma}})A_{z\, \alpha}^{\gamma}e_{\gamma}\otimes e_{\alpha}^*.
        \end{equation}
        Finally, this yields
        \[
        \begin{aligned}
       (\nabla u^{\sigma}) u^{-\sigma/2} &= \sigma e^{\frac{\sigma}{2}\lambda_{\alpha}}\del \lambda_{\alpha} e_{\alpha}\otimes e_{\alpha}^* + e^{-\frac{\sigma}{2}\lambda_{\alpha}}(e^{\sigma \lambda_{\alpha}}-e^{\sigma\lambda_{\gamma}})A_{z\, \alpha}^{\gamma}e_{\gamma}\otimes e_{\alpha}^*\\
       u^{-\sigma/2}(\nabla u^{\sigma}) &= \sigma e^{\frac{\sigma}{2}\lambda_{\alpha}}\del \lambda_{\alpha} e_{\alpha}\otimes e_{\alpha}^*+ e^{-\frac{\sigma}{2}\lambda_{\gamma}}(e^{\sigma \lambda_{\alpha}}-e^{\sigma\lambda_{\gamma}})A_{z\, \alpha}^{\gamma}e_{\gamma}\otimes e_{\alpha}^*
       \end{aligned}
        \]
        We now compute
        \begin{equation}\label{eq: covarDerivLocalFormula}
        \begin{aligned}
        \langle (\nabla u)u^{-1} ,\nabla u^{\sigma}\rangle_{h} &= \sigma \sum_{\alpha} e^{\sigma \lambda_{\alpha}}|\del\lambda_{\alpha}|^2 + \sum_{\alpha \ne \gamma}(e^{\sigma\lambda_{\alpha}}-e^{\sigma \lambda_{\gamma}})(1-e^{\lambda_{\gamma}-\lambda_{\alpha}})|A_{z\, \alpha}^{\gamma}|^2\\
        \langle u^{-1}(\nabla u) ,\nabla u^{\sigma}\rangle_{h}&= \sigma \sum_{\alpha} e^{\sigma \lambda_{\alpha}}|\del\lambda_{\alpha}|^2 + \sum_{\alpha \ne \gamma}(e^{\sigma\lambda_{\alpha}}-e^{\sigma \lambda_{\gamma}})(e^{\lambda_{\alpha}-\lambda_{\gamma}}-1)|A_{z\, \alpha}^{\gamma}|^2
        \end{aligned}
        \end{equation}
        and
        \[
        \begin{aligned}
        |(\nabla u^{\sigma})u^{-\sigma/2}|^2_{h} &= \sigma^2 e^{\sigma\lambda_{\alpha}}|\del\lambda_{\alpha}|^2 +\sum_{\alpha\ne \gamma}e^{-\sigma \lambda_{\alpha}}(e^{\sigma\lambda_{\alpha}}-e^{\sigma\lambda_{\gamma}})^2|A_{z\, \alpha}^{\gamma}|^2\\
        |u^{-\sigma/2}(\nabla u^{\sigma})|^2_{h} &= \sigma^2 e^{\sigma\lambda_{\alpha}}|\del\lambda_{\alpha}|^2 +\sum_{\alpha\ne \gamma}e^{-\sigma \lambda_{\gamma}}(e^{\sigma\lambda_{\alpha}}-e^{\sigma\lambda_{\gamma}})^2|A_{z\, \alpha}^{\gamma}|^2
        \end{aligned}
        \]
        To complete the proof one observes that
        \[
        \begin{aligned}
        e^{-\sigma \lambda_{\alpha}}(e^{\sigma\lambda_{\alpha}}-e^{\sigma\lambda_{\gamma}})^2 &\leq (e^{\sigma\lambda_{\alpha}}-e^{\sigma \lambda_{\gamma}})(1-e^{\lambda_{\gamma}-\lambda_{\alpha}})\\
        e^{-\sigma \lambda_{\gamma}}(e^{\sigma\lambda_{\alpha}}-e^{\sigma\lambda_{\gamma}})^2 &\leq (e^{\sigma\lambda_{\alpha}}-e^{\sigma \lambda_{\gamma}})(e^{\lambda_{\alpha}-\lambda_{\gamma}}-1)
        \end{aligned}
        \]
    \end{proof}

    We shall also need a related formula for the ``Uhlenbeck-Yau excess".  Before stating our result, let us fix notation.  Suppose that $h,h_0$ are metrics on a holomorphic vector bundle $E$.  We denote by $\nabla$ the $(1,0)$ part of the Chern connection of $h$ and $\hat{\nabla}$ the $(1,0)$ part of the Chern connection of $h_0$. Let $u$ be defined by $u=h_0^{-1}h$.  Then $u$ is Hermitian with respect to both $h$ and $h_0$.  Choose an $h$-orthonormal frame $\{e_{\alpha}\}$ so that
    \[
    u= e^{\lambda_{\alpha}} e_{\alpha}\otimes e_{\alpha}^*, \qquad \nabla e_{\alpha}= A_{z}^{\gamma}\,_{\alpha}e_{\gamma}.
    \]
    Then, we have the following lemma
    \begin{lem}\label{lem: UYExcessFormula}
    In the above setting, for any $\mathtt{s} \in (0,1)$ we have
        \[
        \langle u^{-1}\hat{\nabla}u, \hat{\nabla}u^{\mathtt{s}} \rangle_{h_0} - |u^{-\mathtt{s}/2}\hat{\nabla}u^{\mathtt{s}}|^2_{h_0} \geq \sum_{\alpha \ne \gamma}e^{\mathtt{s}\lambda_{\alpha}}(1-r_{\gamma\alpha}^{\mathtt{s}})(1-r_{\gamma\alpha}^{(1-\mathtt{s})})|A_{z}^{\gamma}\,_{\alpha}|^2
        \]
        where $ r_{\gamma \alpha}= e^{\lambda_{\gamma}-\lambda_{\alpha}}$.
    \end{lem}
    \begin{proof}
    The first step in the proof is to write the quantities on the right hand side of the inequality in terms of $\nabla$. Indeed, by the formula for the change of the Chern connection, together with the fact that $u$ and $u^{\mathtt{s}}$ commute we have
    \[
    \nabla u^{\mathtt{s}}= u^{-1}(\hat{\nabla} u^{\mathtt{s}})u
    \]
    for any $\mathtt{s}$, and so
        \[
\begin{aligned}
\langle u^{-1}(\hat{\nabla} u), \hat{\nabla}u^{\mathtt{s}}\rangle_{h_0} &={\rm Tr}\left( u^{-1}(\hat{\nabla} u)(\dbar u^{\mathtt{s}}) \right) \\
&\quad = {\rm Tr}\left( (\nabla u)u^{-1}(\dbar u^{\mathtt{s}}) \right)\\
&\quad = \langle (\nabla u)u^{-1}, \nabla^{\ell} u^{\mathtt{s}}\rangle_{h}.
\end{aligned}
\]
For the second term on the right hand side we compute
\[
\begin{aligned}
|u^{-\mathtt{s}/2}\hat{\nabla}u^{s}|_{h_0}^2 &= {\rm Tr}\left(u^{-\mathtt{s}/2}\hat{\nabla}u^{s}(\dbar u^{s})u^{-\mathtt{s}/2}\right)\\
&={\rm Tr}\left(u^{-\mathtt{s}/2}u(\nabla u^{s})u^{-1}(\dbar u^{s})u^{-\mathtt{s}/2}\right)\\
&={\rm Tr}\left(u^{(1-\mathtt{s})/2}(\nabla^{\ell}u^{s})u^{-1/2}u^{-1/2}(\dbar u^{s})u^{(1-\mathtt{s})/2}\right)\\
&= |u^{(1-\mathtt{s})/2}(\nabla^{\ell}u^{s})u^{-1/2}|^2_{h}
\end{aligned}
\]
Now we write each of these quantities in an $h$ orthonormal frame using the formulas from Lemma~\ref{lem: UYFormulaImprovement}.
\[
\begin{aligned}
    (\nabla u)u^{-1} &= \del \lambda_{\alpha} e_{\alpha} \otimes e_{\alpha}^{*} + (1-e^{(\lambda_{\gamma} - \lambda_{\alpha})})A_{z}^{\gamma}\,_{\alpha}e_{\gamma}\otimes e_{\alpha}^{*}\\
    \nabla u^{\mathtt{s}} &= \mathtt{s}e^{\mathtt{s}\lambda_{\alpha}}\del\lambda_{\alpha} e_{\alpha}\otimes e_{\alpha}^* + (e^{\mathtt{s} \lambda_{\alpha}}-e^{\mathtt{s}\lambda_{\gamma}})A_{z}^{\gamma}\,_{\alpha} e_{\gamma}\otimes e_{\alpha}^{*}.
\end{aligned}
\]
Combining these formulas gives
\[
\langle (\nabla u)u^{-1}, \nabla u^{\mathtt{s}}\rangle_{h}=\mathtt{s}{\sum_{\alpha}}e^{\mathtt{s}\lambda_{\alpha}}|\del \lambda_{\alpha}|^2 +\sum_{\alpha \ne \gamma}(1-e^{(\lambda_{\gamma}-\lambda_{\alpha})})(e^{\mathtt{s}\lambda_{\alpha}}-e^{\mathtt{s}\lambda_{\gamma}})|A_{z}^{\gamma}\,_{\alpha}|^2.
\]
For the second term we compute
\[
u^{(1-\mathtt{s})/2}(\nabla u^{s}) = \mathtt{s}e^{\frac{(1+\mathtt{s})}{2}\lambda_{\alpha}}\del \lambda_{\alpha}e_{\alpha}\otimes e_{\alpha}^* + e^{\frac{(1-\mathtt{s})}{2}\lambda_{\gamma}}(e^{\mathtt{s} \lambda_{\alpha}}-e^{\mathtt{s}\lambda_{\gamma}})A_{z}^{\gamma}\,_{\alpha} e_{\gamma}\otimes e_{\alpha}^{*}
\]
and so
\[
u^{(1-\mathtt{s})/2}(\nabla u^{s})u^{-1/2} = \mathtt{s}e^{\frac{\mathtt{s}}{2}\lambda_{\alpha}} \del \lambda_{\alpha} e_{\alpha}\otimes e_{\alpha}^* + e^{\frac{(1-\mathtt{s})}{2}\lambda_{\gamma}}(e^{\mathtt{s} \lambda_{\alpha}}-e^{\mathtt{s}\lambda_{\gamma}})A_{z}^{\gamma}\,_{\alpha}e^{-\frac{\lambda_{\alpha}}{2}} e_{\gamma}\otimes e_{\alpha}^{*}.
\]
Thus we have
\[
\begin{aligned}
|u^{(1-\mathtt{s})/2}(\nabla u^{s})u^{-1/2}|^2_{h}&= \mathtt{s}^2\sum_{\alpha}e^{\mathtt{s}\lambda_{\alpha}}|\del \lambda_{\alpha}|^2 + \sum_{\alpha \ne \gamma}e^{(1-\mathtt{s})\lambda_{\gamma}}(e^{\mathtt{s} \lambda_{\alpha}}-e^{\mathtt{s}\lambda_{\gamma}})^2e^{-\lambda_{\alpha}^{\gamma}}|A_{z}^{\gamma}\,_{\alpha}|^2\\
&= \mathtt{s}^2\sum_{\alpha}e^{\mathtt{s}\lambda_{\alpha}}|\del \lambda_{\alpha}|^2 + \sum_{\alpha \ne \gamma}e^{(1-\mathtt{s})\lambda_{\gamma}}(e^{\mathtt{s} \lambda_{\alpha}}-e^{\mathtt{s}\lambda_{\gamma}})^2e^{-\lambda_{\alpha}}|A_{z}^{\gamma}\,_{\alpha}|^2\\
\end{aligned}
\]
We now compute the difference, using the notation $r_{\gamma\alpha} = e^{\lambda_{\gamma}-\lambda_{\alpha}}$.  We get
\[
\begin{aligned}
    &\langle (\nabla u)u^{-1}, \nabla u^{\mathtt{s}}\rangle_{h}-|u^{(1-\mathtt{s})/2}(\nabla u^{s})u^{-1/2}|^2_{h} \\
    &= \mathtt{s}(1-\mathtt{s})\sum_{\alpha}e^{\mathtt{s}\lambda_{\alpha}}|\del \lambda_{\alpha}|^2 +\sum_{\alpha\ne \gamma}e^{\mathtt{s}\lambda_{\alpha}}(1-r_{\gamma\alpha})(1-r^{s}_{\gamma\alpha})|A_{z}^{\gamma}\,_{\alpha}|^2\\
    &\quad -  \sum_{\alpha \ne \gamma}e^{\mathtt{s}\lambda_{\alpha}}r_{\gamma\alpha}^{(1-\mathtt{s})}(1-r_{\gamma\alpha}^s)^2|A_{z}^{\gamma}\,_{\alpha}|^2.
\end{aligned}
\]
Now we have
\[
(1-r_{\gamma\alpha})(1-r^{s}_{\gamma\alpha})-r_{\gamma\alpha}^{(1-\mathtt{s})}(1-r_{\gamma\alpha}^{\mathtt{s}})^2=(1-r_{\gamma\alpha}^{\mathtt{s}})(1-r_{\gamma \alpha}^{(1-\mathtt{s})})
\]
which gives the desired result.
    \end{proof}

\begin{rk}\label{rk: UYexcessHomog}
    Note that the estimate in Lemma~\ref{lem: UYExcessFormula} continues to hold if we replace $u$ by $\tilde{u}= Cu$ for a constant $C>0$.  This is clear since both sides of the inequality are homogeneous of degree $\mathtt{s}$ in $u$. 
\end{rk}

\end{document}